\documentclass[10pt,a4paper]{article}
\usepackage[english]{babel}
\usepackage[T1]{fontenc}
\usepackage{fullpage}
\usepackage{cite}

\usepackage{xargs}
\usepackage{tikz, etoolbox}
\usetikzlibrary{matrix, cd, positioning}

\tikzcdset{%
  every label/.style={
    /tikz/auto,
    /tikz/inner sep=+0.5ex}}

\usepackage{amsmath}
\usepackage{amssymb}
\usepackage{amsfonts}
\usepackage{amsthm}
\usepackage{mathabx}
\usepackage{catex}

\usepackage{ifsym}
\usepackage{dsfont}
\usepackage{bm}

\makeatletter
\newcommand*{\pmzerodot}{%
  \nfss@text{%
    \sbox0{$\vcenter{}$}
    \sbox2{0}%
    \sbox4{0\/}%
    \ooalign{%
      0\cr
      \hidewidth
      \kern\dimexpr\wd4-\wd2\relax 
      \raise\dimexpr(\ht2-\dp2)/2-\ht0\relax\hbox{%
        \if b\expandafter\@car\f@series\@nil\relax
          \mathversion{bold}%
        \fi
        $\cdot\m@th$%
      }%
      \hidewidth
      \cr
      \vphantom{0}
    }%
  }%
}
\newcommand*{\pmzeroslash}{%
  \nfss@text{%
    \sbox0{0}%
    \sbox2{/}%
    \sbox4{%
      \raise\dimexpr((\ht0-\dp0)-(\ht2-\dp2))/2\relax\copy2 %
    }%
    \ooalign{%
      \hfill\copy4 \hfill\cr
      \hfill0\hfill\cr
    }%
    \vphantom{0\copy4 }
  }%
}
\makeatother

\usepackage{amstext}

\usepackage{multicol}
\usepackage{multirow}

\usepackage[filecolor=blue]{hyperref}
\usepackage{verbatim}

\DeclareMathOperator{\im}{\mathbf{Im}}
\newcommand{\C}{\mathcal C}
\newcommand{\CC}{\mathbf C}
\newcommand{\D}{\mathcal D}
\newcommand{\DD}{\mathbf D}
\newcommand{\Lax}{\mathbf{Lax}}
\newcommand{\oplax}{\mathbf{OpLax}}
\newcommand{\poplax}{\mathbf{PsOpLax}}
\newcommand{\plax}{\mathbf{PsLax}}
\newcommand{\Set}{\mathbf{Set}}
\newcommand{\G}{\mathbf{G}}

\newcommand{\CCat}[1]{#1 \operatorname{- \, \mathbf{CubCat}}}
\newcommand{\Cat}[1]{#1 \operatorname{- \, \mathbf{Cat}}}
\newcommand{\cube}[1]{#1 \operatorname{- \, \blacksquare}}

\newcommand{\Rect}[1]{#1 \operatorname{-\,\mathbf{Rect}}}

\DeclareMathOperator{\s}{s}
\renewcommand{\t}{\operatorname{t}}
\renewcommand{\d}{\operatorname{d}}
\renewcommand{\l}{\ell}
\newcommand{\tr}{\operatorname{tr}}
\newcommand{\op}{\operatorname{op}}
\newcommand{\e}{\operatorname{e}}
\newcommand{\id}{\operatorname{id}}
\newcommand{\ADC}{\mathbf{ADC}}

\tikzset{>=Implies}
\newlength{\myline}
\newcommandx*{\doublearrow}[2]{
  \draw[line width=rule_thickness,double equal sign distance,#1] #2;
}
\newcommandx*{\triplearrow}[4][1=0, 2=1]{
  \draw[line width=\myline,double distance=5\myline,#3] #4;
  \draw[line width=\myline,shorten <=#1\myline,shorten >=#2\myline,#3] #4;
}
\newcommandx*{\quadarrow}[4][1=0, 2=2.5]{
  \draw[line width=\myline,double distance=8\myline,#3] #4;
  \draw[line width=\myline,double distance=2\myline,shorten <=#1\myline,shorten >=#2\myline,#3] #4;
}

\newcommand{\directions}[2]{
\begin{tikzcd}[cramped, sep = small, ampersand replacement=\&]%
\ar[d, to path = {-- (\tikztostart.center) -- (\tikztotarget)}] \ar[r] \& \scriptstyle{#1} \\ %
\scriptstyle{#2} \& \\ %
\end{tikzcd}}

\def\bbox{\path[use as bounding box] (0,0) rectangle(2ex,2ex);}

\newcommand{\nocorner}{%
\begin{tikzpicture}%
\bbox
\end{tikzpicture}%
}

\newcommand{\cornerur}{%
\begin{tikzpicture}%
\bbox
\draw[line width=.4mm] (1ex,2ex) -- (1ex,1ex) -- (2ex,1ex);%
\end{tikzpicture}%
}

\newcommand{\cornerul}{%
\begin{tikzpicture}%
\bbox
\draw[line width=.4mm] (1ex,2ex) -- (1ex,1ex) -- (0,1ex);%
\end{tikzpicture}%
}

\newcommand{\cornerdr}{%
\begin{tikzpicture}%
\bbox
\draw[line width=.4mm] (2ex,1ex) -- (1ex,1ex) -- (1ex,0);%
\end{tikzpicture}%
}

\newcommand{\cornerdl}{%
\begin{tikzpicture}%
\bbox
\draw[line width=.4mm] (0,1ex) -- (1ex,1ex) -- (1ex,0);%
\end{tikzpicture}%
}

\newcommand{\horiz}{%
\begin{tikzpicture}%
\bbox
\draw[line width=.4mm] (0,1ex) -- (2ex,1ex);%
\end{tikzpicture}%
}

\newcommand{\vertic}{%
\begin{tikzpicture}%
\bbox
\draw[line width=.4mm] (1ex,0) -- (1ex,2ex);%
\end{tikzpicture}%
}

\newcommand{\tikzcircle}{\tikz\draw[black, fill= black, radius=0.7ex] (0,0) circle ;}%

\makeatletter
\newtheorem*{rep@theorem}{\rep@title}
\newcommand{\newreptheorem}[2]{%
\newenvironment{rep#1}[1]{%
 \def\rep@title{#2 \ref{##1}}%
 \begin{rep@theorem}}%
 {\end{rep@theorem}}}
\makeatother

\numberwithin{equation}{section}

\author{Maxime LUCAS\footnote{Univ Paris Diderot, Sorbonne Paris Cit\'e, PPS, UMR 7126 CNRS, PiR2, INRIA Paris-Rocquencourt, F-75205 Paris, France - \texttt{maxime.lucas@pps.univ-paris-diderot.fr}}}

\title{Cubical $(\omega,p)$-categories}

\begin{document}

\maketitle

\begin{abstract} 
In this article we introduce the notion of cubical $(\omega,p)$-categories, for $p \in \mathbb N \cup \{\omega\}$. We show that the equivalence between globular and groupoid $\omega$-categories proven by Al-Agl, Brown and Steiner induces an equivalence between globular and cubical $(\omega,p)$-categories for all $p \geq 0$. In particular we recover in a more explicit fashion the equivalence between globular and cubical groupoids proven by Brown and Higgins. 

We also define the notion of $(\omega,p)$-augmented directed complexes, and show that Steiner's adjunction between augmented directed complexes and globular $\omega$-categories induces adjunctions between $(\omega,p)$-augmented directed complexes and both globular and cubical $(\omega,p)$-categories.

Combinatorially, the difficulty lies in defining the appropriate notion of invertibility for a cell in a cubical $\omega$-category. We investigate three such possible definitions and the relationships between them. We show that cubical $(\omega,1)$-categories have a natural structure of symmetric cubical categories. We give an explicit description of the notions of lax, oplax and pseudo transfors between cubical categories, the latter making use of the notion of invertible cell defined previously.
\end{abstract}

\tableofcontents

\newpage

\theoremstyle{plain}
\newtheorem{thm}{Theorem}[subsection]
\newtheorem{prop}[thm]{Proposition}
\newtheorem{lem}[thm]{Lemma}
\newtheorem{cor}[thm]{Corollary}
\newtheorem{conj}[thm]{Conjecture}
\newtheorem{quest}[thm]{Question}
\newreptheorem{cor}{Corollary}
\newreptheorem{thm}{Theorem}
\newreptheorem{prop}{Proposition}

\theoremstyle{definition}
\newtheorem{defn}[thm]{Definition}
\newtheorem{ex}[thm]{Example}
\newtheorem{remq}[thm]{Remark}
\newtheorem{nota}[thm]{Notation}

\deftwocell[circle, white]{partialM: 0 -> 1}
\deftwocell[circle, black]{partialP: 0 -> 1}
\deftwocell[circle, gray]{partial: 0 -> 1}
\deftwocell[circle, gray]{epsilon : 1 -> 0}
\deftwocell[polygon, white]{gammaM: 2 -> 1}
\deftwocell[polygon, black]{gammaP : 2 -> 1}
\deftwocell[polygon, gray]{gamma : 2 -> 1}
\deftwocell[crossing2]{braid3 : 3 -> 3}
\deftwocell[crossing]{braid2 : 2 -> 2}
\deftwocell[dots]{dots : 2 -> 2}

\deftwocell[text = \sigma]{sigma2 : 2 -> 2}
\deftwocell[text = \partial_i \sigma]{sigma4 : 4 -> 4}
\deftwocell[text = \partial_{i\cdot \sigma^-} \sigma]{sigma4minus : 4 -> 4}
\deftwocell[text = \partial_{i\cdot \sigma^-} \sigma]{sigma5minus : 5 -> 5}
\deftwocell[text = \sigma]{sigma5 : 5 -> 5}
\deftwocell[text = \sigma]{sigma6 : 6 -> 6}
\deftwocell[text = \tau]{tau2 : 2 -> 2}
\deftwocell[text = \partial_{i \cdot \sigma} \tau]{tau4 : 4 -> 4}
\deftwocell[text = \tau]{tau5 : 5 -> 5}
\deftwocell[mid = i]{i0 : 1 -> 1}
\deftwocell[mid = j]{j0 : 1 -> 1}
\deftwocell[mid = j-1]{j1 : 1 -> 1}
\deftwocell[mid = j_i]{ji : 1 -> 1}
\deftwocell[mid = i_j]{ij : 1 -> 1}
\deftwocell[mid = i \cdot \sigma^-]{isigmaminus : 1 -> 1}
\deftwocell[mid = i \cdot \sigma]{isigma : 1 -> 1}
\deftwocell[mid = i \cdot \sigma \cdot \tau]{ist : 1 -> 1}
\deftwocell[dots]{dots : 2 -> 2}

\section{Introduction}

\subsection{Cubical categories and their relationship with other structures}

\subsubsection*{An overview of cubical objects}

Handling higher structures such as higher categories usually involves conceiving them as conglomerates of cells of a certain shape. Such shapes include simplices, globes or cubes. Simplicial sets have been successfully applied to a wide variety of subjects. For example, they occur in May's work on the recognition principle for iterated loop spaces \cite{M64}, in Quillen's approach to rational homotopy theory \cite{Q69}, and in Bousfield
and Kan's work on completions, localisation, and limits in homotopy theory \cite{B72}.

Cubical objects however, have had a less successful history until recent years. Although cubical sets were used in early works by Serre \cite{S51} and Kan \cite{K55}, it became quickly apparent that they suffered from a few shortcomings. For instance, cubical groups were not automatically fibrant, and the cartesian product in the category of cubical sets failed to have the correct homotopy type. As a result, cubical sets mostly fell out of fashion in favour of simplicial sets. However later work on double groupoids, by Brown and Higgins, felt the need to add a new type of degeneracies on cubical sets called connections \cite{BS76} \cite{BH81_1}. By using these connections, a number of shortcomings of cubical objects were overcome. In particular the category of cubes with connections is a strict test category \cite{C06}\cite{M09}, and group objects in the category of cubical sets with connections are Kan \cite{T92}. Cubical objects with connections were particularly instrumental to the proof of a higher dimensional Van-Kampen Theorem by Brown and Higgins \cite{BH11}. 
Other applications of cubical structures arise in concurrency theory \cite{G00} \cite{G01} \cite{Ge03}, type theory \cite{BCH14}, algebraic topology \cite{G07}. Of interest is also the natural expression of the Gray-Crans tensor product of $\omega$-categories \cite{C95} in the cubical setting \cite{AS93} \cite{ABS02}.

\subsubsection*{Relationship with other structures}

A number of theorems relating objects of different shapes exist. For instance, Dold-Kan's correspondence states that in the category of abelian groups, simplicial objects, cubical sets with connections and strict $\omega$-groupoids (globular or cubical with connections) are all equivalent to chain complexes \cite{K58} \cite{B03}. 

Outside the category of abelian groups, the relationships between these notions become less straightforward. We are mainly concerned with the two following results:
\begin{itemize}
\item The first result is the equivalence between cubical and globular $\omega$-groupoids \cite{BH77} \cite{BH81_1} proven in 1981 by Brown and Higgins. Although this equivalence is useful in theory, in practice it is complicated to make explicit the functors composing this equivalence. This is due to the fact that the proof uses the notion of crossed complexes as a common ground between globular and cubical $\omega$-categories.
\item The second result is the equivalence between globular and cubical $\omega$-categories proved in 2002 \cite{ABS02}.
\end{itemize}

Lastly in 2004, Steiner \cite{S04} introduced the notion of augmented directed complexes (a variant of the notion of chain complexes) and proved the existence of an adjunction between augmented directed complexes and globular $\omega$-categories.

\subsubsection*{The case of \texorpdfstring{$(\omega,p)$}{(omega,p)}-categories}
Globular $(\omega,p)$-categories are globular $\omega$-categories where cells of dimension at least $p+1$ are invertible. They form a natural intermediate between globular $\omega$-categories, which correspond to the case $p = \omega$, and globular $\omega$-groupoids, which correspond to the case $p = 0$. As a consequence, they form a natural setting in which to develop directed algebraic topology \cite{G09} or rewriting \cite{GM12}.

However, both directed algebraic topology and rewriting seem to favour the cubical geometry (see once again \cite{G09} for directed algebraic topology, and \cite{L16} for rewriting), hence the need for a suitable notion of cubical $(\omega,p)$-categories.

The aim of this article is to define such a notion, so that when $p = 0$ or $ p = \omega$, we respectively recover the notions of cubical $\omega$-groupoids and cubical $\omega$-categories. Moreover, we bridge the gap between two results we cited previously by proving the following Theorem:
\begin{repthm}{thm:equiv_glob_cub}
Let $\lambda: \CCat{\omega} \to \Cat{\omega}$ and $\gamma: \Cat{\omega} \to \CCat{\omega}$ be the functors from \cite{ABS02} forming an equivalence of categories between globular and cubical $\omega$-categories. For all $p \geq 0$, their restrictions still induce an equivalence of categories:
\[
\begin{tikzpicture}
\matrix (m) [matrix of math nodes, 
			nodes in empty cells,
			column sep = 3cm, 
			row sep = .7cm] 
{
\Cat{(\omega,p)} & \CCat{(\omega,p)} \\
};

\draw[-to] (m-1-2) to [bend right=20] node (A) [above] {$\lambda$} (m-1-1);
\draw[-to] (m-1-1) to [bend right = 20] node (C) [below] {$\gamma$} (m-1-2);

\path (A) to node {$\cong$} (C);

\end{tikzpicture}
\]
\end{repthm}
In particular, we recover the equivalence between globular and cubical $\omega$-groupoids in a more explicit fashion.

We also define a notion of $(\omega,p)$-augmented directed complexes and show how to extend Steiner's adjunction. This is done in two steps. First we define functors $\mathcal Z^\CC : \CCat{\omega} \to \ADC$ and $N^\G : \ADC \to \CCat{\omega}$ (where $\ADC$ is the category of augmented directed complexes), as cubical analogues of the functors $\mathcal Z^\G : \Cat{\omega} \to \ADC$ and $N^\G : \ADC \to \Cat{\omega}$ forming Steiner's adjunction. We study how the relationship between these two pairs of functors and show that the functor $\mathcal Z^\CC$ is left-adjoint to $N^\CC$ (see Proposition \ref{prop:commutation_up_to_iso}). Then we show how to restrict the functors $\mathcal Z^\G$, $N^\G$, $\mathcal Z^\CC$ and $N^\CC$ to $(\omega,p)$-structures.   In the end, we get the following result:
\begin{repthm}{thm:triangle_adjunctions}
Let $\lambda: \CCat{\omega} \to \Cat{\omega}$ and $\gamma: \Cat{\omega} \to \CCat{\omega}$ be the functors from \cite{ABS02} forming an equivalence of categories between globular and cubical $\omega$-categories. Let $\mathcal Z^\G : \Cat{\omega} \to \ADC$ and $N^\G : \ADC \to \Cat{\omega}$ be the functors from \cite{S04} forming an adjunction between globular $\omega$-categories and ADCs. Let $\mathcal Z^\CC : \CCat{\omega} \to \ADC$ and $N^\G : \ADC \to \CCat{\omega}$ be the cubical analogues of $\mathcal Z^\G$ and $N^\G$ defined in Section \ref{subsec:ADC}.

For all $p \in \mathbb N \cup \{ \omega \}$, their restrictions induce the following diagram of equivalence and adjunctions between the categories  $\Cat{(\omega,p)}$, $\CCat{(\omega,p)}$ and $(\omega,p)$-$\ADC$, where both triangles involving $\mathcal Z^\CC$ and $\mathcal Z^\G$ and both triangles involving $N^\CC$ and $N^\G$ commute up to isomorphism:
\[
\begin{tikzpicture}
\matrix (m) [matrix of math nodes, 
			nodes in empty cells,
			column sep = 1cm, 
			row sep = 2.5cm] 
{
\Cat{(\omega,p)} & & \CCat{(\omega,p)}  \\
& (\omega,p)\text{-}\ADC & \\
};
\draw[-to] (m-2-2) to node (NG) [above right] {$N^\G$} (m-1-1);
\draw[transform canvas={xshift=-0.6cm}, -to] (m-1-1) to node (ZG) [below left] {$\mathcal Z^\G$} (m-2-2);
\draw[-to] (m-2-2) to node (NC) [above left] {$N^\CC$} (m-1-3);
\draw[transform canvas={xshift=0.6cm}, -to] (m-1-3) to node (ZC) [below right] {$\mathcal Z^\CC$} (m-2-2);
\draw[transform canvas={yshift=-0.2cm}, -to] (m-1-3) to node (L) [below] {$\gamma$} (m-1-1);
\draw[transform canvas={yshift=0.2cm}, -to] (m-1-1) to node (G) [above] {$\lambda$} (m-1-3);

\path (ZG) to node [left] {\rotatebox{134}{$\bot$}} (NG);
\path (ZC) to node [right] {\rotatebox{226}{$\bot$}} (NC);
\path (L) to node {$\cong$} (G);
\end{tikzpicture}
\]
\end{repthm}

\subsection{Invertibility in cubical categories}

The main combinatorial difficulty of this article consists in defining the appropriate notion of invertibility in cubical $\omega$-categories. Before giving an account of the various invertibility notions that we consider in the cubical setting, we start by recalling the more familiar notion of invertibility in $(2,1)$-categories.

\subsubsection*{Globular \texorpdfstring{$(2,1)$}{(2,1)}-categories}

Informally, a globular $(\omega,p)$-category is a globular $\omega$-category in which every $n$-cell is invertible, for $n > p$. For this definition to make rigorous sense, one first needs to define an appropriate notion of \emph{invertible} $n$-cells. Let us fix a globular $2$-category $\C$. There are two ways to compose two $2$-cells $A$ and $B$ in $\C_2$, that we denote by $\bullet_1$ and $\bullet_0$ and that are respectively known as the \emph{vertical} and \emph{horizontal} compositions. They can respectively be represented as follows:
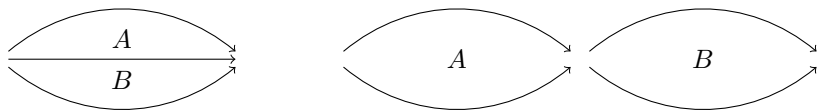
\[
\begin{tikzpicture}
\matrix (m) [matrix of math nodes, 
			nodes in empty cells,
			column sep = 3cm, 
			row sep = 1cm] 
{
 &   \\
};
\draw[-to] (m-1-1) to [bend left = 40] node (A1) [below] {} (m-1-2);
\draw[-to] (m-1-1) to node (B1) {} (m-1-2);
\draw[-to] (m-1-1) to [bend right = 40] node (C1) [above] {} (m-1-2);
\path (A1) to node {$A$} (B1);
\path (B1) to node {$B$} (C1); 
\end{tikzpicture}
\qquad
\begin{tikzpicture}
\matrix (m) [matrix of math nodes, 
			nodes in empty cells,
			column sep = 3cm, 
			row sep = 1cm] 
{
 &  &  \\
};
\draw[-to] (m-1-1) to [bend left = 40] node (A1) [below] {} (m-1-2);
\draw[-to] (m-1-1) to [bend right = 40] node (C1) [above] {} (m-1-2);
\draw[-to] (m-1-2) to [bend left = 40] node (A2) [below] {} (m-1-3);
\draw[-to] (m-1-2) to [bend right = 40] node (C2) [above] {} (m-1-3);
\path (A1) to node {$A$} (C1);
\path (A2) to node {$B$} (C2); 
\end{tikzpicture}
\]

We denote by $I_0 f: y \to x$ the inverse (if it exists) of a $1$-cell $f : x \to y$ in $\C_1$. A $2$-cell $A \in \C_2$ can have two inverses (one for each composition), that we denote respectively by $I_1A$ and $I_0A$. Their source and targets are as follows:
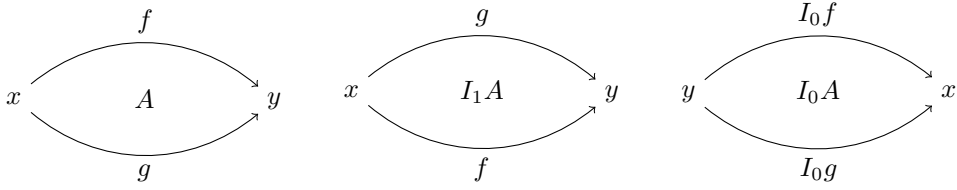
\[
\begin{tikzpicture}
\matrix (m) [matrix of math nodes, 
			nodes in empty cells,
			column sep = 3cm, 
			row sep = 1cm] 
{
x & y  \\
};
\draw[-to] (m-1-1) to [bend left = 40] node [above] {$f$} node (A1) [below] {} (m-1-2);
\draw[-to] (m-1-1) to [bend right = 40] node [below] {$g$} node (C1) [above] {} (m-1-2);
\path (A1) to node {$A$} (C1);
\end{tikzpicture}
\quad
\begin{tikzpicture}
\matrix (m) [matrix of math nodes, 
			nodes in empty cells,
			column sep = 3cm, 
			row sep = 1cm] 
{
x & y  \\
};
\draw[-to] (m-1-1) to [bend left = 40] node [above] {$g$} node (A1) [below] {} (m-1-2);
\draw[-to] (m-1-1) to [bend right = 40] node [below] {$f$} node (C1) [above] {} (m-1-2);
\path (A1) to node {$I_1 A$} (C1);
\end{tikzpicture}
\quad
\begin{tikzpicture}
\matrix (m) [matrix of math nodes, 
			nodes in empty cells,
			column sep = 3cm, 
			row sep = 1cm] 
{
y & x  \\
};
\draw[-to] (m-1-1) to [bend left = 40] node [above] {$I_0 f$} node (A1) [below] {} (m-1-2);
\draw[-to] (m-1-1) to [bend right = 40] node [below] {$I_0 g$} node (C1) [above] {} (m-1-2);
\path (A1) to node {$I_0 A$} (C1);
\end{tikzpicture}
\]
Note that if a $2$-cell is $I_0$-invertible, then so are its source and target, but that the $I_1$-invertibility of a $2$-cell does not imply any property for its source and target. So if $\C$ is a $2$-category where every $2$-cell is $I_0$-invertible, then $\C$ is a globular $2$-groupoid (indeed, a cell $1_f \in C_2$ is $I_0$-invertible if and only if $f$ is $I_0$-invertible). Therefore we say that a $2$-cell is invertible if it is $I_1$-invertible, and $\C$ is a globular $(2,1)$-category if each $2$-cell is $I_1$-invertible.

\subsubsection*{Cubical \texorpdfstring{$(2,1)$}{(2,1)}-categories}

In a cubical $2$-category $\CC$ (in what follows, cubical categories are always equipped with connections), the source and target of a $1$-cell $f \in \CC_1$ are respectively denoted $\partial_1^- f$ and $\partial_1^+ f$, and the source and target operations $\s,\t : \C_2 \to \C_1$ are replaced by four \emph{face operations} $\partial_i^\alpha : \CC_2 \to \CC_1$ (for $i = 1,2$ and $\alpha = \pm$), satisfying the cubical identity $\partial_1^\alpha \partial_2^\beta = \partial_1^\beta \partial_1^\alpha$. A $2$-cell $A \in \CC_2$ can be represented as follows, where the corners of the square are uniquely defined $0$-cells thanks to the cubical identity:
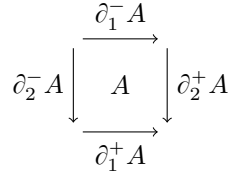
\[
\begin{tikzpicture}
\matrix (m) [matrix of math nodes, 
			nodes in empty cells,
			column sep = 1cm, 
			row sep = 1cm] 
{
 &  \\
 &  \\
};
\draw[-to] (m-1-1) to node [above] {$\partial_1^- A$} (m-1-2.west|-m-1-1);
\draw[-to] (m-2-1) to node [below] {$\partial_1^+ A$} (m-2-2.west|-m-2-1);
\draw[-to] (m-1-1) to node [left] {$\partial_2^- A$} (m-2-1);
\draw[-to] (m-1-2) to node [right] {$\partial_2^+ A$} (m-2-2);
\path (m-1-1) to node {$A$} (m-2-2);
\end{tikzpicture}
\]

There still are two ways to compose two $2$-cells $A,B \in \CC_2$, that we denote respectively by $A \star_1 B$ and $A \star_2 B$, which can be represented as follows:

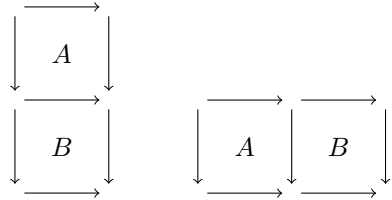
\[
\begin{tikzpicture}
\matrix (m) [matrix of math nodes, 
			nodes in empty cells,
			column sep = 1cm, 
			row sep = 1cm] 
{
 & \\
 & \\
 & \\
};
\draw[-to] (m-1-1) to (m-1-2);
\draw[-to] (m-2-1) to (m-2-2);
\draw[-to] (m-3-1) to (m-3-2);
\draw[-to] (m-1-1) to (m-2-1);
\draw[-to] (m-1-2) to (m-2-2);
\draw[-to] (m-2-1) to (m-3-1);
\draw[-to] (m-2-2) to (m-3-2);
\path (m-1-1) to node {$A$} (m-2-2);
\path (m-2-1) to node {$B$} (m-3-2);
\end{tikzpicture}
\qquad
\begin{tikzpicture}
\matrix (m) [matrix of math nodes, 
			nodes in empty cells,
			column sep = 1cm, 
			row sep = 1cm] 
{
 &  &  \\
 &  &  \\
};
\draw[-to] (m-1-1) to (m-1-2);
\draw[-to] (m-1-2) to (m-1-3);
\draw[-to] (m-2-1) to (m-2-2);
\draw[-to] (m-2-2) to (m-2-3);
\draw[-to] (m-1-1) to (m-2-1);
\draw[-to] (m-1-2) to (m-2-2);
\draw[-to] (m-1-3) to (m-2-3);
\path (m-1-1) to node {$A$} (m-2-2);
\path (m-1-2) to node {$B$} (m-2-3);
\end{tikzpicture}
\]

We say that a $2$-cell $A \in \CC_2$ is $R_i$-invertible if it is invertible for the composition $\star_i$ ($i = 1,2$). The faces of $R_1 A$ and $R_2 A$ are as follows (where $R_1f: y \to x$ denotes the inverse of a $1$-cell $f : x \to y$):

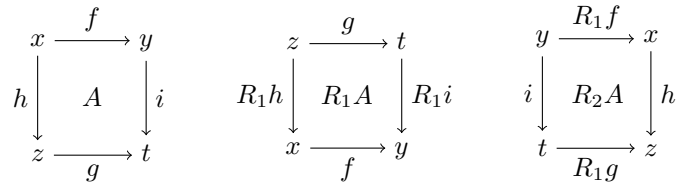
\[
\begin{tikzpicture}
\matrix (m) [matrix of math nodes, 
			nodes in empty cells,
			column sep = 1cm, 
			row sep = 1cm] 
{
x & y \\
z & t \\
};
\draw[-to] (m-1-1) to node [above] {$f$} (m-1-2.west|-m-1-1);
\draw[-to] (m-2-1) to node [below] {$g$} (m-2-2.west|-m-2-1);
\draw[-to] (m-1-1) to node [left] {$h$} (m-2-1);
\draw[-to] (m-1-2) to node [right] {$i$} (m-2-2);
\path (m-1-1) to node {$A$} (m-2-2);
\end{tikzpicture}
\qquad
\begin{tikzpicture}
\matrix (m) [matrix of math nodes, 
			nodes in empty cells,
			column sep = 1cm, 
			row sep = 1cm] 
{
z & t \\
x & y \\
};
\draw[-to] (m-1-1) to node [above] {$g$} (m-1-2.west|-m-1-1);
\draw[-to] (m-2-1) to node [below] {$f$} (m-2-2.west|-m-2-1);
\draw[-to] (m-1-1) to node [left] {$R_1h$} (m-2-1);
\draw[-to] (m-1-2) to node [right] {$R_1i$} (m-2-2);
\path (m-1-1) to node {$R_1A$} (m-2-2);
\end{tikzpicture}
\qquad
\begin{tikzpicture}
\matrix (m) [matrix of math nodes, 
			nodes in empty cells,
			column sep = 1cm, 
			row sep = 1cm] 
{
y & x \\
t & z \\
};
\draw[-to] (m-1-1) to node [above] {$R_1f$} (m-1-2.west|-m-1-1);
\draw[-to] (m-2-1) to node [below] {$R_1g$} (m-2-2.west|-m-2-1);
\draw[-to] (m-1-1) to node [left] {$i$} (m-2-1);
\draw[-to] (m-1-2) to node [right] {$h$} (m-2-2);
\path (m-1-1) to node {$R_2A$} (m-2-2);
\end{tikzpicture}
\]

Note that contrary to the notion of $I_1$-invertibility, the $R_1$ and $R_2$-invertibility of $A$ require respectively that $\partial_2^\alpha A$ and $\partial_1^\alpha A$ are $R_1$-invertible (for $\alpha = \pm$). We say that $A$ has respectively an $R_1$ or an $R_2$-invertible shell if that is the case. As a consequence, if $\CC$ is a cubical $2$-category where every $2$-cell is $R_1$-invertible, then every $1$-cell of $\CC$ is $R_1$-invertible (one can even show that such a cubical $2$-category is actually a cubical $2$-groupoid) and the same property holds for $R_2$. In order to have a good notion of cubical $(\omega,p)$-categories nonetheless, we have to be more careful in our definition of an invertible cell. 

\subsubsection*{Invertibility in cubical \texorpdfstring{$(\omega,p)$}{(omega,p)}-categories}

This is done in Section \ref{subsec:Ri_invertibility}, where we define a notion of invertibility for an $n$-cell ($n \geq 1$). Let us first recall that, using the structure of connections on $\CC$ (an additional structure on cubical $\omega$-categories introduced in \cite{BS76} \cite{BH81_1}), one can associate to any $1$-cell $f : x \to y$ in $\CC_1$, the cells $\Gamma_1^-f$ and $\Gamma_1^+ f$, which can be represented as follows:
\[
\begin{tikzpicture}
\matrix (m) [matrix of math nodes, 
			nodes in empty cells,
			column sep = 1cm, 
			row sep = 1cm] 
{
x & y \\
y & y \\
};
\draw[-to] (m-1-1) to node [above] {$f$} (m-1-2.west|-m-1-1);
\doublearrow{arrows = {-}}
{(m-2-1) to node [below] {$\epsilon_1 y$} (m-2-2.west|-m-2-1)}
\draw[-to] (m-1-1) to node [left] {$f$} (m-2-1);
\doublearrow{arrows = {-}} 
{(m-1-2) to node [right] {$\epsilon_1 y$} (m-2-2)}
\path (m-1-1) to node {$\Gamma_1^- f$} (m-2-2);
\end{tikzpicture}
\qquad
\begin{tikzpicture}
\matrix (m) [matrix of math nodes, 
			nodes in empty cells,
			column sep = 1cm, 
			row sep = 1cm] 
{
x & x \\
x & y \\
};
\doublearrow{arrows = {-}}
{(m-1-1) to node [above] {$\epsilon_1 x$} (m-1-2.west|-m-1-1)}
\draw[-to] (m-2-1) to node [below] {$f$} (m-2-2.west|-m-2-1);
\doublearrow{arrows = {-}}
{(m-1-1) to node [left] {$\epsilon_1 x$} (m-2-1)}
\draw[-to] (m-1-2) to node [right] {$f$} (m-2-2);
\path (m-1-1) to node {$\Gamma_1^+ f$} (m-2-2);
\end{tikzpicture}
\]
We say that a $2$-cell $A \in \CC_2$ is invertible if the following composite (denoted $\Phi_2 A$) is $R_1$-invertible:
\[
\begin{tikzpicture}
\matrix (m) [matrix of math nodes, 
			nodes in empty cells,
			column sep = 1.5cm, 
			row sep = 1.5cm] 
{
 &  &  & \\
 &  &  & \\
};
\doublearrow{arrows = {-}}
{(m-1-1) to (m-1-2)}
\draw[-to] (m-1-2) to (m-1-3);
\draw[-to] (m-2-1) to (m-2-2);
\doublearrow{arrows = {-}}
{(m-2-3) to (m-2-4)}
\doublearrow{arrows = {-}}
{(m-1-1) to (m-2-1)}
\draw[-to] (m-1-2) to (m-2-2);
\doublearrow{arrows = {-}}
{(m-1-4) to (m-2-4)}
\draw[-to] (m-1-3) to (m-1-4);
\draw[-to] (m-2-2) to (m-2-3);
\draw[-to] (m-1-3) to (m-2-3);
\path (m-1-1) to node {$\Gamma_1^+ \partial_2^- A$} (m-2-2);
\path (m-1-2) to node {$A$} (m-2-3);
\path (m-1-3) to node {$\Gamma_1^- \partial_2^+ A$} (m-2-4);
\end{tikzpicture}
\]
Note in particular that $\partial_2^- \Phi_2 A$ and $\partial_2^+ \Phi_2 A$ are both identities (which are always invertible), and so the $R_1$-invertibility of $\Phi_2 A$ does not require the invertibility of any face of $A$. More generally for any $n \geq 0$ there is an operator $\Phi_n : \mathbf{C}_n \to  \mathbf{C}_n$ which ``globularizes'' the $n$-cells (see \cite{ABS02} or Definition \ref{defn:folding_op}. The link between invertibility, $R_i$-invertibility and having an $R_i$-invertible shell is given by the following Proposition:
\begin{repprop}{prop:i_inv_implies_inv}
Let $\CC$ be a cubical $\omega$-category, $A \in \CC_n$ and $1 \leq j \leq n$. A cell $A \in \CC_n$ is $R_j$-invertible if and only if $A$ is invertible and has an $R_j$-invertible shell. 
\end{repprop}

We also investigate in Section \ref{subsec:T_i_invertibility} another notion of invertibility, with respect to a kind of ``diagonal'' composition, that we call the $T_i$-invertibility. If $A$ is a $2$-cell in a cubical $2$-category, then the $T_1$-inverse of $A$ (if it exists) has the following faces:

\[
\begin{tikzpicture}
\matrix (m) [matrix of math nodes, 
			nodes in empty cells,
			column sep = 1cm, 
			row sep = 1cm] 
{
x & y \\
z & t \\
};
\draw[-to] (m-1-1) to node [above] {$f$} (m-1-2.west|-m-1-1);
\draw[-to] (m-2-1) to node [below] {$g$} (m-2-2.west|-m-2-1);
\draw[-to] (m-1-1) to node [left] {$h$} (m-2-1);
\draw[-to] (m-1-2) to node [right] {$i$} (m-2-2);
\path (m-1-1) to node {$A$} (m-2-2);
\end{tikzpicture}
\qquad
\begin{tikzpicture}
\matrix (m) [matrix of math nodes, 
			nodes in empty cells,
			column sep = 1cm, 
			row sep = 1cm] 
{
x & z \\
y & t \\
};
\draw[-to] (m-1-1) to node [above] {$h$} (m-1-2.west|-m-1-1);
\draw[-to] (m-2-1) to node [below] {$i$} (m-2-2.west|-m-2-1);
\draw[-to] (m-1-1) to node [left] {$f$} (m-2-1);
\draw[-to] (m-1-2) to node [right] {$g$} (m-2-2);
\path (m-1-1) to node {$T_1A$} (m-2-2);
\end{tikzpicture}
\]
We then define a suitable notion of $T_i$-invertible shells and prove the following result, analogous to Proposition \ref{prop:i_inv_implies_inv}:

\begin{repprop}{prop:carac_Ti_invertibility_inv}
Let $\CC$ be a cubical $\omega$-category, and $A \in \CC_n$, with $n \geq 2$. Then $A$ is $T_i$-invertible if and only if $A$ is invertible and has a $T_i$-invertible shell.
\end{repprop}

The study of the relationship between $R_i$-invertibility, $T_i$-invertibility and (plain) invertibility gives rise to the following Proposition:
\begin{repprop}{prop:caracterisation_omega_p}
Let $\CC$ be a cubical $\omega$-category, and fix $n > 0$. The following five properties are equivalent:
\begin{enumerate}
\item Any $n$-cell in $\CC_n$ is invertible.
\item For all $1 \leq i \leq n$, any $n$-cell in $\CC_n$ with an $R_i$-invertible shell is $R_i$-invertible.
\item Any $n$-cell in $\CC_n$ with an $R_1$-invertible shell is $R_1$-invertible.
\item Any $n$-cell $A \in \CC_n$ such that $\partial_j^\alpha A \in \im \epsilon_1$ for all $j \neq 1$ is $R_1$-invertible.
\item Any $n$-cell in $\Phi_n(\CC_n)$ is $R_1$-invertible.
\end{enumerate}
Moreover, if $n > 1$, then all the previous properties are also equivalent to the following:
\begin{enumerate}
\item[6.] For all $1 \leq i < n$, any $n$-cell in $\CC_n$ with a $T_i$-invertible shell is $T_i$-invertible
\item[7.] Any $n$-cell in $\CC_n$ with a $T_1$-invertible shell is $T_1$-invertible.
\end{enumerate}
\end{repprop}

We can now define a cubical $(\omega,p)$-category as a cubical $\omega$-category where every $n$-cell is invertible, for $n > p$, and we prove the equivalence with the globular notion.

\subsection{Permutations and cubical \texorpdfstring{$(\omega,p)$}{(omega,p)}-categories}
\subsubsection*{Cubical \texorpdfstring{$(\omega,p)$}{(omega,p)}-categories are symmetric}
In Section \ref{subsec:sym_cubical_categories}, we extend the notion of $T_i$-invertibility of an $n$-cell to that of $\sigma$-invertibility, for $\sigma$ an element of the symmetric group $S_n$. In particular, we show that if $\CC$ is a cubical $(\omega,1)$-category, then every cell of $\CC$ is $T_i$-invertible, and therefore $\sigma$-invertible, for any $\sigma \in S_n$. Consequently, we get an action of the symmetric group $S_n$ on the set of $n$-cells $\CC_n$, making $\CC$ a symmetric cubical category (in a sense related to that of Grandis \cite{G07}).

\subsubsection*{Definition of \texorpdfstring{$k$}{k}-transfors}
In Section \ref{subsec:transfors}, we apply the notion of invertibility to $k$-transfors between cubical $\omega$-categories. A $k$-transfor (following terminology by Crans \cite{C03}) from $\CC$ to $\DD$ is a family of maps $\CC_n \to \DD_{n+k}$ satisfying some compatibility conditions. These compatibility conditions come in two varieties, leading to the notions of lax and oplax $k$-tranfors (respectively called $k$-fold left and right homotopies in \cite{ABS02}). In particular, the lax or oplax $0$-transfors are just the functors from $\CC$ to $\DD$, and a lax or oplax $1$-transfor $\eta$ between functors $F$ and $G$ is the cubical analogue of a lax or oplax natural transformation from $F$ to $G$. For example, a $0$-cell in $x \in  \CC_0$ is sent to a $1$-cell $\eta_x : F(x) \to G(x)$ in $\DD_1$, and a $1$-cell $f : x \to y$ in $\CC_1$ is sent to a $2$-cell $\eta_f$ in $\DD_2$ of the following shape (respectively if $\eta$ is lax or oplax):

\[
\begin{tikzpicture}
\matrix (m) [matrix of math nodes, 
			nodes in empty cells,
			column sep = 1cm, 
			row sep = 1cm] 
{
F(x) & F(y) \\
G(x) & G(y) \\
};
\draw[-to] (m-1-1) to node [above] {$F(f)$} (m-1-2.west|-m-1-1);
\draw[-to] (m-2-1) to node [below] {$G(f)$} (m-2-2.west|-m-2-1);
\draw[-to] (m-1-1) to node [left] {$\eta_x$} (m-2-1);
\draw[-to] (m-1-2) to node [right] {$\eta_y$} (m-2-2);
\path (m-1-1) to node {$\eta_f$} (m-2-2);
\end{tikzpicture}
\qquad
\begin{tikzpicture}
\matrix (m) [matrix of math nodes, 
			nodes in empty cells,
			column sep = 1cm, 
			row sep = 1cm] 
{
F(x) & G(x) \\
F(y) & G(y) \\
};
\draw[-to] (m-1-1) to node [above] {$\eta_x$} (m-1-2.west|-m-1-1);
\draw[-to] (m-2-1) to node [below] {$\eta_y$} (m-2-2.west|-m-2-1);
\draw[-to] (m-1-1) to node [left] {$F(f)$} (m-2-1);
\draw[-to] (m-1-2) to node [right] {$G(f)$} (m-2-2);
\path (m-1-1) to node {$\eta_f$} (m-2-2);
\end{tikzpicture}
\]

As shown in \cite{ABS02}, Section 10, lax and oplax transfors from $\CC$ to $\DD$ respectively form cubical $\omega$-categories $\Lax(\CC,\DD)$ and $\oplax(\CC,\DD)$. We define notions of pseudo transfors as transfors satisfying some invertibility conditions. In particular in the case of $1$-transfors, we require for any $1$-cell $f$ in $\CC_1$ that $\eta_f$ is $T_1$-invertible. We show that pseudo lax and pseudo oplax transfors from $\CC$ to $\DD$ still form cubical $\omega$-categories $\plax(\CC,\DD)$ and $\poplax(\CC,\DD)$, and prove the following result:

\begin{repprop}{prop:transfors_isomorphism}
For all cubical $\omega$-categories $\CC$ and $\DD$, the cubical $\omega$-categories $\plax(\CC,\DD)$ and $\poplax(\CC,\DD)$ are isomorphic.
\end{repprop}

For example if $\eta$ is a lax $1$-transfor, then the map $\CC_1 \to \DD_2$ which is part of the oplax $1$-transfor associated to $\eta$ maps a cell $f$ in $\CC_1$ to a $2$-cell  $T_1 \eta_f$ in $\DD_2$.

\subsection{Organisation}

In Section \ref{sec:recall_cub_cat}, we recall a number of results on cubical $\omega$-categories. In particular, we recall the definition of the two functors forming the equivalence between globular and cubical $\omega$-categories. 

In Section \ref{sec:invertibility} we study the various forms of invertibility that exist in cubical $\omega$-categories. In particular we define the notions of $R_i$-invertibility and (plain) invertibility in Section \ref{subsec:Ri_invertibility}, and the notion of $T_i$-invertibility in Section \ref{subsec:T_i_invertibility}. 

In Section \ref{sec:relationship_omega_p}, we finally define cubical $(\omega,p)$-categories. In Section \ref{subsec:cubical_omega_p_categories} we use the results on invertibility that we collected throughout Section \ref{sec:invertibility}, and we prove the equivalence with the globular notion and characterize the notions of cubical $(\omega,0)$ and $(\omega,1)$-categories. In Section \ref{subsec:ADC} we introduce the notion of $(\omega,p)$-ADCs and study its relationship with both globular and cubical $(\omega,p)$-categories.

Lastly in Section \ref{sec:applications}, we apply the notions of invertibility as studied beforehand, to show firstly that cubical $(\omega,1)$-categories carry a natural structure of symmetric cubical categories in Section \ref{subsec:sym_cubical_categories}, then in Section \ref{subsec:transfors} we define and study the notion of pseudo transfors between cubical $\omega$-categories.

\paragraph*{Acknowlegments}

This work was supported by the Sorbonne-Paris-Cit\'e IDEX grant Focal and
the ANR grant ANR-13-BS02-0005-02 CATHRE. I also thank Yves Guiraud for the numerous helpful discussions we had during the preparation of this article.

\section{Cubical categories}
\label{sec:recall_cub_cat}

In this section we recall the notion of $\omega$-cubical categories (with connections) and the following functors 
\[
\begin{tikzpicture}
\matrix (m) [matrix of math nodes, 
			nodes in empty cells,
			column sep = 3cm, 
			row sep = .7cm] 
{
\Cat{\omega} & \CCat \omega \\
};
\draw[-to] (m-1-2) to [bend right=20] node (A) [above] {$\lambda$} (m-1-1);
\draw[-to] (m-1-1) to [bend right = 20] node (C) [below] {$\gamma$} (m-1-2);

\path (A) to node {$\cong$} (C); 
\end{tikzpicture}
\]
defined in \cite{ABS02} that form an equivalence between the category of cubical $\omega$-categories and that of globular $\omega$-categories.

While our description of the functor $\lambda$ matches exactly the description given in \cite{ABS02}, we rephrase slightly the definition of $\gamma$. Our construction consists in defining a co-\emph{cubical $\omega$-category} object in $\Cat{\omega}$ (that is a cubical $\omega$-category object in $\Cat{\omega}^{\op}$), in order to define $\gamma$ as a nerve functor. The starting point of this construction consists in describing the \emph{standard globular $\omega$-category of the $n$-cube} (denoted $\cube{n}^\G$ in this paper, and $M(I^n)$ in \cite{ABS02}). Here we use the closed monoidal structure on $\Cat{\omega}$ to construct these $\omega$-categories, but one could equivalently define them as in \cite{ABS02} using directed complexes \cite{S93}, or using augmented directed complexes \cite{S04}.

\subsection{Cubical sets}

\begin{defn}
We denote by $\Cat{n}$ the category of strict globular $n$-categories (with $n \in \mathbb N \cup \{ \omega \}$). We implicitly consider all globular $n$-categories (with $n \in \mathbb N$) to be globular $\omega$-categories whose only cells in dimension higher than $n$ are identities.  Let $\C$ be a globular $\omega$-category and $n \geq 0$. We denote by $\C_n$ the set of $n$-cells of $\C$. For $f \in \C_n$, and $0 \leq k < n$, we denote by $\s_k(f) \in \C_k$ (resp. $\t_k(f)$) the $k$-dimensional source (resp. target) of $f$, and we simply write $\s(f)$ (resp. $\t(f)$) for $\s_{n-1}(f)$ (resp. $\t_{n-1}(f)$). For $f,g \in \C_n$ such that $\t_k(f)=\s_k(g)$ we write $f \bullet_k g$ for their composite. For $f \in \C_n$ we write $1_f$ for the identity of $f$. Finally for $x,y \in \C_0$, we denote by $\C(x,y)$ the globular $\omega$-category of arrows between them.

We say that an $n$-cell $f \in \C_n$ is \emph{invertible} if it is invertible for the composition $\bullet_{n-1}$, that is if there exists an $n$-cell $g \in \C_n$ such that $f \bullet_{n-1} g = 1_{\s(f)}$ and $g \bullet_{n-1} f = 1_{\t(f)}$. For $p \geq 0$, a \emph{globular $(\omega,p)$-category} is a globular $\omega$-category in which any $n$-cell is invertible, for $n >p$. In particular, a globular $(\omega,0)$-category is just a globular $\omega$-groupoid.
\end{defn}

\begin{defn}\label{def:preCSet}
A pre-cubical set is a series of sets $C_n$ (for $n \geq 0$) together with maps (called \emph{face operations}) $\partial_i^\alpha : C_n \to C_{n-1}$, for $\alpha = \pm$ and $1 \leq i \leq n$, satisfying for all $1 \leq i < j \leq n$:
\begin{equation}\label{eq:partial_vs_partial}
\partial_{j-1}^\alpha \partial_i^\beta = \partial_{i}^\beta \partial_{j}^\alpha
\end{equation}

A morphism of pre-cubical sets is is a family of maps $F_n : C_n \to D_n$ commuting with the faces operations.
\end{defn}

\begin{ex}
Following work of Grandis and Mauri \cite{GM03}, pre-cubical sets can be seen as presheaves over the free PRO generated by cells $\twocell{partialM} : 0 \to 1$ and $\twocell{partialP} : 0 \to 1$. Then the maps $\partial_i^- : C_n \to C_{n-1}$ and $\partial_i^+ : C_n \to C_{n-1}$ correspond respectively to the following cells, with $i-1$ strings on the left and $n-i$ on the right:
\[
\twocell{dots *0 partialM *0 dots}
\qquad 
\twocell{dots *0 partialP *0 dots}
\]

Equation \eqref{eq:partial_vs_partial} corresponds to equations of the following form, replacing the occurrences of $\twocell{partial}$ either by $\twocell{partialM}$ or $\twocell{partialP}$ depending on $\alpha$ and $\beta$:
\begin{equation}\label{eq:commute_partial}
\twocell{(4 *0 partial *0 2) *1 (dots *0 dots *0 1 *0 dots) *1 (2 *0 partial *0 5)}
=
\twocell{(2 *0 partial *0 4) *1 (dots *0 1 *0 dots *0 dots) *1 (5 *0 partial *0 2)}
\end{equation}
In general, reading an expression $\partial_i^\alpha \ldots \partial_j^\beta $ from left to right corresponds to reading a string diagram in the PRO from top to bottom.

Note that the symmetry of Equation \eqref{eq:commute_partial} is broken in Equation \eqref{eq:partial_vs_partial}. This is hidden in the fact that the cells $\twocell{partialP}$ and $\twocell{partialM}$ cause a re-indexing of the strings. More precisely, numbering the strings from left to right in the string diagram representation of $\partial_i^\alpha$, we have, respectively for $j < i$ and $j > i$:
\[
\twocell{(2 *0 j0 *0 4) *1 (dots *0 1 *0 dots *0 partial *0 dots) *1 (2*0 j0 *0 2 *0 i0 *0 2)}
\qquad
\twocell{(4 *0 j1 *0 2) *1 (dots *0 partial *0 dots *0 1 *0 dots) *1 (2 *0 i0 *0 2 *0 j0 *0 2)}
\]
As a consequence, we introduce in Definition \ref{def:nota_indices} the notation $j_i := j$ when $j < i$ and $j_i := j-1$ when $j > i$. These same string diagrams then become:
\[
\twocell{(2 *0 ji *0 2 *0 2) *1 (dots *0 1 *0 dots *0 partial *0 dots) *1 (2 *0 j0 *0 2 *0 i0 *0 2)}
\qquad
\twocell{(2 *0 2 *0 ji *0 2) *1 (dots *0 partial *0 dots *0 1 *0 dots) *1 (2 *0 i0 *0 2 *0 j0 *0 2)}
\]
Using this notation, Equation \eqref{eq:commute_partial} becomes the following:
\[
\twocell{(4 *0 partial *0 2) *1 (2 *0 2 *0 1 *0 2) *1 (dots *0 dots *0 ji *0 dots) *1 (2 *0 partial *0 5) *1 (2 *0 i0 *0 2 *0 j0 *0 2)}
=
\twocell{(2 *0 partial *0 4) *1 (2 *0 1 *0 2 *0 2) *1 (dots *0 ij *0 dots *0 dots) *1 (5 *0 partial *0 2) *1 (2 *0 i0 *0 2 *0 j0 *0 2)}
\]
Equation \eqref{eq:partial_vs_partial} can then be written as follows, for $1 \leq i < j \leq n$:
\begin{equation}\label{eq:partial_vs_partial_bis}
\partial_{j_i}^\alpha \partial_i^\beta = \partial_{i_j}^\beta \partial_j^\alpha
\end{equation}
Note in particular this expression is symmetric in $i$ and $j$, so we can relax the condition $i < j$ in $i \neq j$.
\end{ex}

Before moving on towards the definition of cubical $\omega$-category, we define properly the notation $i_j$ introduced in the previous example, and its symmetric $i^j$, together with a few properties. This will allow us to express the axioms of a cubical set and of a cubical $\omega$-category in a more symmetric manner, which will be useful in later sections.

\begin{defn}\label{def:nota_indices}
For every $i \in \mathbb N$, we define two maps $(\_)^i : \mathbb N \to \mathbb N \setminus \{i\}$ and $(\_)_i : \mathbb N \setminus \{i\} \to \mathbb N$ as follows:
\begin{align*}
j^i :=&  \begin{cases}
j & j < i \\
j+1 & j \geq i
\end{cases}
&
j_i :=&  \begin{cases}
j & j < i \\
j-1 & j > i
\end{cases}
\end{align*}

Finally, let $i$, $j$ be distinct integers. We define maps $(\_)_{i,j}$,  $(\_)^{i,j}$ and $(\_)^j_i$ respectively as follows: 
\[
\begin{cases}
\mathbb N \setminus \{i,j\} &\to \quad \mathbb N \\
k &\mapsto \quad (k_{i})_{j_i}
\end{cases}
\qquad
\begin{cases}
\mathbb N &\to \quad \mathbb N \setminus \{i,j\} \\
k & \mapsto \quad (k^{i_j})^{j}
\end{cases}
\qquad
\begin{cases}
\mathbb N \setminus \{i_j\} &\to \quad \mathbb N \setminus \{j_i\} \\
k & \mapsto \quad (k^j)_i
\end{cases}
\]
\end{defn}

\begin{lem}\label{lem:permut_indices}
The following equalities hold, for every $k$ and every $i \neq j$:
\[
\begin{cases}
k^{i,j} = k^{j,i} & \\
k_{i,j} = k_{j,i} & k \neq i,j \\
k^j_i = (k_{i_j})^{j_i} & k \neq i_j
\end{cases}
\]
\end{lem}
\begin{proof}

Recall that there is at most one isomorphism between any two well-ordered sets. Here $(\_)^{i,j}$ and $(\_)^{j,i}$ are both isomorphism from $\mathbb N$ to $\mathbb N \setminus \{i,j\}$, hence they are equal. The same reasoning proves the other two equalities.
\end{proof}

Definitions \ref{def:CSet} and \ref{def:CCat} are exactly the same as in \cite{ABS02}, except that we make use of the notations introduced in Definition \ref{def:nota_indices}.

\begin{defn}\label{def:CSet}
A cubical set (with connections) is given by:
\begin{itemize}
\item For all $n \in \mathbb N$, a set $C_n$,
\item For all $n \in \mathbb N^*$, all $1 \leq i \leq n$ and all $\alpha \in \{+,-\}$, maps $\partial_i^\alpha : C_n \to C_{n-1}$.
\item For all $n \in \mathbb N$ and all $1 \leq i \leq n+1$, maps $\epsilon_i: C_n \to C_{n+1}$.
\item For all $n \in \mathbb N^*$, all $1 \leq i \leq n$ and all $\alpha \in \{+,-\}$, maps $\Gamma_i^\alpha: C_n \to C_{n+1}$.
\end{itemize}
This data must moreover verify the following axioms:

\begin{multicols}{2}

\begin{equation}
\partial_i^\alpha \epsilon_j = 
\begin{cases}
\epsilon_{j_i} \partial_{i_j}^\alpha & i \neq j \\
\id_{C_n} & i = j
\end{cases}
\end{equation}

\begin{equation}
\partial_i^\alpha \Gamma_j^\beta = 
\begin{cases}
\Gamma_{j_i}^\beta \partial_{i_j} & i \neq j,j+1 \\
\id_{C_n} & i = j,j+1 \text{ and } \alpha = \beta \\
\epsilon_j \partial_j^\alpha & i = j,j+1 \text{ and } \alpha = - \beta
\end{cases}
\end{equation}

\begin{equation}
\epsilon_i \epsilon_{j^i} = \epsilon_{j} \epsilon_{i^j}
\end{equation}

\begin{equation}
\Gamma_{i^j}^\alpha \Gamma_j^\beta = 
\begin{cases}
\Gamma_{j^i}^\beta \Gamma_i^\alpha & i \neq j \\
\Gamma_i^\alpha \Gamma_i^\alpha & i = j \text{ and } \alpha = \beta
\end{cases}
\end{equation}

\begin{equation}
\Gamma_{i}^\alpha \epsilon_j = 
\begin{cases}
\epsilon_{j^i} \Gamma_{i_j}^\alpha & i \neq j \\
\epsilon_i  \epsilon_i & i = j
\end{cases}
\end{equation}
\end{multicols}
\end{defn}

\begin{ex}
Following once again \cite{GM03}, cubical sets with connections can be seen as presheaves over the following PRO, denoted by $\mathbb J$ and called the \emph{intermediate cubical site} in \cite{GM03}:
\begin{itemize}
\item The generators are the cells :
\[
\twocell{partialM} : 0 \to 1
\qquad
\twocell{partialP} : 0 \to 1
\qquad
\twocell{epsilon} : 1 \to 0
\qquad 
\twocell{gammaM} : 2 \to 1
\qquad
\twocell{gammaP} : 2 \to 1
\]
\item They satisfy the following relations :
\begin{multicols}{2}
\[
\twocell{partialM *1 epsilon} = \quad
\]
\[
\twocell{(partialM *0 1) *1 gammaM} = \twocell{1} = \twocell{(1 *0 partialM) *1 gammaM}
\]
\[
\twocell{(partialP *0 1) *1 gammaM} = \twocell{epsilon *1 partialP} = \twocell{(1 *0 partialP) *1 gammaM}
\]
\[
\twocell{(gammaM *0 1) *1 gammaM} = \twocell{(1 *0 gammaM) *1 gammaM}
\]
\[
\twocell{gammaM *1 epsilon} = \twocell{epsilon *0 epsilon}
\]

\[
\twocell{partialP *1 epsilon} = \quad
\]
\[
\twocell{(partialP *0 1) *1 gammaP} = \twocell{1}  = \twocell{(1 *0 partialP) *1 gammaP}
\]
\[
\twocell{(partialM *0 1) *1 gammaP} = \twocell{epsilon *1 partialM} = \twocell{(1 *0 partialM) *1 gammaP}
\]
\[
\twocell{(gammaP *0 1) *1 gammaP} = \twocell{(1 *0 gammaP) *1 gammaP}
\]
\[
\twocell{gammaP *1 epsilon} = \twocell{epsilon *0 epsilon} 
\]
\end{multicols}
\end{itemize}
Then the maps $\Gamma_i^- : C_n \to C_{n+1}$,  $\Gamma_i^+ : C_n \to C_{n+1}$ and $\epsilon_i$ correspond respectively to composites of the form $\twocell{dots *0 gammaM *0 dots}$, $\twocell{dots *0 gammaP *0 dots}$ and $\twocell{dots *0 epsilon *0 dots}$, with the appropriate number of strings on each side.
\end{ex}

\subsection{Cubical \texorpdfstring{$\omega$}{omega}-categories}

\begin{defn}\label{def:CCat}
A cubical $\omega$-category is given by a cubical set $\CC$, equipped with, for all $n \in \mathbb N^*$ and all $1 \leq i \leq n$, a partial map $\star_i$ from $\CC_n \times \CC_n$ to $\CC_n$ defined exactly for any cells $A$,$B$ such that $\partial_i^+ A = \partial_i^- B$. This data must moreover satisfy the following axioms:

\begin{multicols}{2}
\begin{equation}\label{eq:exchange}
(A \star_i B) \star_j (C \star_i D) = (A \star_j C) \star_i (B \star_j D)
\end{equation}

\begin{equation}
A \star_i (B \star_i C) = (A \star_i B) \star_i C
\end{equation}

\begin{equation}
\epsilon_i (A \star_j B) = 
\epsilon_i^\alpha A \star_{j^i} \epsilon_i^\alpha B
\end{equation}

\begin{equation}
A \star_i \epsilon_i \partial_i^+ A = \epsilon_i \partial_i^- A \star_i A = A
\end{equation}
\begin{equation}\label{eq:compo_Gamma_1}
\Gamma_i^+ A \star_i \Gamma_i^- A = \epsilon_{i+1} A 
\end{equation}
\begin{equation}\label{eq:compo_Gamma_2}
\Gamma_i^+ A \star_{i+1} \Gamma_i^- A = \epsilon_i A
\end{equation}
\end{multicols}

\begin{equation}
\partial_i^\alpha (A \star_j B) = 
\begin{cases}
\partial_i^\alpha A \star_{j_i} \partial_i^\alpha B & i \neq j \\
\partial_i^- A & i = j \text{ and } \alpha = - \\
\partial_i^+ B & i = j \text{ and } \alpha = +
\end{cases}
\end{equation}
\begin{equation}\label{eq:gamma_of_comp}
\Gamma_i^\alpha(A \star_j B) = 
\begin{cases}
\Gamma_i^\alpha A \star_{j^i} \Gamma_i^\alpha B & i \neq j \\
  \begin{tabular}{ | c | c | }
  \hline			
    $\Gamma_i^-  A$ & $ \epsilon_{i+1} B $\\
  \hline
   $\epsilon_i B$ & $\Gamma_i^- B$ \\
  \hline  
\end{tabular}
\qquad
\directions{i}{i+1}
& i = j \text{ and } \alpha = - \\
  \begin{tabular}{ | c | c | }
  \hline			
    $\Gamma_i^+  A$ & $ \epsilon_{i} A $\\
  \hline
   $\epsilon_{i+1} A$ & $\Gamma_i^+ B$ \\
  \hline  
\end{tabular}
\qquad
\directions{i}{i+1}
& i = j \text{ and } \alpha = +
\end{cases}
\end{equation}
where in the last relation we denote by $\begin{tabular}{ | c | c | }
  \hline			
    $A$ & $B $\\
  \hline
   $C$ & $D$ \\
  \hline  
\end{tabular}
\quad
\directions{i}{j}$ the composite $(A \star_i B) \star_j (C \star_i D)$ (which is made possible by relation \eqref{eq:exchange}). We denote by $\CCat{\omega}$ the category of cubical $\omega$-category.
\end{defn}

\begin{defn}\label{defn:folding_op}
Let $\CC$ be a cubical $\omega$-category. For any $n > 0$, we define \emph{folding operations} $\psi_i,\Psi_r, \Phi_m: \CC_n \to \CC_n$, with $1 \leq i \leq n-1$, $1 \leq r \leq n$ and $0 \leq m \leq n$ as follows:

\begin{eqnarray*}
\psi_i A & = & \Gamma_i^+ \partial_{i+1}^- A \star_{i+1} A \star_{i+1} \Gamma_i^- \partial_{i+1}^+ A \\
\Psi_r A & = & \psi_{r-1} \cdots \psi_1 A \\
\Phi_m A & = & \Psi_1 \cdots \Psi_{m} A
\end{eqnarray*}
\end{defn}

\begin{ex}
The folding operations are used to ``globularize'' an $n$-cell. For example, if $A$ is a $2$-cell of a $2$-category $\mathbf C$, then $\Phi_2 A = \Psi_2 A = \psi_1 A$ is the following cell:
\[
\begin{tikzpicture}
\matrix (m) [matrix of math nodes, 
			nodes in empty cells,
			column sep = 1.5cm, 
			row sep = 1.5cm] 
{
 &  &  & \\
 &  &  & \\
};
\doublearrow{arrows = {-}}
{(m-1-1) to (m-1-2)}
\draw[-to] (m-1-2) to node [above] {$\partial_1^- A$} (m-1-3);
\draw[-to] (m-2-1) to node [below] {$\partial_1^+ A$} (m-2-2);
\doublearrow{arrows = {-}}
{(m-2-3) to (m-2-4)}
\doublearrow{arrows = {-}}
{(m-1-1) to (m-2-1)}
\draw[-to] (m-1-2) to (m-2-2);
\doublearrow{arrows = {-}}
{(m-1-4) to (m-2-4)}
\draw[-to] (m-1-3) to node [above] {$\partial_2^+ A$} (m-1-4);
\draw[-to] (m-2-2) to node [below] {$\partial_2^- A$} (m-2-3);
\draw[-to] (m-1-3) to (m-2-3);
\path (m-1-1) to node {$\Gamma_1^+ \partial_2^- A$} (m-2-2);
\path (m-1-2) to node {$A$} (m-2-3);
\path (m-1-3) to node {$\Gamma_1^- \partial_2^+ A$} (m-2-4);
\end{tikzpicture}
\]
As we will see, the $2$-cells of the globular $\omega$-category  $\gamma \mathbf C$ (where  $\gamma : \CCat{\omega} \to \Cat{\omega}$ is the functor forming the equivalence of categories between globular and cubical $\omega$-categories) will be exactly the $2$-cells of $\mathbf C$ of the form $\Phi_2 A$, with source and target given by:
\[
\begin{tikzpicture}
\matrix (m) [matrix of math nodes, 
			nodes in empty cells,
			column sep = 2cm, 
			row sep = 1cm] 
{
 &  &  \\
 &  &  \\
 &  &  \\
};
\draw[-to] (m-2-1) to [bend left] node  [above left] {$\partial_1^- A$} (m-1-2);
\draw[-to] (m-1-2) to [bend left] node  [above right] {$\partial_2^+ A$} (m-2-3);
\draw[-to] (m-2-1) to [bend right] node  [below left] {$\partial_2^- A$} (m-3-2);
\draw[-to] (m-3-2) to [bend right] node  [below right] {$\partial_1^+ A$} (m-2-3);
\doublearrow{-Implies, shorten <= 0.4cm, shorten > = 0.4cm }{(m-1-2) to node [fill = white] {$\Phi_2 A$} (m-3-2)}
\end{tikzpicture}
\]
\end{ex}

\begin{defn}
Let $\CC$ be a cubical $\omega$-category, and $A \in \CC_n$. We say that $A$ is a \emph{thin} cell if $\psi_1 \ldots \psi_{n-1} A \in \im \epsilon_1$.
\end{defn}

Thin cells are the cubical notion corresponding to identity cells in a globular $\omega$-category. In a cubical $\omega$-category $\mathbf C$ for example, thin $2$-cells correspond to commutative squares of $1$-cells of $\mathbf{C}$. The main result about thin cells is Theorem \ref{thm:thin_cells}. Before we can state this theorem, we need to define the notion of \emph{shell} of an $n$-cell.

\begin{defn}
Let $n  \in \mathbb N$. There is a truncation functor $\tr_n : \CCat{(n+1)} \to \CCat{n}$. This functor admits both a left and a right adjoint (see \cite{H05} for an explicit description of both functors).

\[
\begin{tikzpicture}
\matrix (m) [matrix of math nodes, 
			nodes in empty cells,
			column sep = 3cm, 
			row sep = .7cm] 
{
\CCat{(n+1)} & \CCat n \\
};
\draw[transform canvas={yshift=0.4cm},-to] (m-1-2) to [bend right=20] node (A) [above] {$\boxslash$} (m-1-1);
\draw[-to] (m-1-1) to node (B) [fill = white] {$\tr_n$} (m-1-2);
\draw[transform canvas={yshift=-0.4cm},-to] (m-1-2) to [bend left = 20] node (C) [below] {$\Box$} (m-1-1);

\path (A) to node {$\bot$} (B);
\path (B) to node {$\bot$} (C); 
\end{tikzpicture}
\]

For $\CC \in \CCat{n}$, the $(n+1)$-category $\Box \CC$ coincides with $\CC$ up to dimension $n$, and the rest of the structure is defined as follows:
\begin{itemize}
\item The set of $(n+1)$-cells of $\Box \CC$ is the set of all families $(A_i^\alpha) \in \CC_n$ (with $1 \leq i \leq n+1$ and $\alpha = \pm$) such that:
 \[
 \partial_{i_j}^\alpha A_j^\beta = \partial_{j_i}^\beta A_i^\alpha.
 \] 
\item For $A \in (\Box \CC)_{n+1}$, $\partial_i^\alpha A = A_i^\alpha$.
\item For $A \in \CC_n$, the families $\epsilon_i A \in (\Box \CC)_{n+1}$ and $\Gamma_i^\alpha A \in (\Box \CC)_{n+1}$ are defined by:
\[
(\epsilon_i A)_j^\beta = \begin{cases}
A & j = i \\
\epsilon_{i_j} \partial_{j_i}^\beta A & j \neq i
\end{cases}
\qquad
(\Gamma_i^\alpha)_j^\beta = \begin{cases}
A & j = i,i+1 \text{ and } \beta = \alpha \\
\epsilon_i \partial_i^\beta A & j = i,i+1 \text{ and } \beta = - \alpha \\
\Gamma_{i_j}^\alpha \partial_{j_i} A & j \neq i,i+1
\end{cases}
\]
\item For $A,B \in (\Box \CC)_{n+1}$ such that $A_i^+ = B_i^-$, the family $A \star_i B \in (\Box \CC)_{n+1}$ is defined by:
\[
(A \star_i B)_j^\alpha = \begin{cases}
A_i^- & j = i \text { and } \alpha = - \\
B_i^+ & j = i \text { and } \alpha = + \\
A_j^\alpha \star_{i_j} B_j^\alpha & j \neq i
\end{cases}
\] 
\end{itemize}

Let $\CC$ be a cubical $(n+1)$-category. The unit of the adjunction $\tr \dashv \Box$ induces a morphism of cubical $(n+1)$-categories $\bm \partial : \CC \to \Box \tr \CC$. This functor $\bm$ associates, to any $A \in \CC_{n+1}$ the family $\bm{\partial}A := (\partial_i^\alpha A)$. We call $\bm{\partial} A$ the \emph{shell} of $A$.

More generally, if $\CC$ is a cubical $\omega$-category, we denote by $\Box_n \CC$ the $(n+1)$-category $\Box \tr_n \CC$, and for any $A \in \CC_{n+1}$, by $\bm \partial A$ the cell $\bm \partial \tr_{n+1} A \in \Box_n \CC$.
\end{defn}

\begin{thm}[Proposition 2.1 and Theorem 2.8 from \cite{H05}] \label{thm:thin_cells}
Let $\CC$ be a cubical category. Thin cells of $\CC$ are exactly the composites of cells of the form $\epsilon_i f$ and $\Gamma_i^\alpha f$. Moreover, if two thin cells have the same shell, then they are equal.
\end{thm}

\begin{nota}\label{nota:string_diagrams}
As a consequence, when writing thin cells in $2$-dimensional compositions (as in Equation \eqref{eq:gamma_of_comp} for example), we make use of the notation already used in \cite{ABS02} and \cite{H05}: a thin cell $A$ is replaced by a string diagram linking the non-thin faces of $A$. For example $\Gamma_i^+ A$ and $\Gamma_i^- A$ will respectively be represented by the symbols $\fbox \cornerdr$ and $\fbox \cornerul$, and the cells $\epsilon_i A$  by the symbol $\fbox \horiz$ or $\fbox \vertic$. If every face of a thin cell is thin (such as $\epsilon_i \epsilon_i f$), then we simply denote it by an empty square $\fbox \nocorner$ . Following this convention, Equations \eqref{eq:compo_Gamma_1} and \eqref{eq:compo_Gamma_2} can be represented by the following string diagrams:
\[
  \renewcommand{\arraystretch}{1.5}
  \begin{tabular}{ | c | c | }
  \hline			
    $\cornerdr$ & $ \cornerul $\\
  \hline  
\end{tabular}
=
\fbox \vertic
\qquad
  \begin{tabular}{ | c | }
  \hline			
    $\cornerdr$ \\ \hline
    $ \cornerul $\\
  \hline  
\end{tabular}
=
\fbox \horiz
\qquad
\directions{i}{i+1}
\]
And the last two cases of Equation \eqref{eq:gamma_of_comp} become respectively:
\[
\renewcommand{\arraystretch}{1.5}
\fbox \cornerul = 
  \begin{tabular}{ | c | c | }
  \hline			
    $\cornerul$ & $ \vertic $\\
  \hline
   $\horiz$ & $\cornerul$ \\
  \hline  
\end{tabular}
\qquad
\fbox \cornerdr =
  \begin{tabular}{ | c | c | }
  \hline			
    $\cornerdr$ & $ \horiz $\\
  \hline
   $\vertic$ & $\cornerdr$ \\
  \hline  
\end{tabular}
\qquad
\directions{i}{i+1}
\]
Finally, for any $A \in \C_n$, $\psi_i A$ is the following composite:
\[
\renewcommand{\arraystretch}{1.5}
\psi_i A =   
\begin{tabular}{ | c | c | c |}
  \hline			
    $\cornerdr$ & $ A $ & $\cornerul$ \\
  \hline  
\end{tabular}
\qquad
\directions{i+1}{i}
\]

\end{nota}

\subsection{Equivalence between cubical and globular \texorpdfstring{$\omega$}{omega}-categories}

The functor $\gamma : \CCat \omega \to \Cat \omega$ was described in \cite{ABS02} as follows.

\begin{prop}
Let $\CC$ be a cubical category. The following assignment defines a globular $\omega$-category $\gamma \CC$ :
\begin{itemize}
\item The set of $n$-cells of $\gamma \CC$ is the set $ \Phi_n(\CC_n)$,
\item For all $A \in \gamma \CC_n$, $1_A := \epsilon_1 A$,
\item For all $A \in \gamma \CC_n$, $\s(A) := \partial_1^- A$,
\item For all $A \in \gamma \CC_n$, $\t(A) := \partial_1^+ A$,
\item For all $A,B \in \gamma \CC_n$ and $0 \leq k < n$, $A \bullet_k B := A \star_{n-k} B$.
\end{itemize}
\end{prop}

To define the functor $\lambda : \Cat{\omega} \to \CCat{\omega}$, we start by constructing a co-cubical $\omega$-category object in $\Cat{\omega}$. This is a reformulation of \cite{ABS02}.

\begin{defn}\label{def:standard_cub_set}
Let $\mathbb I$ be the category with two $0$-cells $(-)$ and $(+)$ and one non-identity $1$-cell $(\pmzerodot)$:
\[
(\pmzerodot): (-) \to (+)
\]

We denote by $\cube{n}^{\G}$, and call the \emph{$n$-cube category} the globular $\omega$-category $\mathbb I^{\otimes n}$, where $\otimes$ is the Crans-Gray tensor product, which equips $\Cat{\omega}$ with a closed monoidal structure.
\end{defn}

\begin{ex}
For example $\cube{2}^\G$ is the free $2$-category with four $0$-cells, four generating $1$-cells and one generating $2$-cell, with source and targets given by the following diagram:

\[
\begin{tikzpicture}
\matrix (m) [matrix of math nodes, 
			nodes in empty cells,
			column sep = 2cm, 
			row sep = 2cm] 
{
(--) &  (+-) \\
(-+) &  (++) \\
};
\draw[-to] (m-1-1) to node [above] {$(\pmzerodot -)$} (m-1-2.west|-m-1-1);
\draw[-to] (m-2-1) to node [below] {$(\pmzerodot +)$} (m-2-2.west|-m-2-1);
\draw[-to] (m-1-1) to node [left] {$(- \pmzerodot)$} (m-2-1);
\draw[-to] (m-1-2) to node [right] {$(+ \pmzerodot)$} (m-2-2);
\doublearrow{arrows = {->}, shorten <= .3cm, shorten >= .3cm}
{(m-1-2) -- node [fill = white] {$A$} (m-2-1)}
\end{tikzpicture}
\]
\end{ex}

\begin{defn}
For $\alpha = \pm$, we denote by $\check \partial^\alpha : \top \to \mathbb I$ the functor sending the (unique) $0$-dimensional cell of $\top$ to $(\alpha)$, where $\top$ denotes the terminal category.

For any $n \geq 0$,  any $1 \leq i \leq n$ and any $\alpha = \pm$, we denote by $\check \partial_i^\alpha : \cube n^{\G} \to \cube{(n+1)}^\G$ the functor $\mathbb I^{i-1} \otimes \check \partial^\alpha \otimes \mathbb I^{n-i}$.
\end{defn}

\begin{defn}
We denote by $\check \epsilon : \cube 1^\G \to \cube 0^\G$ the (unique) functor from $\mathbb I$ to $\top$.

For any $n > 0$ and any $1 \leq i \leq n$, we denote by $\check \epsilon_i : \cube{(n-1)}^\G \to \cube n^\G$ the functor $\mathbb I^{\otimes(i-1)} \otimes \,\check \epsilon \otimes \mathbb I^{\otimes(n-i)}$.
\end{defn}

\begin{defn}
For $\alpha = \pm$, let $\check \Gamma^\alpha : \cube 2^\G \to \cube 1^\G$ be the functor defined as follows, where $\beta = -\alpha$:
\[
\begin{cases}
\check \Gamma^\alpha (\alpha \alpha) = (\alpha) \\
\check \Gamma^\alpha (\alpha \beta) = (\beta) \\
\check \Gamma^\alpha (\beta \alpha) = (\beta) \\
\check \Gamma^\alpha (\beta \beta) = (\beta) \\
\end{cases}
\qquad
\begin{cases}
\check \Gamma^\alpha (\pmzerodot \alpha) = (\pmzerodot) \\
\check \Gamma^\alpha (\pmzerodot \beta) = 1_{(\beta)} \\
\check \Gamma^\alpha (\alpha \pmzerodot) = (\pmzerodot) \\
\check \Gamma^\alpha (\beta \pmzerodot) = 1_{(\beta)} \\
\end{cases}
\qquad
\check \Gamma^\alpha (\pmzerodot \pmzerodot) = 1_{(\pmzerodot)}
\]

For any $n > 0$ any $1 \leq i \leq n$ and any $\alpha = \pm$, we denote by $\check \Gamma_i^\alpha : \cube{n}^\G \to \cube{(n+1)}^\G$ the functor $ \mathbb I^{\otimes(i-1)} \otimes \,\check \Gamma^\alpha \otimes \mathbb I^{\otimes(n-i)}$.
\end{defn}

\begin{defn}
We denote by $\mathbf{Rect^\G}$ the following coproduct in $\Cat{\omega}$:
\begin{equation}\label{eq:diag_for_Rect1}
\begin{tikzcd}
\cube{0}^\G  \ar[d, "\check \partial^-"'] \ar[r, "\check \partial^+"] \ar[dr, phantom, "\ulcorner", very near end] & \cube 1^\G \ar[d]\\
\cube 1^\G \ar[r] & \mathbf{Rect^\G}\\
\end{tikzcd}
\end{equation}

Explicitly, the $0$-cells of $\mathbf{Rect}^\G$ are elements $(\alpha_j)$, where $\alpha = \pm$ and $i = 1,2$, with the identification $(+_1) = (-_2)$. The $1$-cells of $\mathbf{Rect}^\ADC_0$ are freely generated by $(\pmzerodot_i): (-_i) \to (+_i)$, for $i = 1,2$.

For every $n >0$ and every $1 \leq i \leq n$, let $\Rect{(n,i)}^\G$ be the cubical $\omega$-category:
\[
\Rect{(n,i)}^\G := \mathbb I^{(i-1)} \otimes \; \mathbf{Rect}^\G \otimes \; \mathbb I^{(n-i)}
\]
\end{defn}

\begin{remq}
Since the monoidal structure on $\Cat{\omega}$ is biclosed \cite{ABS02}, $\Rect{(n,i)}^\G$ is the colimit of the following diagram:
\begin{equation}
\begin{tikzcd}\label{eq:diagram_for_Rectn}
\cube{(n-1)}^\G  \ar[d, "\check \partial_i^-"'] \ar[r, "\check \partial_i^+"] \ar[dr, phantom, "\ulcorner", very near end] & \cube n^\G \ar[d]\\
\cube n^\G \ar[r] & \Rect{(n,i)}^\G\\
\end{tikzcd}
\end{equation}
\end{remq}

\begin{defn}
We denote by $\check \star : \cube{1}^\G \to \mathbf{Rect}^\G$ the following functor:
\[
\begin{cases}
\check \star (-) = (-_1) \\
\check \star (+) = (+_2) 
\end{cases}
\qquad 
\check \star(\pmzerodot) = (\pmzerodot_1) \bullet_0 (\pmzerodot_2)
\]

For any $n > 0$ and any $1 \leq i \leq n$, we denote by $\check \star_i : \cube{n}^\G \to \Rect{(n,i)}^\G$ the functor $\mathbb I^{\otimes(i-1)} \otimes \; \check \star \otimes \; \mathbb I^{\otimes(n-i)}$.
\end{defn}

This result is a reformulation of Section 2 of \cite{ABS02}:
\begin{prop}
The objects $\cube{n}^\G$ equipped with the maps $\check \partial_i^\alpha$, $\check \epsilon_i$, $\check \Gamma_i^\alpha$ and $\check \star_i$ form a co-\emph{cubical $\omega$-category} object in the category $\Cat{\omega}$. 

Consequently, for $\C$ a globular $\omega$-category, the family $(\lambda \C)_n = \Cat{\omega}(\cube n^\G,\C)$ comes equipped with a structure of cubical $\omega$-category, that we denote by $\lambda \C$. This defines a functor $\lambda : \Cat{\omega} \to \CCat{\omega}$.
\end{prop}
Finally, the main result of \cite{ABS02} is the following:
\begin{thm}
The following functors form an equivalence of Categories:
\[
\begin{tikzpicture}
\matrix (m) [matrix of math nodes, 
			nodes in empty cells,
			column sep = 3cm, 
			row sep = .7cm] 
{
\Cat{\omega} & \CCat \omega \\
};
\draw[-to] (m-1-2) to [bend right=20] node (A) [above] {$\lambda$} (m-1-1);
\draw[-to] (m-1-1) to [bend right = 20] node (C) [below] {$\gamma$} (m-1-2);

\path (A) to node {$\cong$} (C); 
\end{tikzpicture}
\]
\end{thm}

\section{Invertible cells in cubical \texorpdfstring{$\omega$}{omega}-categories}
\label{sec:invertibility}

In this Section, we investigate three notions of invertibility in cubical $\omega$-categories. We start by defining in Section \ref{subsec:Ri_invertibility} the notion of $R_i$-invertibility, which is a direct cubical analogue of the usual notion of invertibilty with respect to a binary composition. In Section \ref{subsec:plain_invertibility} we define the notion of (plain) invertibility, which is specific to cubical $\omega$-categories, and relate it to $R_i$-invertibility. Finally in Section \ref{subsec:T_i_invertibility}, we define the notion of $T_i$-invertibility, a variant of the notion of $R_i$-invertibility using a kind of diagonal composition. 

\subsection{\texorpdfstring{$R_i$}{Ri}-invertibility}\label{subsec:Ri_invertibility}

We start by defining the notion of $R_i$-invertibility and prove a number of preliminary Lemmas. In particular we relate the $R_i$-invertibility of a cell to that of its shell and study the $R_i$-invertibility of thin cells. Those Lemmas will prove useful in Section \ref{subsec:plain_invertibility}. 

\begin{defn}
Let $\CC$ be a cubical $\omega$-category, and $1 \leq k \leq n$ be integers. We say that a cell $A \in \CC_n$ is \emph{$R_k$-invertible} if there exists $B \in \CC_n$ such that $A \star_k B = \epsilon_k \partial_k^- A$ and $B \star_k A = \epsilon_k \partial_k^+ A$. We call $B$ the $R_k$-inverse of $A$, and we write $R_kA$ for $B$. 

In particular, we say that $A \in \CC_n$ has an $R_k$-invertible shell if $\bm \partial A$ is $R_k$-invertible in $\Box_n \CC$.
\end{defn}

\begin{lem}\label{lem:faces_and_inverses}
Let $\CC$ be a cubical $n$-category, and $A \in (\Box \CC)_{n+1}$. Then $A$ is $R_i$-invertible if and only if for all $j \neq i$, $A_j^\alpha$ is $R_{i_j}$-invertible, and:
\begin{align*}
\partial_i^\alpha R_k A = 
\begin{cases}
\partial_k^{-\alpha} A & i = k \\
R_{k_i} \partial_{i}^\alpha A & i \neq k 
\end{cases}
\end{align*}

In particular, for $\CC$ a cubical $\omega$-categories, a cell $A \in \CC_n$ has an $R_i$-invertible shell if and only if $\partial_j^\alpha A$ is $R_{i_j}$-invertible for any $j \neq i$.
\end{lem}
\begin{proof}
Let $B$ be the $R_k$-inverse of $A$, and $i \neq k$. We have:
\begin{align*}
A_i^\alpha \star_{k_i} B_i^\alpha &= (A\star_k B)_i^\alpha = \partial_i^\alpha \epsilon_k A_k^- = \epsilon_{k_i} \partial_{i^k}^\alpha A_k^- = \epsilon_{k_i} \partial_{k_{i^k}}^- A_i^\alpha  = \epsilon_{k_i} \partial_{k_i}^- A_i^\alpha, \\
B_i^\alpha \star_{k_i} A_i^\alpha &= (B \star_k A)_i^\alpha  = \partial_i^\alpha \epsilon_k A_k^+  = \epsilon_{k_i} \partial_{i^k}^\alpha A_k^+  = \epsilon_{k_i} \partial_{k_{i^k}}^+ A_i^\alpha  = \epsilon_{k_i} \partial_{k_i}^+ A_i^\alpha . 
\end{align*}
Thus $B_i^\alpha $ is the $k_i$-inverse of $A_i^\alpha $, that is $\partial_i^\alpha R_k A = R_{k_i} \partial_i^\alpha A$.

Moreover, for the composite $A \star_k R_k A$ (resp. $R_k A \star_k A$) to make sense we necessarily have $\partial_k^- R_k A = \partial_k^+ A$ (resp. $\partial_k^+ R_k A = \partial_k^- A$).
\end{proof}

The following Lemma will be useful in order to compute the $R_i$-inverses of thin cells.

\begin{lem}\label{lem:carac_inverse_thin_cells}
Let $\CC$ be a cubical $\omega$-category, and let $A$ be a thin cell in $\CC_n$. We fix an integer $i \leq n$. If there exists a thin cell $B$ in $\CC_n$ such that $\partial_i^\alpha B = \partial_i^{-\alpha} A$, and for all $j \neq i$, $\partial_j^\alpha B = R_{i_j} \partial_j^\alpha A$, then $A$ is $R_i$-invertible, and $B = R_iA$.
\end{lem}
\begin{proof}
Since $\partial_i^- B = \partial_i^+ A$, $A$ and $B$ are $i$-composable. Let us look at the cell $A \star_i B$. It is a thin cell, and it has the following shell:
\[
\partial_j^\alpha (A \star_i B) = \begin{cases}
\partial_i^- A = \partial_i^- \epsilon_i \partial_i^- A & j = i \text{ and } \alpha = - \\
\partial_i^+ B = \partial_i^- A = \partial_i^- \epsilon_i \partial_i^- A & j = i \text{ and } \alpha = + \\
\partial_j^\alpha A \star_{i_j} \partial_j^\alpha B =
\partial_j^\alpha A\star_{i_j} R_{i_j} \partial_j^\alpha A =
\epsilon_{i_j} \partial_{i_j}^- \partial_j^\alpha A =
\partial_j^\alpha \epsilon_i \partial_i^- A & j \neq i
\end{cases}
\]
Therefore, $A \star_i B$ and $\epsilon_i \partial_i^- A$ are two thin cells that have the same shell. By Theorem \ref{thm:thin_cells}, they are equal. The same computation with $B \star_i A$ leads to the equality $B \star_i A = \epsilon_i \partial_i^+ A$. Finally, $A$ is $R_i$-invertible, and $R_i A = B$.
\end{proof}

\begin{lem}\label{lem:closure_i_invertibility}
Let $\CC$ be a cubical $\omega$-category, and fix $A,B \in \CC_n$ and $1 \leq k \leq n$.
\begin{itemize}
\item For any $i \leq n$, if $A,B$ are $R_k$-invertible and $i$-composable, then $A \star_i B$ is $R_k$-invertible, and:
\begin{equation}\label{eq:Rk_of_star}
R_k(A \star_i B) = \begin{cases}
R_k A \star_i R_k B & i \neq k \\
R_k B \star_k R_k A & i = k
\end{cases}
\end{equation}
\item For any $i \leq n+1$, $\epsilon_i A$ is $R_i$-invertible and $R_i \epsilon_i A = \epsilon_i A$. Moreover if $A$ is $R_k$-invertible then $\epsilon_i A$ is also $R_{k^i}$ invertible, with 
\begin{equation}\label{eq:Rk_of_epsilon}
R_{k^i} \epsilon_i A = \epsilon_i R_k A
\end{equation}
\item For any $i \neq k$ and $\alpha = \pm$, if $A$ is $R_k$-invertible, then $\Gamma_i^\alpha A$ is $R_{k^i}$ invertible, and $\Gamma_k^\alpha A$ is both $R_k$ and $R_{k+1}$-invertible, and:
\begin{equation}\label{eq:Rk_of_Gamma1}
R_{k^i} \Gamma_i^\alpha A = \Gamma_i^\alpha R_k A
\end{equation}
\begin{equation}\label{eq:Rk_of_Gamma2}
R_k \Gamma_k^\alpha A = \begin{cases}
\epsilon_{k+1} R_k A \star_{k+1} \Gamma_k^+ A & \alpha = - \\
\Gamma_k^- A \star_{k} \epsilon_{k+1} R_k A & \alpha  = +
\end{cases} 
\qquad
R_{k+1} \Gamma_k^\alpha A = \begin{cases}
\epsilon_k R_k A \star_{k+1} \Gamma_k^+ A & \alpha = - \\
\Gamma_k^- A \star_{k+1} \epsilon_k R_k A & \alpha = +
\end{cases}
\end{equation}
\end{itemize}
\end{lem}
\begin{proof}
Suppose $A$ and $B$ are $k$-invertible, and let $i \leq n$. If  $i \neq k$, Then we have: 
\begin{align*}
(R_k A \star_i R_k B) \star_k (A \star_i B) & = (R_k A \star_k A) \star_i (R_k B \star_k B) = \epsilon_k \partial_k^+ A \star_i \epsilon_k \partial_k^+ B = \epsilon_k \partial_k^+ (A \star_i B) \\
(A \star_i B) \star_k (R_k A \star_i R_k B) & = (A \star_k R_k A) \star_i (B \star_k R_k B) = \epsilon_k \partial_k^- A \star_i \epsilon_k \partial_k^- B = \epsilon_k \partial_k^- (A \star_i B).
\end{align*}
Thus $A \star_i B$ is $R_k$-invertible and $R_k(A \star_i B) = R_k A \star_i R_k B$. Suppose now that $i = k$. Then we have:
\begin{align*}
R_k B \star_k R_k A \star_k A \star_k B & = \epsilon_k \partial_k^+ B = \epsilon_k \partial_k^+ (A \star_k B) \\
A \star_k B \star_k R_k B \star_k R_k A & = \epsilon_k \partial_k^- A = \epsilon_k \partial_k^- (A \star_k B).
\end{align*}
So $A \star_k B$ is $R_k$-invertible, and $R_k (A \star_k B) = R_k B \star_k R_k A$.

Suppose $i \neq k$. Then we have:
\begin{align*}
\Gamma_i^\alpha A \star_{k^i} \Gamma_i^\alpha R_k A & = \Gamma_i^\alpha (A \star_k R_k A) = \Gamma_i^\alpha \epsilon_k \partial_k^- A = \epsilon_{k^i} \Gamma_{i_k}^\alpha \partial_k^- A = \epsilon_{k^i} \partial_{k^i}^- \Gamma_i^\alpha A \\
\Gamma_i^\alpha R_k A \star_{k^i} \Gamma_i^\alpha A & = \Gamma_i^\alpha (R_k A \star_k A) = \Gamma_i^\alpha \epsilon_k \partial_k^+ A = \epsilon_{k^i} \Gamma_{i_k}^\alpha \partial_k^+ A = \epsilon_{k^i} \partial_{k^i}^+ \Gamma_i^\alpha A
\end{align*}
Thus $\Gamma_i^\alpha A$ is $R_{k^i}$-invertible, and $R_{k^i} \Gamma_i^\alpha A = \Gamma_i^\alpha R_k A$.

Suppose now $i = k$, and $\alpha = -$. In order to show that $R_k \Gamma_k^- A = \epsilon_{k+1} R_k A \star_{k+1} \Gamma_k^+ A$, we are going to use Lemma \ref{lem:carac_inverse_thin_cells}. Note first that both $\Gamma_k^- A$ and $\epsilon_{k+1} R_k A \star_{k+1} \Gamma_k^+ A$ are thin, so we only need to check the hypothesis about the shell of $\epsilon_{k+1} R_k A \star_{k+1} \Gamma_k^+ A$. Note that the hypotheses on directions $k$ and $k+1$ are always satisfied:
\[
\partial_j^\alpha (\epsilon_{k+1} R_k A \star_{k} \Gamma_k^+ A) =
\begin{cases}
\epsilon_k \partial_k^- R_k A = \epsilon_k \partial_k^+ A = \partial_k^+ \Gamma_k^- A & j = k \text{ and } \alpha = - \\
\partial_k^+ \Gamma_k^+ A = A = \partial_k^- \Gamma_k^- A & j = k \text{ and } \alpha = + \\
R_k A \star_k \partial_{k+1}^- \Gamma_k^+ A = R_k A \star_k \epsilon_k \partial_k^- A = R_k A = R_k \partial_k^- \Gamma_k^- A & j = k+1 \text{ and } \alpha = - \\
R_k A \star_k \partial_{k+1}^+ \Gamma_k^+ A = R_k A \star_k A = \epsilon_k \partial_k^+ A = R_k \partial_k^+ \Gamma_k^- A & j = k+1 \text{ and } \alpha = +
\end{cases}
\]

As for the remaining directions, we reason by induction on $n$, the dimension of $A$. In the case in which $n = 1$ (and thus $k = 1$), there is no other direction to check and so $R_1 \Gamma_1^- = \epsilon_2 R_1 A \star_2 \Gamma_1^+ A$. Suppose now $n > 1$, and let $j \leq n+1$, with $j \neq k, k+1$. Then we have the following equalities (where the fourth one uses the induction hypothesis): 
\begin{align*}
\partial_j^\alpha (\epsilon_{k+1} R_k A \star_{k} \Gamma_k^+ A) & =
\partial_j^\alpha \epsilon_{k+1} R_k A \star_{k_j} \partial_j^\alpha \Gamma_k^+ A \\
& = \epsilon_{(k+1)_j} \partial_{j_{k+1}}^\alpha R_k A \star_{k_j} \Gamma_{k_j}^+ \partial_{j_k}^\alpha A \\
& = \epsilon_{k_j + 1}  R_{k_j} \partial_{j_k}^\alpha A \star_{k_j} \Gamma_{k_j}^+ \partial_{j_k}^\alpha A \\
& = R_{k_j} \Gamma_{k_j}^- \partial_{j_k}^\alpha A \\
& = R_{k_j} \partial_j^\alpha \Gamma_k^- A
\end{align*}
Thus by Lemma \ref{lem:carac_inverse_thin_cells}, $\Gamma_k^- A$ is $R_k$-invertible, and $R_k \Gamma_k^- A = \epsilon_{k+1} R_k A \star_{k+1} \Gamma_k^+ A$.

The proofs of the remaining three cases ($i = k$ with $\alpha = +$, and $i = k+1$ with $\alpha = \pm$) are similar.
\end{proof}

\begin{remq}
Note that  Lemma \ref{lem:closure_i_invertibility} shows in particular that, if $A$ is $R_k$-invertible, then $R_{k^i} \Gamma_i^\alpha$, $R_k \Gamma_k^\alpha A$ and $R_{k+1} \Gamma_k^\alpha A$ are thin. In particular, applying the Notation defined in \ref{nota:string_diagrams}, we get the equations:
\[
  R_k \: \fbox \cornerdr = \fbox \cornerdl
  \qquad
  R_{k+1} \: \fbox \cornerdr = \fbox \cornerur
  \qquad
  R_k \: \fbox \cornerul = \fbox \cornerur
  \qquad
  R_{k+1} \: \fbox \cornerul = \fbox \cornerdl
  \qquad
  \directions{k}{k+1}
\]
\end{remq}

\subsection{Plain invertibility}\label{subsec:plain_invertibility}

\begin{defn}\label{def:plain_invertibility}
We say that a cell $A \in \CC_n$ is \emph{invertible} if $\psi_1 \ldots \psi_{n-1} A$ is $R_1$-invertible.
\end{defn}

This Section is devoted to establishing the link between $R_i$-invertibility and (plain)
invertibility. This is achieved in Proposition \ref{prop:i_inv_implies_inv}. In order to do this, we relate in Lemmas \ref{lem:psi_and_i_inversibility} and
\ref{lem:shell_and_invertiblility} the $R_i$-invertibility of a cell $A$ with that of $\psi_j A$.

\begin{remq}
Let $\CC$ be a cubical $n$-category and $A \in (\Box \CC)_{n+1}$. Recall from \cite{H05} that for all $i \neq 1$, $\partial_i^\alpha \psi_1 \ldots \psi_n A \in \im \epsilon_1$, and therefore by Lemma \ref{lem:faces_and_inverses}, $\psi_1 \ldots \psi_n A$ is $R_1$-invertible. As a consequence, any $(n+1)$-cell in $\Box \CC$ is invertible.
\end{remq}

\begin{lem}\label{lem:psi_and_i_inversibility}
Let $\CC$ be a cubical $\omega$-category, and $A \in \CC_n$. Suppose $A$ is $R_j$-invertible for some $j \leq n$. Then :
\begin{itemize}
\item The $n$-cell $\psi_i A$ is $R_j$-invertible for any $i \neq j-1$.
\item The $n$-cell $\psi_{j-1} A$ is $R_{j-1}$-invertible
\end{itemize}
\end{lem}
\begin{proof}
Suppose first $j \neq i, i+1$. Then we have $\psi_i A \star_j \psi_i R_j A = \psi_i (A \star_j R_j A) = \psi_i \epsilon_j \partial_j^- A =  \epsilon_j \partial_j^- \psi_i A$, and also $\psi_i R_j A \star_j \psi_i  A = \psi_i (R_j A \star_j A) = \psi_i \epsilon_j \partial_j^+ A =  \epsilon_j \partial_j^+ \psi_i A$. Hence $\psi_i A$ is $R_j$-invertible, and $R_j \psi_i A = \psi_i R_j A$.

Suppose now $j = i$. Then $\psi_i A$ is a composite of $R_i$-invertible cells. As a consequence it is $R_i$-invertible. 

Suppose now $j = i+1$. Let $B$ be the following composite:

\[
\renewcommand{\arraystretch}{1.5}
  \begin{tabular}{ | c | c | c | c | c |}
  \hline			
     $\cornerdr$ & $\horiz$ & $\cornerdl$ & $\vertic$ & $\vertic$ \\
  \hline
      $\vertic$ & $\cornerdr$ &  $R_j A$ & $\cornerul$  & $\vertic$ \\
  \hline
      $\vertic$ & $\vertic$ & $\cornerur$ & $\horiz$ & $\cornerul $ \\
  \hline
\end{tabular}
\qquad
\directions{j}{j-1}
\]
The following computation shows that $B$ is the $R_{j-1}$-inverse of $\psi_{j-1} A$ (where empty squares denote thin cells whose faces are thin):

\begin{align*}
\psi_{j-1} A \star_{j-1} B & = 
\renewcommand{\arraystretch}{1.5}
  \begin{tabular}{ | c | c | c | c | c | c |}
    \hline			
      &  &  & $\cornerdr$ & $A$ & $\cornerul$ \\
  \hline			
     $\cornerdr$ & $\horiz$ & $\cornerdl$ & $\vertic$ & $\vertic$ & \\
  \hline
      $\vertic$ & $\cornerdr$ &  $R_j A$ & $\cornerul$  & $\vertic$ & \\
  \hline
      $\vertic$ & $\vertic$ & $\cornerur$ & $\horiz$ & $\cornerul$ & \\
  \hline
\end{tabular}
 = 
\renewcommand{\arraystretch}{1.5}
  \begin{tabular}{ | c | c | c | c | c | c |}
  \hline			
     $\cornerdr$ & $\horiz$ & $\cornerdl$ & $\cornerdr$ & $A$ & $\cornerul$ \\
  \hline
      $\vertic$ & $\cornerdr$ &  $R_j A$ & $\cornerul$  & $\vertic$ & \\
  \hline
      $\vertic$ & $\vertic$ & $\cornerur$ & $\horiz$ & $\cornerul$ & \\
  \hline
\end{tabular}
\qquad
\directions{j}{j-1} \\
& = 
\renewcommand{\arraystretch}{1.5}
  \begin{tabular}{ | c | c | c | c | c | c |}
  \hline			
     $\cornerdr$ & $\horiz$ & $\cornerdl$ & $\vertic$ & $\vertic$  \\
  \hline
      $\vertic$ & $\cornerdr$ &  $R_j A$ & $A$  & $\cornerul$  \\
  \hline
      $\vertic$ & $\vertic$ & $\cornerur$ & $\cornerul$ &  \\
  \hline
\end{tabular}
 = 
\renewcommand{\arraystretch}{1.5}
  \begin{tabular}{ | c | c | c | c | c | c |}
  \hline			
     $\cornerdr$ & $\horiz$ & $\cornerdl$ & $\vertic$ & $\vertic$  \\
  \hline
      $\vertic$ &  &  $\cornerur$ & $\cornerul$  &  $\vertic$  \\
  \hline
    $\vertic$ & $\cornerdr$ & $\horiz$ & $ \horiz$ &  $\cornerul$ \\
  \hline
\end{tabular}
\qquad
\directions{j}{j-1} \\
& = 
\renewcommand{\arraystretch}{1.5}
  \begin{tabular}{ | c | c |}
  \hline			
     $\vertic$ & $\vertic$  \\
  \hline
\end{tabular}
\qquad
\directions{j}{j-1} \quad = \epsilon_{j-1} \partial_{j-1}^- \psi_{j-1} A 
\end{align*}

A similar computation shows that $B \star_{j-1} \psi_{j-1} A = \epsilon_{j-1} \partial_{j-1}^+ \psi_{j-1} A$ and thus $\psi_{j-1} A$ is $R_{j-1}$-invertible.
\end{proof}

\begin{lem}\label{lem:shell_and_invertiblility}
Let $\CC$ be a cubical $\omega$-category, and $A \in \CC_n$ be an $n$-cell with an $R_j$-invertible shell for some $j \leq n$. Then:
\begin{itemize}
\item If $\psi_i A$ is $R_j$-invertible for some $i \neq j-1$, then $A$ is $R_j$-invertible. Moreover if  $R_j \psi_i A$ is thin then so is $R_j A$.
\item If $\psi_{j-1}$ A is $R_{j-1}$-invertible, then $A$ is $R_j$-invertible. Moreover if $R_{j-1} \psi_{j-1} A$ is thin then so is $R_j A$.
\end{itemize}
\end{lem}
\begin{proof}
Suppose $\psi_i A$ is $R_j$-invertible, with $i \neq j$. Recall that the following composite is equal to $A$
\[
\renewcommand{\arraystretch}{1.5}
\begin{tabular}{| c | c |}
\hline
$\epsilon_{i+1} \partial_i^- A$ &
$\Gamma_i^+ \partial_{i+1}^+ A$ \\ \hline
\multicolumn{2}{| c |}{$\psi_i A$} \\ \hline
$\Gamma_i^- \partial_{i+1}^- A$ & 
$\epsilon_{i+1} \partial_i^+ A$ \\ \hline
\end{tabular}
\]

Using the string notation for thin cells, this composite can be represented as follows:
\[
\renewcommand{\arraystretch}{1.5}
\begin{tabular}{| c | c |}
\hline
$\vertic$ &
$\cornerdr$ \\ \hline
\multicolumn{2}{| c |}{$\psi_i A$} \\ \hline
$\cornerul$ & 
$\vertic$ \\ \hline
\end{tabular}
\qquad
\directions{i+1}{i}
\]
This notation is ambiguous, since it does not specify which factorisations of $\partial_i^\alpha \psi_i A$ are used. However, we use the convention that in any diagram of this form, the standard factorisations $\partial_i^- \psi_i A = \partial_i^- A \star_i \partial_{i+1}^+ A$ and $\partial_i^+ \psi_i A =  \partial_{i+1}^- A \star_i \partial_{i}^+ A$ are used.

Since $A$ has an $R_j$-invertible shell, by Lemma \ref{lem:closure_i_invertibility}, every cell in this composite is $R_j$-invertible, and $A$ is $R_j$-invertible. Moreover if $R_j \psi_i A$ is thin, then the explicit formulas from Lemma \ref{lem:closure_i_invertibility} prove that $R_j A$ is thin.

Suppose now that $\psi_{j-1} A$ is $R_{j-1}$-invertible. We denote by $B$ the following composite:
\[
\renewcommand{\arraystretch}{1.5}
  \begin{tabular}{| c | c | c | c |}
  \hline			
  $\vertic$
  & 
  $\cornerdr$ & $\horiz$
  & $\horiz$ \\
  \hline
  $\vertic$
  & $\vertic$ & $\cornerdr$ & $\cornerdl$ \\
  \hline
  $\vertic$
    & \multicolumn{2}{| c |}{$R_{j-1} \psi_{j-1} A$} & $\vertic$ \\
  \hline
  $\cornerur$
    & $\cornerul$ & $\vertic$ & $\vertic$ \\
 \hline
  $\horiz$
    & $\horiz$ & $\cornerul$ & $\vertic$ \\
  \hline
\end{tabular}
\qquad
\directions{j}{j-1}
\]

We are going to show that $B$ is the $R_j$-inverse of $A$. Notice that if $R_{j-1} \psi_{j-1} A$ is thin, then $B$ is thin, using Lemma \ref{lem:closure_i_invertibility}. Let us evaluate the composite $A \star_j B$:
\begin{align*}
\renewcommand{\arraystretch}{1.5}
  \begin{tabular}{| c | c | c | c | c |}
  \hline	
  $\vertic$
  &		
  $\vertic$
  & 
  $\cornerdr$ & $\horiz$
  & $\horiz$ \\
  \hline
  $ \vertic$ & $ \vertic$
  & $\vertic$ & $\cornerdr$ & $\cornerdl$ \\
  \hline
   $ \vertic $ & $\vertic$
    & \multicolumn{2}{| c |}{$R_{j-1} \psi_{j-1} A$} & $\vertic$ \\
  \hline
  $\vertic$
  &
  $\cornerur$
    & $\cornerul$ & $\vertic$ & $\vertic$ \\
  \hline
  $A$
  &
  $\horiz$
    & $\horiz$ & $ \cornerul$ & $\vertic$ \\
  \hline
\end{tabular} & =
\renewcommand{\arraystretch}{1.5}
  \begin{tabular}{ | c | c | c | c |}
  \hline	
  \multirow{3}{*}{}
  & 
  $\cornerdr$ & $\horiz$
  & $\horiz$ \\
  \hline
  & $\vertic$ & $\cornerdr$ & $\cornerdl$ \\
  \hline
  & \multicolumn{2}{| c |}{$R_{j-1} \psi_{j-1} A$} & $\vertic$ \\
  \hline
  $\cornerdr$
    & $\cornerul$ & $\vertic$ & $\vertic$ \\
  \hline
  $A$
    & $\horiz$ & $\cornerul$ & $\vertic$ \\
  \hline
\end{tabular}
=
\renewcommand{\arraystretch}{1.5}
  \begin{tabular}{ | c | c | c | c |}
  \hline	
  \multirow{3}{*}{}
  & 
  $\cornerdr$ & $\horiz$ 
  & $\horiz$ \\
  \hline
  & $\vertic $ & $\cornerdr$ & $\cornerdl$ \\
  \hline
  & \multicolumn{2}{| c |}{$R_{j-1} \psi_{j-1} A$} & $\vertic$ \\
  \hline
   $\cornerdr$ & $A$ & $\cornerul$ & $\vertic$ \\
  \hline
   $\cornerul$ &  $\vertic$ &  & $\vertic$ \\
  \hline
\end{tabular}
\qquad
\directions{j}{j-1} \\ &
=
\renewcommand{\arraystretch}{1.5}
  \begin{tabular}{ | c | c | c |}
  \hline	
  $\cornerdr$ & $\horiz$
  & $\horiz$ \\
  \hline
$\vertic$ & $\cornerdr$ & $\cornerdl$ \\
  \hline
 \multicolumn{2}{| c |}{$R_{j-1} \psi_{j-1} A$} & $\vertic$ \\
  \hline
     \multicolumn{2}{| c |}{$\psi_{j-1} A$} & $\vertic$ \\
  \hline
   $\cornerul$ &  $\vertic$  & $\vertic$ \\
  \hline
\end{tabular}
=
\renewcommand{\arraystretch}{1.5}
  \begin{tabular}{ | c | c | c |}
  \hline	
  $\cornerdr$ & $\horiz$
  & $\horiz$ \\
  \hline
$\vertic$ & $\cornerdr$ & $\cornerdl$ \\
  \hline
 \multicolumn{2}{| c |}{$R_{j-1} \psi_{j-1} A$} & $\vertic$ \\
  \hline
     \multicolumn{2}{| c |}{$\psi_{j-1} A$} & $\vertic$ \\
  \hline
   $\cornerul$ &  $\vertic$  & $\vertic$ \\
  \hline
\end{tabular}
\qquad
\directions{j}{j-1} \\ & = 
\renewcommand{\arraystretch}{1.5}
  \begin{tabular}{ | c | c | c |}
  \hline	
  $\cornerdr$ & $\horiz$
  & $\horiz$ \\
  \hline
$\vertic$ & $\cornerdr$ & $\cornerdl$ \\
  \hline
   $\cornerul$ &  $\vertic$  & $\vertic$ \\
  \hline
\end{tabular}
\qquad
\directions{j}{j-1} \\ &
= \epsilon_{j} \partial_j^- A
\end{align*}

The evaluation of $B \star_j A$ is similar.
\end{proof}

\begin{prop}\label{prop:i_inv_implies_inv}
Let $\CC$ be a cubical $\omega$-category, $A \in \CC_n$ and $1 \leq j \leq n$. A cell $A \in \CC_n$ is $R_j$-invertible if and only if $A$ is invertible and has an $R_j$-invertible shell. Moreover if $A$ is thin, then so is its $R_j$-inverse.
\end{prop}
\begin{proof}
Suppose first that $A$ is $R_j$-invertible. Then its shell is $R_j$-invertible, and for all $i \geq j$, $\psi_i \ldots \psi_{n-1} A$ is $R_j$-invertible. Repeated applications of Lemma \ref{lem:Ti_inverses_in_shells} show that $\psi_{j} \ldots \psi_{n-1} A$ is $R_j$-invertible. As a result (still by   Lemma \ref{lem:Ti_inverses_in_shells}), $\psi_{j-1} \ldots \psi_{n-1} A$ is $R_{j-1}$-invertible. Inductively we show that for any $i \leq j$, $\psi_i \ldots \psi_{n-1} A$ is $R_i$-invertible. Finally, we get that $\psi_1 \ldots \psi_{n-1} A$ is $R_1$-invertible, in other words that $A$ is invertible.

Suppose now that $A$ is invertible and has an $R_j$-invertible shell. By multiple applications of Lemma \ref{lem:psi_and_i_inversibility}, we get that $\psi_k \ldots \psi_{n-1} A$ has an $R_j$-invertible shell, for $k \geq j$, and an $R_k$-invertible one for $k \leq j$. Applying Lemma \ref{lem:shell_and_invertiblility} multiple times, we get that for all $k \leq j$, $\psi_k \ldots \psi_{n-1} A$ is $R_k$-invertible, and finally that for all $k \geq j$, $\psi_k \ldots \psi_{n-1} A$ is $R_j$-invertible. In particular for $k = n$, $A$ is $R_j$-invertible. 

Finally if $A$ is thin, then $\psi_1 \ldots \psi_{n-1} A \in \im \epsilon_1$ and so $R_1 \psi_1 \ldots \psi_{n-1} A = \psi_1 \ldots \psi_{n-1} A $ is thin. Multiple applications of Lemma \ref{lem:shell_and_invertiblility} imply that $R_j A$ is thin. 
\end{proof}

Finally, the following Lemma will be useful in Proposition \ref{prop:carac_Ti_invertibility_inv}:

\begin{lem}\label{lem:thin_and_composite_invertible}
The composite of two invertible cells is also invertible.
\end{lem}
\begin{proof}
Let $1 \leq i \leq n$, and let $E_i$ be the set of all cells $A \in \CC_n$ such that  $\psi_1 \ldots \psi_{i-1} A$ is $R_1$-invertible. Note first that $E_i$ contains all $R_i$-invertible cells by Lemma \ref{lem:psi_and_i_inversibility} and that $E_n$  is the set of all invertible cells.  We are going to show by induction on $i$ that $E_i$ is closed under composition, for $1 \leq i \leq n$.

For $i  = 1$, $E_1$ is the set of all $R_1$-invertible cells, which is closed under composition by Lemma \ref{lem:closure_i_invertibility}. Suppose now $i > 1$. Take $A,B \in E_i$. We have:
\[
\psi_{i-1} (A \star_j B) = 
\begin{cases}
\psi_{i-1} A \star_j \psi_{i-1} B & j \neq i,i-1 \\
(\psi_{i-1} A \star_i \epsilon_{i-1} \partial_{i}^+ B) \star_{i-1} (\epsilon_{i-1} \partial_{i}^- A \star_i \psi_{i-1} B) & j = i-1 \\
(\epsilon_{i-1} \partial_{i-1}^- A \star_i \psi_{i-1} B) \star_{i-1} (\psi_{i-1} A \star_i \epsilon_{i-1} \partial_{i-1}^+ B) & j = i
\end{cases}
\]

Note that:
\begin{itemize}
\item Since $\psi_1 \ldots \psi_{i-1} A$ and $\psi_1 \ldots \psi_{i-1} B$ are $R_1$-invertible, $\psi_{i-1} A$ and $\psi_{i-1} B$ are in $E_{i-1}$.
\item The cells $\epsilon_{i-1} \partial_k^\alpha A$ and $\epsilon_{i-1} \partial_k^\alpha B$ are $R_{i-1}$-invertible by Lemma \ref{lem:closure_i_invertibility}, and therefore are in $E_{i-1}$.
\end{itemize}
By induction hypothesis, $E_{i-1}$ is closed under composition, and therefore $\psi_{i-1} (A \star_j B)$ is in $E_i$, so $\psi_1 \ldots \psi_{i-1} (A \star_j B)$ is $R_1$-invertible, and so $A \star_j B$ is in $E_i$, which is therefore close under composition.
\end{proof}

\subsection{\texorpdfstring{$T_i$}{Ti}-invertiblility}
\label{subsec:T_i_invertibility}

The notion of $T_i$-invertibility is closely related to that of $R_i$-invertibility, as we show in Lemma \ref{lem:carac_Ti_invertibility_Ri}. Consequently, a number of results from the previous Section have analogues in terms of $T_i$-invertibility. In particular, the characterisation of $T_i$-invertibility in terms of invertibility given in Proposition \ref{prop:carac_Ti_invertibility_inv} is the direct analogue of Proposition \ref{prop:i_inv_implies_inv}.

\begin{defn}
Let $\CC$ be a cubical $\omega$-category, and $i < n$ be integers. Let $A,B$ be cells in $\CC_n$ such that $\partial_i^\alpha A = \partial_{i+1}^\alpha B$ and $\partial_{i+1}^\alpha A = \partial_i^\alpha B$, for $\alpha = \pm$. If the following two equations are verified, we say that $A$ is \emph{$T_i$-invertible}, and that $B$ is the \emph{$T_i$-inverse} of $A$, and we denote $B$ by $T_i A$:

\begin{equation}\label{eq:axiom_inv}
  \renewcommand{\arraystretch}{1.5}
  \begin{tabular}{ | c | c | }
  \hline			
    $\cornerdr$ & $B$ \\
  \hline
  $A$ & $\cornerul$ \\
  \hline  
\end{tabular}
=
  \begin{tabular}{ | c | c | }
  \hline			
    $\cornerul$ &  $\cornerdr$ \\
  \hline  
\end{tabular}
\qquad
\directions{i}{i+1}
\end{equation} 
\begin{equation}\label{eq:axiom_inv_bis}
  \renewcommand{\arraystretch}{1.5}
  \begin{tabular}{ | c | c | }
  \hline			
    $\cornerdr$ & $A$ \\
  \hline
  $B$ & $\cornerul$ \\
  \hline  
\end{tabular}
=
  \begin{tabular}{ | c | c | }
  \hline			
    $\cornerul$ 
   & $\cornerdr$ \\
  \hline  
\end{tabular}
\qquad
\directions{i}{i+1}
\end{equation} 

In particular, we say that $A \in \CC_{n+1}$ has a $T_i$-invertible shell if $\bm \partial A$ is $T_i$-invertible in $\Box_n \CC$.
\end{defn}

\begin{remq}
Note that $T_i A$ is uniquely defined. Indeed, if $B$ and $C$ are both $T_i$-inverses of $A$, then evaluating the following square in two different ways shows that $B = C$:
\[
  \renewcommand{\arraystretch}{1.5}
  B =
    \begin{tabular}{  | c | c | }
  \hline			
   $\cornerdr$ & $B$ \\
  \hline
   $\cornerul$ & $\vertic$ \\
  \hline  
  $\cornerdr$ & $\cornerul$ \\ \hline
\end{tabular}
=
  \begin{tabular}{ | c | c | c | }
  \hline			
  &  $\cornerdr$ & $B$ \\
  \hline
  $\cornerdr$ & $A$ & $\cornerul$ \\
  \hline  
  $C$ & $\cornerul$ & \\ \hline
\end{tabular}
=
  \begin{tabular}{ | c | c | c | }
  \hline
  $\cornerdr$ & $\cornerul$ & $\cornerdr$ \\
  \hline  
  $C$ & $\horiz$ & $\cornerul$ \\ \hline
\end{tabular}
=
C
\qquad 
\directions{i}{i+1}
\]
\end{remq}

The relationship between $T_i$ and $R_i$-invertibility is given by the following Lemma.
\begin{lem}\label{lem:carac_Ti_invertibility_Ri}
Let $\CC$ be a cubical $\omega$-category, and $A \in \CC_n$ be an $n$-cell, with $n \geq 2$. Then $A$ is $T_i$-invertible (with $i < n$) if and only if $\psi_i A$ is $R_i$-invertible, and we have the equalities:
\begin{multicols}{2}
\begin{equation}
R_i \psi_i A = \psi_i T_i A
\end{equation}

\begin{equation} \label{eq:Ti_using_psi_i}
T_iA =   
  \renewcommand{\arraystretch}{1.5}
\begin{tabular}{ | c | c | }
  \hline			
    $\vertic$ & $ \cornerdr $\\
  \hline
   \multicolumn{2}{| c |}{$R_{i} \psi_i A$} \\
  \hline  
    $\cornerul$ & $\vertic$ \\
  \hline
\end{tabular}
\qquad
\directions{i+1}{i}
\end{equation}
\end{multicols}

In particular, if $A$ is thin, then so is $T_i A$.
\end{lem}
\begin{proof}
Suppose first that $A$ is $T_i$-invertible: then the composite $\psi_i T_i A \star_i \psi_i A$ is equal to the following:
\[
  \renewcommand{\arraystretch}{1.5}
\begin{tabular}{| c | c | c | c |}
\hline
& 
$\cornerdr$ & 
$T_i A$ &
$\cornerul$ \\ \hline
$\cornerdr$ &
$A$ &
$\cornerul$ & \\ \hline
\end{tabular}
\qquad 
\directions{i}{i+1}
\]
Using \eqref{eq:axiom_inv}, we show that this composite is equal to $\epsilon_{i} \partial_{i}^+ \psi_i A$. We prove in the same way (using \eqref{eq:axiom_inv_bis}), that $\psi_i A \star_i \psi_i T_i A = \epsilon_i \partial_{i}^- \psi_i A$, which shows that $\psi_i T_i A$ is the $R_i$-inverse of $\psi_i A$.

Suppose now that $\psi_i A$ is $T_i$-invertible. Then we have:

\begin{align*}
  \renewcommand{\arraystretch}{1.5}
  \begin{tabular}{ | c | c | c | }
  \hline			
     & $\vertic$ & $\cornerdr$ \\
  \hline
    & \multicolumn{2}{| c |}{$R_i \psi_i A$} \\
  \hline
  $\cornerdr$ & $\cornerul$ & $\vertic$ \\
  \hline  
  $A$ & $\horiz$ & $\cornerul$ \\
  \hline
\end{tabular}
& =   \renewcommand{\arraystretch}{1.5}
  \begin{tabular}{ | c | c | c | c | }
  \hline			
     & & $\vertic$ & $\cornerdr$ \\
  \hline
    & & \multicolumn{2}{| c |}{$R_i \psi_i A$} \\
  \hline
  & $\cornerdr$ & $\cornerul$ & $\vertic$ \\
  \hline  
 $\cornerdr$ & $A$ & $\horiz$ & $\cornerul$ \\
  \hline
  $\cornerul$ & $\vertic$ & & \\ \hline
\end{tabular} 
 =   \renewcommand{\arraystretch}{1.5}
  \begin{tabular}{ | c | c | }
  \hline			
       $\vertic$ & $\cornerdr$ \\
  \hline
      \multicolumn{2}{| c |}{$R_i \psi_i A$} \\
  \hline  
 \multicolumn{2}{| c |}{$\psi_i A$} \\
  \hline
  $\cornerul$ & $\vertic$   \\ \hline
\end{tabular} 
 = \renewcommand{\arraystretch}{1.5}
  \begin{tabular}{ | c | c | }
  \hline			
       $\cornerul$ & $\cornerdr$ \\
\hline
\end{tabular} 
\qquad
\directions{i+1}{i}
\\
\end{align*}
Lastly, if $A$ is thin, then $\psi_i A$ is also thin, and by Proposition \ref{prop:i_inv_implies_inv} $R_i \psi_i A$ is too. Equation \eqref{eq:Ti_using_psi_i} finally shows that $T_i A$ is thin.
\end{proof}

\begin{lem}\label{lem:Ti_inverses_in_shells}
Let $\CC$ be a cubical $n$-category. Let $1 \leq i < n$ and $A \in \Box \CC$. Then $A$ is $T_j$-invertible if and only if for all $i \neq j,j+1$, $A_i^\alpha$ is $T_{j_i}$-invertible, and:
\begin{equation}\label{eq:Ti_and_face}
\partial_i^\alpha T_j A = 
\begin{cases}
T_{j_i} \partial_i^\alpha A & i \neq j,j+1 \\
\partial_{j+1}^\alpha A & i = j, \\
\partial_j^\alpha A & i = j+1, \\
\end{cases}
\end{equation}

In particular, if $\CC$ is a cubical $\omega$-category, and a cell $A \in \CC_n$  has a $T_i$-invertible shell, then $\partial_j^\alpha A$ is $T_{i_j}$-invertible for any $j \neq i,i+1$.
\end{lem}
\begin{proof}
Suppose first that $A \in \Box \CC$ is $T_j$-invertible, and let $i \neq j,j+1$. Then we have:
\begin{align*}
\partial_i^\alpha T_j A & =  
  \renewcommand{\arraystretch}{1.5}
\begin{tabular}{ | c | c | }
  \hline			
    $\partial_i^\alpha \epsilon_j \partial_{j+1}^- A$ & $ \partial_i^\alpha \Gamma_j^+ \partial_j^+ A $\\
  \hline
   \multicolumn{2}{| c |}{$\partial_i^\alpha R_{j} \psi_j A$} \\
  \hline  
    $\partial_i^\alpha \Gamma_j^- \partial_j^- A$ & $\partial_j^\alpha \epsilon_j \partial_{j+1}^+ A$ \\
  \hline
\end{tabular}
\qquad
\directions{(j+1)_i}{j_i} \\
& =  
  \renewcommand{\arraystretch}{1.5}
\begin{tabular}{ | c | c | }
  \hline			
    $\epsilon_{j_i} \partial_{(j+1)_i}^- \partial_i^\alpha A$ & $ \Gamma_{j_i}^+ \partial_{j_i}^+ \partial_i^\alpha A $\\
  \hline
   \multicolumn{2}{| c |}{$R_{j_i} \psi_j \partial_i^\alpha A$} \\
  \hline  
    $\Gamma_{j_i}^- \partial_{j_i}^- \partial_i^\alpha A$ & $\epsilon_{j_i} \partial_{(j+1)_i}^+ \partial_i^\alpha A$ \\
  \hline
\end{tabular}
\qquad
\directions{(j+1)_i}{j_i} \\
& = T_{j_i} \partial_i^\alpha A
\end{align*}

For $i = j$, we have:

\begin{align*}
\partial_{i}^- T_i A & = 
\partial_i^- \epsilon_i \partial_{i+1}^- A \star_{i} \partial_i^- \Gamma_i^+ \partial_i^+ A &
\partial_{i}^+ T_i A & = 
\partial_i^- \Gamma_i^- \partial_i^- A \star_{i} \partial_i^- \epsilon_i \partial_{i+1}^+  A \\
& = \partial_{i+1}^- A \star_{i} \epsilon_i \partial_i^- \partial_i^+ A 
&
& = \epsilon_i \partial_i^+ \partial_i^- A \star_{i} \partial_{i+1}^+ A \\
& = \partial_{i+1}^- A
&
& = \partial_{i+1}^+ A
\end{align*}

Finally for $i = j+1$:
\begin{align*}
\partial_{i+1}^- T_i A & = \partial_{i+1}^- \epsilon_i \partial_{i+1}^- A \star_i \partial_{i+1}^- R_i \psi_i A \star_i \partial_{i+1}^- \Gamma_i^- \partial_i^- A \\
& = \epsilon_i \partial_i^- \partial_{i+1}^- A \star_i R_i \epsilon_i \partial_i^\alpha \partial_{i+1}^\alpha A \star_i \partial_i^- A \\
& = \epsilon_i \partial_i^\alpha \partial_{i+1}^\alpha A \star_i \partial_i^- A  = \partial_i^- A
\end{align*}
\begin{align*}
\partial_{i+1}^+ T_i A & = \partial_{i+1}^+  \Gamma_i^+ \partial_i^+ A \star_{i} \partial_{i+1}^+  R_i \psi_i A \star_i \partial_{i+1}^+  \epsilon_i \partial_{i+1}^+ A
\end{align*}

Reciprocally suppose that for all $i \neq j,j+1$, $A_j^\alpha$ is $T_{j_i}$-invertible. Let $B_i^\alpha = T_{j_i} A_i^\alpha$ if $i \neq j,j+1$, $B_j^\alpha = A_{j+1}^\alpha$ and $B_{j+1}^\alpha = A_j^\alpha$: this is an element of $\Box \CC$, and we verify that it is the $T_i$-inverse of $A$.
\end{proof}

\begin{prop}\label{prop:carac_Ti_invertibility_inv}
Let $\CC$ be a cubical $\omega$-category, and $A \in \CC_n$, with $n \geq 2$. Then $A$ is $T_i$-invertible if and only if $A$ is invertible and has a $T_i$-invertible shell.
\end{prop}
\begin{proof}
Suppose $A$ is $T_i$-invertible. Then $\psi_i A$ is $R_i$-invertible, and therefore it is invertible. Recall from \cite{ABS02} that $A$ is equal to the following composite:
\[
  \renewcommand{\arraystretch}{1.5}
\begin{tabular}{ | c | c | }
  \hline			
    $\vertic$ & $ \cornerdr $\\
  \hline
   \multicolumn{2}{| c |}{$\psi_i A$} \\
  \hline  
    $\cornerul$ & $\vertic$ \\
  \hline
\end{tabular}
\qquad
\directions{i+1}{i}
\]
All the cells in this composite are invertible, and invertible cells are closed under composition (Lemma \ref{lem:thin_and_composite_invertible}), therefore $A$ is invertible. Moreover since $\psi_i A$ is $R_i$-invertible, it has an $R_i$-invertible shell. In particular, for $j \neq i,i+1$, we have that $\partial_j^\alpha A = \partial_j^\alpha \psi_{i} A = \psi_{i_j} \partial_j^\alpha A$ is $R_{i_j}$-invertible. By Lemma \ref{lem:carac_Ti_invertibility_Ri}, $\partial_j^\alpha A$ is $T_{i_j}$-invertible, so finally $A$ has a $T_j$-invertible shell.

Suppose now that $A$ is invertible and has a $T_i$-invertible shell. By application of Lemma \ref{lem:carac_Ti_invertibility_Ri} in $\Box \CC_n$, $\psi_i A$ is invertible and has an $R_i$-invertible shell, so $\psi_i A$ is $R_i$-invertible, and $A$ is $T_i$-invertible by Lemma \ref{lem:carac_Ti_invertibility_Ri}.
\end{proof}

\begin{prop}\label{prop:Ti_inv_and_other}
Let $\CC$ be a cubical $\omega$-category.
\begin{itemize}
\item Let $A \in \CC_n$. For all $1 \leq j \leq n+1$, $\epsilon_j A$ is $T_j$ and $T_{j-1}$-invertible and:
\begin{equation}\label{eq:Ti_and_epsilon1}
T_j \epsilon_j A = \epsilon_{j+1} A 
\qquad
T_{j-1} \epsilon_j A = \epsilon_{j-1} A
\end{equation}
Moreover if $A$ is $T_i$-invertible (for $i \neq j-1$), then $\epsilon_j A$ is $T_{i^j}$-invertible, and:
\begin{equation}\label{eq:Ti_and_epsilon2}
T_{i^j} \epsilon_j A = \epsilon_j T_i A
\end{equation}

\item Let $A \in \CC_n$. For all $1 \leq j \leq n$, $\Gamma_j^\alpha A$ is $T_j$-invertible, and 
\begin{equation}\label{eq:Ti_and_Gamma1}
T_j \Gamma_j^\alpha A = \Gamma_j^\alpha A
\end{equation}
Moreover, if $A$ is $T_i$-invertible, then $\Gamma_j^\alpha A$ is $T_{i^j}$-invertible, and:
\begin{equation}\label{eq:Ti_and_Gamma2}
T_{i^j} \Gamma_j^\alpha A = \Gamma_j^\alpha T_i A
\end{equation}
Finally if $A$ is $T_i$-invertible, then $\Gamma_{i+1}^\alpha A$ (resp. $\Gamma_i^\alpha A$) is $T_{i}$-invertible (resp. $T_{i+1}$-invertible) and $\Gamma_i^\alpha T_i A$ (resp. $\Gamma_{i+1}^\alpha T_i A$) is $T_{i+1}$-invertible (resp. $T_i$-invertible), and:
\begin{equation}\label{eq:Ti_and_star}
T_{i+1} \Gamma_i^\alpha T_i A = T_i \Gamma_{i+1}^\alpha A
\qquad
T_i \Gamma_{i+1}^\alpha T_i A = T_{i+1} \Gamma_i^\alpha A
\end{equation}
\item Let $A,B \in \CC_n$. If $A$ and $B$ are $T_i$-invertible, then $A \star_j B$ is $T_i$-invertible, and:
\begin{equation}
T_i (A \star_j B) = 
\begin{cases}
(T_i A) \star_{i+1} (T_i B) & j = i,  \\
(T_i A) \star_{i} (T_i B) & j = i+1, \\
(T_i A) \star_j (T_i B) & \text{otherwise.} \\
\end{cases}
\end{equation}
\end{itemize}
\end{prop}
\begin{proof}
For the first seven equations, notice that both sides of the equations are thin by Lemma \ref{lem:carac_Ti_invertibility_Ri}, and therefore by Theorem \ref{thm:thin_cells}, it is enough to check that their shells are equal.

For the last one, we return to the definition of $T_i$-invertibility.
\end{proof}

\section{Relationship of cubical \texorpdfstring{$(\omega,p)$}{(omega,p)}-categories with other structures}
\label{sec:relationship_omega_p}

In Section \ref{subsec:cubical_omega_p_categories}, we collect the results of Section \ref{sec:invertibility} to give a series of equivalent characterisation of the invertibility in a cubical $\omega$-category of all cells of dimension $n$ (Proposition \ref{prop:caracterisation_omega_p}). From that we then deduce the equivalence between globular and cubical $(\omega,p)$-categories (Theorem \ref{thm:equiv_glob_cub}).

In Section \ref{subsec:ADC}, we generalise the adjunctions between globular $\omega$-groupoids and chain complexes and the one between globular $\omega$-categories and ADCs from \cite{S04}. To do so we introduce the notion of $(\omega,p)$-ADCs, such that $(\omega,\omega)$-ADCs are just ADCs, and $(\omega,0)$-ADCs coincide with augmented chain complexes.
\subsection{Cubical and globular \texorpdfstring{$(\omega,p)$}{(omega,p)}-categories}
\label{subsec:cubical_omega_p_categories}

In this Section we start by defining the notion of cubical $(\omega,p)$-categories. In Proposition \ref{prop:caracterisation_omega_p} we give various equivalent characterisations of those using the result from Section \ref{sec:invertibility}. As a result, we show in Theorem \ref{thm:equiv_glob_cub} that the equivalence between globular and cubical $\omega$-category induces equivalences between globular and cubical $(\omega,p)$-categories. Finally in Corollary \ref{cor:carac_omega_1} we give a simple characterisation of the notions of cubical $(\omega,0)$ and $(\omega,1)$-categories.

\begin{defn}
Let $\CC$ be a cubical $\omega$-category, and $p$ a natural number. We say that $\CC$ is a \emph{cubical $(\omega,p)$-category} if any $n$-cell is invertible, for $n > p$. We denote by $\CCat{(\omega,p)}$ the full subcategory of $\CCat{\omega}$ spanned by cubical $(\omega,p)$-categories.
\end{defn}

\begin{prop}\label{prop:caracterisation_omega_p}
Let $\CC$ be a cubical $\omega$-category, and fix $n > 0$. The following five properties are equivalent:
\begin{enumerate}
\item Any $n$-cell in $\CC_n$ is invertible.
\item For all $1 \leq i \leq n$, any $n$-cell in $\CC_n$ with an $R_i$-invertible shell is $R_i$-invertible.
\item Any $n$-cell in $\CC_n$ with an $R_1$-invertible shell is $R_1$-invertible.
\item Any $n$-cell $A \in \CC_n$ such that for all $j \neq 1$, $\partial_j^\alpha A \in \im \epsilon_1$ is $R_1$-invertible.
\item Any $n$-cell in $\Phi_n(\CC_n)$ is $R_1$-invertible.
\end{enumerate}
Moreover, if $n > 1$, then all the previous properties are also equivalent to the following:
\begin{enumerate}
\item[6.] For all $1 \leq i < n$, any $n$-cell in $\CC_n$ with a $T_i$-invertible shell is $T_i$-invertible
\item[7.] Any $n$-cell in $\CC_n$ with a $T_1$-invertible shell is $T_1$-invertible.
\end{enumerate}
\end{prop}
\begin{proof}
\emph{(1) $\Rightarrow$ (2)} holds by Proposition \ref{prop:i_inv_implies_inv}, \emph{(2) $\Rightarrow$ (3)} is clear, and \emph{(3) $\Rightarrow$ (4)} holds because if $A \in \CC_n$ satisfies $\partial_j^\alpha A \in \im \epsilon_1$, then its shell is $R_1$-invertible. Also, \emph{(4) $\Rightarrow$ (5)} holds because for any $A \in \Phi_n(\CC_n)$, $\partial_j^\alpha A \in \im \epsilon_1$ for all $j \neq 1$. Let us finally show that \emph{(5) $\Rightarrow$ (1)}. From Lemmas \ref{lem:psi_and_i_inversibility} and \ref{lem:shell_and_invertiblility}, for any $i < n$, a cell $A \in \CC_n$ with an $R_1$-invertible shell is $R_1$-invertible if and only if $\psi_i A$. Iterating this result, we get that for all $A \in \CC_n$ $\psi_1 \ldots \psi_{n-1} A$ is $R_1$-invertible if and only if $\Phi \psi_1 \ldots \psi_{n-1} A$ is. Since $\Phi \psi_i = \Phi$ for all $i < n$, $A$ is invertible if and only if $\Phi A$ is $R_1$-invertible.

Suppose now $n > 1$. Then \emph{(1) $\Rightarrow$ (6)} by Proposition \ref{prop:carac_Ti_invertibility_inv}, and clearly \emph{(6) $\Rightarrow$ (7)}. Suppose now that any $n$-cell with a $T_1$-invertible shell is $T_1$-invertible, and let us show that \emph{(4)} holds. Let $A \in \CC_n$  such that $\partial_j^\alpha A \in \im \epsilon_1$ for all $j \neq 1$ is $R_1$-invertible: then $A$ has a $T_1$-invertible shell, and is therefore $T_1$-invertible by hypothesis. As a consequence $A$ is invertible, and since it has an $R_1$-invertible shell, it is $R_1$-invertible.
\end{proof}

\begin{thm}\label{thm:equiv_glob_cub}
The functors $\lambda$ and $\gamma$ restrict to an equivalence of categories:
\[
\begin{tikzpicture}
\matrix (m) [matrix of math nodes, 
			nodes in empty cells,
			column sep = 3cm, 
			row sep = .7cm] 
{
\Cat{(\omega,p)} & \CCat{(\omega,p)} \\
};

\draw[-to] (m-1-2) to [bend right=20] node (A) [above] {$\lambda$} (m-1-1);
\draw[-to] (m-1-1) to [bend right = 20] node (C) [below] {$\gamma$} (m-1-2);

\path (A) to node {$\cong$} (C);

\end{tikzpicture}
\]
\end{thm}
\begin{proof}
Let $\CC$ be a cubical $(\omega,p)$-category. The globular $\omega$-category  $\gamma \CC$ is a globular $(\omega,p)$-category if and only if, for all $n > p$, every cell in $\Phi_n(\CC_n)$ is $R_1$-invertible. By Proposition \ref{prop:caracterisation_omega_p}, this is equivalent to $\CC$ being a cubical $(\omega,p)$-category. Since $\Cat{(\omega,p)}$ and $\CCat{(\omega,p)}$ are replete full sub-categories respectively of $\Cat{\omega}$ and $\CCat{\omega}$, this proves the result.
\end{proof}

\begin{cor}\label{cor:carac_omega_1}
Let $\CC$ be a cubical $\omega$-category. Then:
\begin{itemize}
\item $\CC$ is a cubical $\omega$-groupoid if and only if every $n$-cell of $\CC$ is $R_i$-invertible for all $1 \leq i \leq n$.
\item $\CC$ is a cubical $(\omega,1)$-category if and only if every $n$-cell is $T_i$-invertible, for all $1 \leq i < n$.
\end{itemize}
\end{cor}
\begin{proof}
If every $n$-cell of $\CC_n$ is $R_i$-invertible then in particular every cell of $\CC_n$ is invertible, and so $\CC$ is a cubical $\omega$-groupoid.
Reciprocally, if $\CC$ is a cubical $\omega$-groupoid, we prove by induction on $n$ that every cell is $R_i$-invertible. For $n = 1$, every $1$-cell has an $R_1$-invertible shell, and so every cell is $R_1$-invertible. Suppose now the property true for all $n$-cells. Then any cell $A \in \CC_{n+1}$ necessarily has a $R_i$-invertible shell by Lemma \ref{lem:faces_and_inverses}, and so the property holds for all $(n+1)$-cells.

The proof of the second point is similar, using the fact that any $2$-cell in a cubical $\omega$-category has a $T_1$-invertible shell.
\end{proof}

\subsection{Augmented directed complexes and \texorpdfstring{$(\omega,p)$}{(omega,p)}-categories}
\label{subsec:ADC}

From \cite{ABS02} and \cite{S04}, we have the following functors, where $\ADC$ is the category of augmented directed complexes.
\[
\begin{tikzpicture}
\matrix (m) [matrix of math nodes, 
			nodes in empty cells,
			column sep = 3cm, 
			row sep = .7cm] 
{
\ADC & \Cat{\omega} & \CCat \omega \\
};
\draw[-to] (m-1-2) to [bend right=20] node (A1) [above] {$\mathcal Z^\G$} (m-1-1);
\draw[-to] (m-1-1) to [bend right = 20] node (C1) [below] {$N^\G$} (m-1-2);
\draw[-to] (m-1-3) to [bend right=20] node (A) [above] {$\lambda$} (m-1-2);
\draw[-to] (m-1-2) to [bend right = 20] node (C) [below] {$\gamma$} (m-1-3);

\path (A) to node {$\cong$} (C);
\path (C1) to node {$\bot$} (A1); 
\end{tikzpicture}
\]
In this section we define cubical analogues to $N^\G$ and $\mathcal Z^\G$, and show that they induce an adjunction between $\ADC$ and $\CCat{\omega}$. Finally we show that all these functor can be restricted to the case of $(\omega,p)$-categories, with a suitable notion of $(\omega,p)$-$\ADC$.

\begin{defn}
An \emph{augmented chain complex}  $K$ is a sequence of abelian groups $K_n$ (for $n \geq 0$) together with maps $\d : K_{n+1} \to K_n$ for every $n \geq 0$ and a map $\e : K_0 \to \mathbb Z$ satisfying the equations:
\[
\d \circ \d = 0 \qquad \e \circ \d = 0
\]

A morphism of augmented chain complexes from $(K,\d,\e) \to (L,\d,\e)$ is a family of morphisms $f_n : K_n \to L_n$ satisfying: 
\[
\d \circ f_{n+1} = f_n \circ \d
\qquad
\e = \e \circ f_0.
\]
\end{defn}

\begin{defn}
An \emph{augmented directed chain complex} (or ADC for short) is an augmented chain complex $K$ equipped with a submonoid $K_n^*$ of $K_n$ for any $n \geq 0$.

A morphism of ADCs $K \to L$ is a morphism of augmented chain complexes $f$ satisfying $f(K_n^*) \subseteq L_n^*$. We denote by $\ADC$ the category of augmented directed chain complexes.
\end{defn}

The following is a reformulation of Steiner \cite{S04}:
\begin{prop}
Let us fix $n \geq 0$, and let $K$ the following ADC:
\[
K_k = \begin{cases}
\mathbb Z[s_k,t_k] & k < n \\
\mathbb Z[x] & k = n \\
0 & k > n
\end{cases}
\qquad
K^*_k = \begin{cases}
\mathbb N[s_k,t_k] & k < n \\
\mathbb N[x] & k = n \\
0 & k > n
\end{cases}
\qquad
\begin{cases}
\d[x] = t_{n-1} - s_{n-1} & \\
\d[s_{k+1}] = \d[t_{k+1}] = t_{k} - s_{k} & k \geq 0 \\
\e[s_0] = \e[t_0] = 1& 
\end{cases}
\]
We denote this ADC by $n$-$\tikzcircle^\ADC$. 

Equipped with morphisms $\check \s,\check \t : (n+1)\text{-}\tikzcircle^\ADC \to n\text{-}\tikzcircle^\ADC$, $\check 1 : n\text{-}\tikzcircle^\ADC \to (n+1)\text{-}\tikzcircle^\ADC$ and $\check \star_i : n\text{-}\tikzcircle^\ADC \bigsqcup_{i\text{-}\tikzcircle^\ADC} n\text{-}\tikzcircle^\ADC$, those form a co-globular $\omega$-category object in $\ADC$, and therefore they induce a functor $N^\G : \ADC \to \Cat{\omega}$ defined by $(N^\G L)_n = \ADC( n\text{-}\tikzcircle^\ADC,L)$
\end{prop}

The category $\ADC$ is equipped with a tensor product defined as follows \cite{S04}: 
\begin{defn}
Let $K$ and $L$ be ADCs. We define an object $K \otimes L$ in $\ADC$ as follows:
\begin{itemize}
\item For all $n \geq 0$, $(K \otimes L)_n = \bigoplus_{i+j=n} K_i \otimes L_j$.
\item For all $n \geq 0$, $(K \otimes L)_n^*$ is the sub-monoid of $(K \otimes L)_n$ generated by the elements of the form $x \otimes y$, with $x \in K_i^*$ and $y \in L_{n-i}^*$.
\item For all $x \in K_i$ and $y \in L_{n-i}$, $\d[x \otimes y] = \d[x]\otimes y + (-1)^i x \otimes \d[y]$.
\item For all $x \in K_0$ and $y \in L_0$, $\e[x \otimes y] = \e[x] \e[y]$.
\end{itemize}
\end{defn}

\begin{prop}
Let $\C$ be a globular $\omega$-category. Following Steiner \cite{S04}, we define an ADC $\mathcal Z^\G \C$ as follows:
\begin{itemize}
\item For all $n \in \mathbb N$, $(\mathcal Z^\G \C)_n$ is the quotient of the group $\mathbb Z [\C_n]$ by the relation $[A \bullet_k B] = [A] + [B]$.
\item For all $n  \in \mathbb N$, $(\mathcal Z^\G \C)_n^*$ is the image of $\mathbb N[\C_n]$ in $(\mathcal Z^\G \C)_n$.
\item For all $A \in \C_n$, $\d [A] = [\s(A)] - [\t(A)]$.
\item For all $A \in \C_0$, $\e [A] = 1$.
\end{itemize}
\end{prop}

\begin{prop}[\cite{S04}, Theorem 2.11]
The functor $\mathcal Z^\G$ is left-adjoint to the functor $N^\G$.

\[
\begin{tikzpicture}
\matrix (m) [matrix of math nodes, 
			nodes in empty cells,
			column sep = 3cm, 
			row sep = .7cm] 
{
\ADC & \Cat{\omega}  \\
};
\draw[-to] (m-1-2) to [bend left=20] node (A1) [below] {$N^\G$} (m-1-1);
\draw[-to] (m-1-1) to [bend left = 20] node (C1) [above] {$\mathcal Z^\G$} (m-1-2);

\path (A1) to node {$\bot$} (C1); 
\end{tikzpicture}
\]
\end{prop}

\begin{defn}
Let $\cube{n}^\ADC$ be the augmented directed complex $\mathcal Z^\G(\cube{n}^\G)$. The maps $\check \partial_i^\alpha$, $\check \epsilon_i$, $\check \Gamma_i^\alpha$ and $\check \star_i$ still induce a structure of co-cubical $\omega$-category object in $\ADC$ on the family $\cube{n}^\ADC$. Consequently, for any $K \in \ADC$ the family of sets $\ADC(\cube{n}^\ADC,K)$ is equipped with a structure of cubical $\omega$-category. This defines a functor $N^\CC : \ADC \to \CCat{\omega}$.

Let $\CC$ be a cubical $\omega$-category. We define an ADC $ \mathcal Z^\CC \CC$ as follows:
\begin{itemize}
\item For all $n \in \mathbb N$, $(\mathcal Z^\CC \CC)_n$ is the quotient of $\mathbb Z[\CC_n]$ by the relations $[A \star_k B] = [A] + [B]$ and $[\Gamma_i^\alpha A] = 0$.
\item For all $n \in \mathbb N$, $(\mathcal Z^\CC \CC)_n^*$ is the image of $\mathbb Z[\CC_n]$ is $K_n$.
\item For all $A \in \CC_n$, 
\[
\d[A] = \sum_{\substack{1\le i\le n\\ \alpha = \pm}} \alpha (-1)^i [\partial_i^\alpha A]
\]
\item For all $A \in \CC_0$, $\e[A] = 1$.
\end{itemize}
\end{defn}

\begin{prop}\label{prop:commutation_up_to_iso}
There are isomorphisms of functors:
\[
\mathcal Z^\CC \approx \mathcal Z^\G \circ \gamma 
\qquad 
N^\CC \approx \lambda \circ N^\G
\]

As a result, we have the following diagram of equivalence and adjunctions between  $\Cat{\omega}$, $\CCat{\omega}$ and $\ADC$, where both triangles involving $\mathcal Z^\CC$ and $\mathcal Z^\G$ and both triangles involving $N^\CC$ and $N^\G$ commute up to isomorphism:
\[
\begin{tikzpicture}
\matrix (m) [matrix of math nodes, 
			nodes in empty cells,
			column sep = 1cm, 
			row sep = 2.5cm] 
{
\Cat{\omega} & & \CCat{\omega}  \\
& \ADC & \\
};
\draw[-to] (m-2-2) to node (NG) [above right] {$N^\G$} (m-1-1);
\draw[transform canvas={xshift=-0.6cm}, -to] (m-1-1) to node (ZG) [below left] {$\mathcal Z^\G$} (m-2-2);
\draw[-to] (m-2-2) to node (NC) [above left] {$N^\CC$} (m-1-3);
\draw[transform canvas={xshift=0.6cm}, -to] (m-1-3) to node (ZC) [below right] {$\mathcal Z^\CC$} (m-2-2);
\draw[transform canvas={yshift=-0.2cm}, -to] (m-1-3) to node (L) [below] {$\gamma$} (m-1-1);
\draw[transform canvas={yshift=0.2cm}, -to] (m-1-1) to node (G) [above] {$\lambda$} (m-1-3);

\path (ZG) to node [left] {\rotatebox{125}{$\bot$}} (NG);
\path (ZC) to node [right] {\rotatebox{230}{$\bot$}} (NC);
\path (L) to node {$\cong$} (G);
\end{tikzpicture}
\]
\end{prop}
\begin{proof}
Let $K$ be an ADC. We have for all $n \geq 0$, using the adjunction between $N^\G$ and $\mathcal Z^\G$:
\begin{align*}
\lambda \circ N^\G (K)_n & = \Cat{\omega}(\cube{n}^\G,N^\G K) \\
& \approx \ADC(\mathcal Z^\G(\cube{n}^\G),K) \\
& = \ADC(\cube{n}^\ADC,K) \\
& = (N^\CC K)_n
\end{align*}
Moreover, because these equalities are functorial, they preserve the cubical $\omega$-category structures on the families $\lambda \circ N^\G(K)_n$ and $(N^\CC K)_n$, so finally we have the isomorphism $N^\CC \approx \lambda \circ N^\G$.

\medskip 

Let now $\CC$ be a cubical $\omega$-category.  For all $n \geq 0$, the group $\mathcal Z^\G (\gamma (\CC))_n$ is the free abelian group generated by elements $[A]$, for $A \in \im \Phi_n$, subject to the relations $[A \star_i B] = [A] + [B]$, for all $A,B \in \im \Phi_n$. Let us show that for all $n \geq 0$, $\mathcal Z^\G (\gamma (\CC))_n$ and $\mathcal Z^\CC(\CC)_n$ are isomorphic.

First, the inclusion $\im \Phi_n \to \CC_n$ gives rise to a map $\mathbb Z[\im \Phi_n] \to \mathcal Z^\CC(\CC)_n$. Moreover this map respects the relations defining $Z^\G (\gamma (\CC))_n$, so it induces a morphism $\iota: \mathcal Z^\G (\gamma (\CC))_n \to \mathcal Z^\CC(\CC)_n$.

For all $A \in \CC_n$, we have in $\mathcal Z^\CC(\CC)_n$: $[\psi_i A] = [\Gamma_i^+ \partial_{i+1}^- A] + [A] + [\Gamma_i^- \partial_{i+1}^+ A] = [A]$. By iterating this formula, we get that for all $A \in \CC_n$, $[\Phi_n(A)] = [A]$. Hence $\iota$ is surjective. Let us now show that it is injective. Using the relation $[\Phi_n(A)] = [A]$, we get that $\mathcal Z^\CC(\CC)_n$ is isomorphic to the free group generated by $[\im \Phi_n]$, subject to the relations $[\Phi_n(A \star_i B)]=[\Phi_n(A)]+[\Phi_n(B)]$ for all $A,B \in \CC_n$ and $[\Phi_n(\Gamma_i^\alpha A)]=0$, for all $A \in \CC_{n-1}$. Let us prove that these equalities already hold in $\mathcal Z^\G (\gamma (\CC))_n$.

Let $x$ be a thin cell in $\CC_n$. Then $\Phi_n(x)$ is in the image of $\epsilon_1$, and $\Phi_n(x) \star_1 \Phi_n(x )= \Phi_n(x)$, and so in $\mathcal Z^\G (\gamma (\CC))_n$: $2 \cdot [\Phi_n(x) ] = [\Phi_n(x)]$, and finally $[\Phi_n(x)] = 0$. In particular $[\Phi_n(\Gamma_i^\alpha A)]=0$ in $\mathcal Z^\G (\gamma (\CC))_n$. 
Let now $A$ and $B$ be $i$-composable $n$-cells. Following Proposition 6.8 from \cite{ABS02}, $\Phi_n(A \star_i B)$ is a composite of cells of the form $\epsilon_1^{n-m} \Phi_m D A$ and $\epsilon_1^{n-m} \Phi_m D B$, where $0 \leq m \leq n$ is an integer, and $D$ is a composite of length $m$ of faces operations. Using the fact that $\epsilon_1^{n-m} \Phi_m = \Phi_n \epsilon_1^{n-m}$, we get that $\Phi_n(A \star_i B)$ is a composite of cells $\Phi_n(x)$, where $x$ is thin, with the cells $\Phi_n(A)$ and $\Phi_n(B)$. As a consequence, we get that in  $\mathcal Z^\G (\gamma (\CC))_n$, $[\Phi_n(A \star_i B)] = k_1 [\Phi_n(A)] + k_2[\Phi_n(B)]$ for some integers $k_1$ and $k_2$. Moreover, following Section 6 of \cite{ABS02}, we verify that the cells $\Phi_nA$ and $\Phi_n B$ appear exactly once in this composition. As a result $[\Phi_n(A \star_i B)] = [\Phi_n(A)] + [\Phi_n(B)]$ in $\mathcal Z^\G (\gamma (\CC))_n$, and so $\mathcal Z^\G (\gamma (\CC))_n$ and $\mathcal Z^\CC(\CC)_n$ are isomorphic. 

Let us denote respectively by $\d^\G$ and $\d^\CC$ the boundary maps in $ \mathcal Z^\G (\gamma (\CC))$ and $\mathcal Z^\CC(\CC)_n$. For $A \in \im(\Phi_n)$, we have $\d^\G[A] = [\partial_1^- A] - [\partial_1^+ A]$, and $\d^\CC[A] = \sum_{\substack{1\le i\le n\\ \alpha = \pm}} \alpha (-1)^i [\partial_i^\alpha A]$. Since $A$ is in $\im \Phi_n$, for all $i \neq 1$, $\partial_i^\alpha A$ is thin and so $[\partial_i^\alpha A]=0$, and $\d^\CC[A] = [\partial_1^- A] - [\partial_1^+ A] = \d^\G[A]$. As a result, $\iota$ induces an isomorphism of chain complexes between $\mathcal Z^\G (\gamma (\CC))$ and $\mathcal Z^\CC(\CC)$. Finally $\mathcal Z^\G (\gamma (\CC))_n^*$ and $\mathcal Z^\CC(\CC)_n^*$ are the submonoids respectively generated by $\im \Phi_n$ and $\CC_n$ and $[A] = [\Phi_n(A)]$ in  $\mathcal Z^\CC(\CC)_n$, so  $\mathcal Z^\G (\gamma (\CC))$ and $\mathcal Z^\CC(\CC)$ are isomorphic as ADCs.
\end{proof}

\begin{defn}
Let $K$ be an ADC. We say that a cell $A \in K^*_n$ is \emph{invertible} if  $-A$ is in $K_n^*$.

We say that $K$ is an $(\omega,p)$-ADC if for any $n > p$,  $K_n = K^*_n$. We denote by $(\omega,p)$-$\ADC$ the category of $(\omega,p)$-ADCs.
\end{defn}

\begin{prop}\label{prop:adjunction_omega_p_glob}
Let $\C$ be a globular $\omega$-category, and $A \in \C_n$. If $A$ is invertible, then so is $[A]$ in $\mathcal Z^\G(\C)$, and $ [A^{-1}]= -[A]$. In particular if $\C$ is an $(\omega,p)$-category, then $\mathcal Z^\G \C$ is an $(\omega,p)$-ADC.

\medskip

Let $K$ be an ADC, and $A \in \ADC(n\text{-}\tikzcircle^\ADC,K)$. If $A[(x)] \in K_n^*$ is invertible then so is $A$ in $N^\G(K)$, and the inverse of $A$ is given by:
\[
B[x] = -A[x] 
\qquad
\begin{cases}
B[s_{n-1}] = A[t_{n-1}] & \\
B[t_{n-1}] = A[s_{n-1}] & \\
\end{cases}
\qquad
\begin{cases}
B[s_i] = A[s_i] & i < n-1 \\
B[t_i] = A[t_i] & i < n-1 \\
\end{cases}
\]
In particular if $K$ is an $(\omega,p)$-ADC then $N^\G K$ is a globular $(\omega,p)$-category.
\end{prop}
\begin{proof}
Let $\C$ be an $\omega$-category, and $A \in \C_n$. If $A$ is invertible, then there exists $B$ such that $A \bullet_n B = 1_{\s(A)}$. Notice first that $[1_{\s(A)}] + [1_{\s(A)}] =[1_{\s(A)} \bullet_n 1_{\s(A)}] = [1_{\s(A)}]$, and so $[1_{\s(A)}] = 0$. As a consequence, $[A] + [B] = [A \bullet_n B] = 0$. Since both $[A]$ and $[B]$ are in $Z^\G(\C)_n^*$, $[A]$ is invertible. If $\C$ is an $(\omega,p)$-category, then for all $n > p$, $(\mathcal Z^\G \C)_n^*$ is generated by invertible cells. Since invertible cells are closed under addition, $(\mathcal Z^\G \C)_n^*$ is actually a group. Moreover it has the same generators as $(\mathcal Z^\G \C)_n$, so the two groups are actually equal, making $\mathcal Z^\G \C$ an $(\omega,p)$-ADC.

Let now $K$ be an ADC, and $A \in \ADC(n\text{-}\tikzcircle^\ADC,K)$ such that $A[x]$ is invertible. Define $B$ as the following morphism from $n\text{-}\tikzcircle^\ADC$ to $K$:
\[
B[x] = -A[x] 
\qquad
\begin{cases}
B[s_{n-1}] = A[t_{n-1}] & \\
B[t_{n-1}] = A[s_{n-1}] & \\
\end{cases}
\qquad
\begin{cases}
B[s_i] = A[s_i] & i < n-1 \\
B[t_i] = A[t_i] & i < n-1 \\
\end{cases}
\]
Note that since $A[x]$ is invertible, $-A[x]$ is in $K_n^*$, and so $B$ is indeed a morphism of ADC. Moreover, $A$ and $B$ are $(n-1)$-composable, and $A \bullet_{n-1} B$ is given by:
\[
(A \bullet_{n-1} B) [x] = A[x]-A[x] = 0
\qquad
\begin{cases}
(A \bullet_{n-1} B) [s_{n-1}] = A[s_{n-1}] & \\
(A \bullet_{n-1} B) [t_{n-1}] = B[t_{n-1}] = A[s_{n-1}] & \\
\end{cases}
\begin{cases}
(A \bullet_{n-1} B) [s_i] = A[s_i] & \\
(A \bullet_{n-1} B)[s_i] = A[s_i] & \\
\end{cases}
\]
So $A \bullet_{n-1} B = 1_{\s(A)}$, and symmetrically $B \bullet_{n-1} A = 1_{\t(A)}$. The cell $A$ is thus invertible. In particular if $K$ is an $(\omega,p)$-ADC, then for all $n > p$ and all $A \in \ADC(n\text{-}\tikzcircle^\ADC,K)$, $A[x]$ is invertible and $A$ is invertible, and so every cell in $( N^\G K)_n$ is invertible, which means that $N^\G K$ is an $(\omega,p)$-category.
\end{proof}

Recall from \cite{S04} that $\cube{n}^\ADC_k$ is the free abelian group over the set $\cube{n}^\Set_k$ of sequences $s : \{1,\ldots,n\} \to \{(-),(\pmzerodot),(+)\}$ such that $|s^{-1}(\pmzerodot)|=k$. For any such $s$, and any $1 \leq i \leq n$ such that $s(i) \neq (\pmzerodot)$, we denote by $R_is$ the sequence  obtained by replacing $s(i)$ by $-s(i)$ in $s$. The following Proposition is the cubical analogue of the previous one.

\begin{prop}\label{prop:adjunction_omega_p_cub}
Let $\CC$ be a cubical $\omega$-category, and $A \in \CC_n$. If $A$ is $R_i$-invertible or $T_i$-invertible, then $[A]$ is invertible. In particular if $\CC$ is a cubical $(\omega,p)$-category, then $\mathcal Z^\C \CC$ is an $(\omega,p)$-ADC.

\medskip 

Let $K$ be an $ADC$, and let $A \in \ADC(\cube{n}^\ADC,K)$:
\begin{itemize}
\item If for any $0 \leq k \leq n$, and any sequence $s \in \cube{n}^\Set_k$ such that $s(i)=(\pmzerodot)$, $A[s]$ is invertible (in $K$) then $A$ is $R_i$-invertible, and $R_iA$ is given by:
\[
R_iA[s] = \begin{cases}
-A[s] & s(i) = (\pmzerodot) \\
A[R_i s] & s(i) \neq (\pmzerodot) \\
\end{cases}
\]
\item If for any $0 \leq k \leq n$, and any sequence $s \in \cube{n}^\Set_k$ such that $s(i)=s(i+1) = (\pmzerodot)$, $A[s]$ is invertible, then $A$ is $T_i$-invertible, and $T_i A$ is given by:
\[
T_iA[s] = \begin{cases}
-A[s] & s(i) = s(i+1) = (\pmzerodot) \\
A[s \circ \tau_i] & \text{otherwise.}
\end{cases}
\] 
\end{itemize}
In particular, if $K$ is an $(\omega,p)$-ADC, then $N^\CC K$ is a cubical $(\omega,p)$-category.
\end{prop}
\begin{proof}
The proof is similar to that of the previous Proposition.
\end{proof}

\begin{thm}\label{thm:triangle_adjunctions}
For all $p \in \mathbb N \cup \{ \omega \}$, the categories  $\Cat{(\omega,p)}$, $\CCat{(\omega,p)}$ and $(\omega,p)$-$\ADC$ are related by the following diagram of equivalence and adjunctions, where both triangles involving $\mathcal Z^\CC$ and $\mathcal Z^\G$ and both triangles involving $N^\CC$ and $N^\G$ commute up to isomorphism:
\[
\begin{tikzpicture}
\matrix (m) [matrix of math nodes, 
			nodes in empty cells,
			column sep = 1cm, 
			row sep = 2.5cm] 
{
\Cat{(\omega,p)} & & \CCat{(\omega,p)}  \\
& (\omega,p)\text{-}\ADC & \\
};
\draw[-to] (m-2-2) to node (NG) [above right] {$N^\G$} (m-1-1);
\draw[transform canvas={xshift=-0.6cm}, -to] (m-1-1) to node (ZG) [below left] {$\mathcal Z^\G$} (m-2-2);
\draw[-to] (m-2-2) to node (NC) [above left] {$N^\CC$} (m-1-3);
\draw[transform canvas={xshift=0.6cm}, -to] (m-1-3) to node (ZC) [below right] {$\mathcal Z^\CC$} (m-2-2);
\draw[transform canvas={yshift=-0.2cm}, -to] (m-1-3) to node (L) [below] {$\gamma$} (m-1-1);
\draw[transform canvas={yshift=0.2cm}, -to] (m-1-1) to node (G) [above] {$\lambda$} (m-1-3);

\path (ZG) to node [left] {\rotatebox{134}{$\bot$}} (NG);
\path (ZC) to node [right] {\rotatebox{226}{$\bot$}} (NC);
\path (L) to node {$\cong$} (G);
\end{tikzpicture}
\]
\end{thm}
\begin{proof}
We have already proven that the equivalence between $\Cat{\omega}$ and $\CCat{\omega}$ could be restricted to $(\omega,p)$-categories in Theorem \ref{thm:equiv_glob_cub}, and by Propositions \ref{prop:adjunction_omega_p_glob} and \ref{prop:adjunction_omega_p_cub}, so can the two adjunctions. Lastly, the commutations up to isomorphisms come from Proposition \ref{prop:commutation_up_to_iso}.
\end{proof}

\begin{remq}
In the case where $p = 0$, one would expect the previous Theorem to recover the usual adjunction between chain complexes and groupoids. However, the category of  $(\omega,0)$-ADCs is not the category of chain complexes, but that of chain complexes $K$ equipped with a distinguished sub-monoid of $K_0$.

In order to recover the adjunction between groupoids and chain complexes, one could use a variant of the notion of ADC that does not specify a distinguished submonoid of $K_0$. Then an $(\omega,0)$-ADC is indeed just a chain complex. One can check that, \emph{mutatis mutandis}, the results of this Section, and in particular Theorem \ref{thm:triangle_adjunctions}, still hold using this alternative definition.
\end{remq}

\section{Permutations in cubical \texorpdfstring{$(\omega,p)$}{(omega,p)}-categories}
\label{sec:applications}

We apply our results from the previous Section to two different directions. First we show in Section \ref{subsec:sym_cubical_categories} that the operations $T_i$ induce a partial action of the symmetric group $S_n$ on the $n$-cells of a cubical $\omega$-category. To do this, we define a general notion of $\sigma$-invertibility, where $\sigma \in S_n$. In particular when $\sigma$ is  a transposition $\tau_i$ we recover the notion of $T_i$-invertibility of Section \ref{subsec:T_i_invertibility}. In Section \ref{subsec:transfors} we define the notions of lax and oplax transfors between cubical categories. We then define what it means for a transfor to be pseudo using the notion of $\sigma$-invertibility defined previously and finally we show that the cubical $\omega$-categories of pseudo lax and oplax transfors between two cubical $\omega$-categories are isomorphic.

\subsection{Cubical \texorpdfstring{$(\omega,1)$}{(omega,1)}-categories are symmetric}
\label{subsec:sym_cubical_categories}

We start by defining a notion of $u$-invertibility, where $u$ is a word over $T_1, \ldots ,T_i$, and characterise the notion of $u$-invertibility in terms of plain invertibility, just as we have done previously for $R_i$ and $T_i$-invertibility. 

We then show how the notion of $u$-invertibility induces a notion of $\sigma$-invertibility, for $\sigma \in S_n$. The difficulty lies in the fact that, even if two words $u$ and $v$ over $T_1, \ldots ,T_i$ correspond to the same permutations, the notions of $u$ and $v$-invertibility do not necessarily coincide. We circumvent this difficulty by using a classical result about the symmetric group (see Theorem \ref{thm:representability_of_permutations}), which makes use of the notion of representative of minimal length of permutation.

Finally in Proposition \ref{prop:sigma_inv} we extend the results concerning $u$-invertibility to $\sigma$-invertibility, with $\sigma \in S_n$.

\begin{defn}
Let $n \in \mathbb N$. We write $\mathbb T_n$ the free monoid on $n-1$ elements. We denote its generators by $T_1, \ldots, T_{n-1}$, and by $\l:\mathbb T_{n} \to \mathbb N$ the morphism of monoids that sends every $T_i$ on $1$. For $u \in \mathbb T_n$, we call $\l(u)$ the \emph{length} of $u$.

Recall that $S_n$ is a quotient of $\mathbb T_n$ using the relations:

\begin{equation} \label{eq:transpo_nilpotente}
T_i T_i = 1 
\end{equation}
\begin{equation} \label{eq:braid_relation}
T_i T_{i+1} T_i = T_{i+1} T_i T_{i+1}
\end{equation}  
\begin{equation}\label{eq:comm_relation}
T_j  T_i = T_j T_i \qquad |i-j| \geq 2 
\end{equation}

We denote by $\bar u$ the image of an element $u \in \mathbb T_n$ in $S_n$, and $\tau_i = \bar T_i$. Using this projection, one defines a right-action of $\mathbb T_n$ on $\{1,\ldots,n\}$ by setting $k \cdot u := k \cdot \bar u$.

Let $\CC$ be a cubical $\omega$-category. For every $u \in \mathbb T_n$, we define a notion of $u$-invertible cell and a partial application $u \cdot \_ : \CC_n \to \CC_n$ defined on $u$-invertible cells as follows:
\begin{itemize}
\item Any $n$-cell of $\CC_n$ is $1$-invertible, and $1 \cdot A = A$.
\item For any $u \in \mathbb T_n$ and $1 \leq i < n$, a cell $A \in \CC_n$ is said to be $(T_i \cdot u)$-invertible if $A$ is $u$-invertible and $u \cdot  A$ is $T_i$-invertible. Moreover we set: $(T_i \cdot u) \cdot A := T_i (u \cdot A)$.
\end{itemize}
In particular we say that $A$ has a $u$-invertible shell if $\bm{\partial} A$ is $u$-invertible in $\Box_n \CC$.
\end{defn}

\begin{prop}\label{prop:u_inv_in_sells}
Let $\CC$ be a cubical $\omega$-category, and $A$ be an $n$-cell in $\CC$, with $n \geq 2$. Let $u \in \mathbb T_n$. Suppose $u \neq 1$. Then $A$ is $u$-invertible if and only if $A$ is invertible and has a $u$-invertible shell.
\end{prop}
\begin{proof}
We reason by induction on the length of $u$. If $u$ is of length $1$, there exists $1 \leq i < n$ such that $u = T_i$, and the result to prove becomes: \emph{$A$ is $T_i$-invertible if and only if $A$ is invertible and has a $T_i$-invertible shell}, which is exactly Proposition \ref{prop:carac_Ti_invertibility_inv}.

Otherwise, write $u = T_i v$, with $v \neq 1$, and suppose that $A$ is $u$-invertible. By definition $A$ is $v$-invertible and $v \cdot  A$ is $T_i$-invertible, so by induction $A$ is invertible and has a $v$-invertible shell. Moreover $v \cdot A$ is $T_i$-invertible, and hence has a $T_i$-invertible shell by Proposition \ref{prop:carac_Ti_invertibility_inv}. Since $\bm{\partial} (v \cdot A) = v \cdot \bm{\partial} A$, $\bm \partial A$ is $v$-invertible and $v \cdot \bm \partial A$ is $T_i$-invertible, so $\bm \partial A$ is $u$-invertible.

Reciprocally, suppose $A$ is invertible, and has a $(T_i \cdot v)$-invertible shell. Then $A$ has a $v$-invertible shell, and $v \cdot \bm \partial A$ is $T_i$-invertible. Since $A$ is also invertible, by induction $A$ is $v$-invertible, and since $\bm \partial (v \cdot A) = v \cdot \bm \partial A$, the cell $v \cdot A$ has a $T_i$-invertible shell. Moreover it is invertible, and so by Proposition \ref{prop:carac_Ti_invertibility_inv}, $v \cdot A$ is $T_i$-invertible, which means that $A$ is $u$-invertible.
\end{proof}

\begin{defn}
For $1 \leq i \leq n$, we define applications $\partial_i : \mathbb T_n \to \mathbb T_{n-1}$ as follows:
\[
\partial_i 1 = 1 \qquad
\partial_i T_j = 
\begin{cases}
1 & i = j,j+1 \\
T_{j_i} & i \neq j,j+1
\end{cases}
\qquad 
\partial_i (u \cdot v) = \partial_{i} u \cdot \partial_{i \cdot u} v.
\]
Note in particular that the applications $\partial_i$ are \emph{not} morphisms of monoids.
\end{defn}

\begin{lem}\label{lem:image_by_partial}
Let $u \in \mathbb T_n$. For all $1 \leq i \leq n$, and $1 \leq k \leq n$, we have:
\[
k \cdot \partial_i u = (k^i \cdot u)_{i \cdot u}
\]
\end{lem}
\begin{proof}
Note first the formula holds when $u$ is $1$ or a $T_j$. Finally, suppose the property holds for $u$ and $v$. Then we have:
\begin{align*}
k \cdot \partial_i (u \cdot v) & = k \cdot \partial_i u \cdot \partial_{i \cdot u} v  = (k^i \cdot u)_{i \cdot u} \cdot \partial_{i \cdot u} v \\ &
= ((k^i \cdot u)_{i \cdot u}^{i \cdot u} \cdot v)_{i \cdot u \cdot v}
= (k^i \cdot u \cdot v)_{i \cdot u \cdot v}
\end{align*}
\end{proof}

\begin{lem}\label{lem:u_invertibility_of_shells}
Let $\CC$ be a cubical $n$-category, $A \in (\Box \CC)_{n+1}$ and $u \in \mathbb T_{n+1}$. The cell $A$ is $u$-invertible if and only if for all $j \leq n+1$, $A_{j\cdot u}^\alpha$ is $\partial_{j} u$-invertible, and:

\[
\partial_{j}^\alpha (u \cdot A) = \partial_j u \cdot \partial_{j \cdot u}^\alpha A
\]

In particular, if $\CC$ is a cubical $\omega$-category, then $A \in \CC_{n+1}$ has a $u$-invertible shell if and only if for all $j \leq n+1$, $\partial_{j \cdot u}^\alpha A$ is $\partial_{j} u$-invertible.
\end{lem}
\begin{proof}
We reason by induction on the length of $u$. If $u$ is of length $0$, then $u = 1$ and for all $j$, $\partial_j u = 1$, so both conditions are empty and $(1 \cdot A)_j^\alpha = A_j^\alpha$.

Otherwise, write $u = T_i \cdot v$. Suppose that $A$ is $u$-invertible. Then $A$ is $v$-invertible, and $v \cdot A$ is $T_i$-invertible. Fix $j$ and $\alpha$. Then $\partial_j u = T_{i_j} \cdot \partial_{j \cdot T_i} v$. Let us show that $A_{j\cdot u}^\alpha$ is $\partial_j u$-invertible. We distinguish two cases:
\begin{itemize}
\item If $j = i$ (resp. $j = i+1$), then $\partial_j u =  \partial_{i+1} v$ (resp. $\partial_j v$), and $j \cdot u = (i+1) \cdot v$ (resp. $i \cdot v$). By induction, $A_{(i+1) \cdot v}^\alpha$ (resp. $A_{i \cdot v}^\alpha$) is $\partial_{i+1} v$-invertible (resp. $\partial_i v$-invertible).

\item Otherwise, then $\partial_j u = T_{i_j} \cdot \partial_j v$ and $j \cdot u = j \cdot v$. By induction hypothesis, $A^\alpha_{j \cdot v}$ is $\partial_j v$-invertible. Let us show that $\partial_j v \cdot A_{j \cdot v}^\alpha$ is $T_{i_j}$-invertible. First since $A$ is $T_i \cdot v$-invertible, $v \cdot A$ is $T_i$-invertible, and so by Lemma \ref{lem:Ti_inverses_in_shells}, $\partial_j^\alpha (v \cdot A)$ is $T_{i_j}$-invertible. Finally by induction, $\partial_j^\alpha (v \cdot A) = \partial_j v \cdot A_{j \cdot v}^\alpha$.
\end{itemize}
Finally, using the induction property on $v$, we get:
\[
(u \cdot A)_j^\alpha = (T_i \cdot v \cdot A)_j^\alpha =
\begin{cases}
(v \cdot A)_{i+1}^\alpha = \partial_{i+1} v \cdot A_{(i+1) \cdot v}^\alpha = \partial_i u \cdot A_{i \cdot u}^\alpha & j = i \\
(v \cdot A)_{i}^\alpha = \partial_{i} v \cdot A_{i \cdot v}^\alpha = \partial_{i+1} u \cdot A_{(i+1) \cdot u}^\alpha & j = i+1 \\
T_{i_j} (\partial_j v \cdot A)_j^\alpha = T_{i_j} \partial_j v \cdot A_{j \cdot v}^\alpha = \partial_j u \cdot A_{j \cdot u}^\alpha & j \neq i,i+1
\end{cases}
\]

Suppose now that for all $j$, $A_{j \cdot u}^\alpha$ is $\partial_ju$-invertible. Let us show that $A$ is $u$-invertible. First, let us prove that $A$ is $v$-invertible. Indeed let $j \leq n$, and let us show that $A_{j \cdot v}$ is $\partial_j v$-invertible.
\begin{itemize}
\item If $j \neq i,i+1$, we have that $A_{j \cdot u}^\alpha$ is $\partial_ju$-invertible. Since $\partial_j u = T_{i_j} \partial_j v$, and $j \cdot u = j \cdot v$, this means that $A_{j \cdot v}^\alpha$ is $\partial_j v$-invertible (and $\partial_j v \cdot A_{j \cdot v}^\alpha$ is $T_{i_j}$-invertible). 
\item If $j = i$ (resp. $j = i+1$) then $\partial_{i+1} u = \partial_{i} v$ (resp. $\partial_{i} u = \partial_{i+1} v$) and $(i+1) \cdot u = i \cdot v$ (resp. $i \cdot u = (i+1) \cdot v$), and as a consequence $A_{j \cdot v}^\alpha$ is $\partial_j v$-invertible. 
\end{itemize}
Finally by induction, $A$ is $v$-invertible. Let us show that $v \cdot A$ is $T_i$-invertible. Indeed, for $j \neq i,i+1$, $(v \cdot A)_j^\alpha = \partial_j v \cdot A_{j \cdot v}^\alpha$ is $T_{i_j}$-invertible, and so $v \cdot A$ is $T_i$-invertible by Lemma \ref{lem:Ti_inverses_in_shells}.
\end{proof}

\begin{lem}\label{lem:relations_respect_invert}
Let $\CC$ be a cubical $\omega$-category. 
\begin{itemize}
\item If $A$ is $T_i T_i$-invertible, then:
 \begin{equation} \label{eq:Ti1}
 T_i T_i \cdot A = A
 \end{equation}
\item A cell $A \in \CC_n$ is $T_i T_{i+1} T_i$-invertible if and only if it is $T_{i+1} T_i T_{i+1}$-invertible, and 
\begin{equation}\label{eq:Ti2}
T_i T_{i+1} T_i A = T_{i+1} T_i T_{i+1} A
\end{equation}
\item Let $i,j < n$ such that $|i-j| \geq 2$. A cell $A \in \CC_n$ is $T_i T_j$-invertible if and only if it is $T_j T_i$-invertible, and 
\begin{equation}\label{eq:Ti3}
T_i T_j \cdot A = T_j T_i \cdot A
\end{equation}
\end{itemize}
\end{lem}
\begin{proof}
For the first one, notice that the axioms \eqref{eq:axiom_inv} and \eqref{eq:axiom_inv_bis} are linked by an obvious symmetry, meaning that if $B$ is the $T_i$-inverse of $A$, then $A$ is the $T_i$-inverse of $A$. This means in particular that $T_i T_i A = A$.

For the second one, a cell $A \in \CC_n$ is $T_iT_{i+1}T_i$-invertible if and only if it is invertible and $\bm \partial A$ is $T_iT_{i+1}T_i$-invertible, that is for all $j \leq n$, $\partial_{j \cdot T_i T_{i+1} T_i}^\alpha A$ is $\partial_j(T_i T_{i+1} T_i)$. Notice that:
\begin{equation}\label{eq:braid_relation_and_boundary}
\partial_j(T_i T_{i+1} T_i) = 
\begin{cases}
T_{i_j} T_{i_j+1} T_{i_j} & j \neq i,i+1,i+2 \\
T_i & j = i,i+1,i+2 \\
\end{cases}
\qquad
\partial_j (T_{i+1} T_i T_{i+1}) = 
\begin{cases}
T_{i_j+1} T_{i_j} T_{i_j +1} & j \neq i,i+1,i+2 \\
T_i &  j =i,i+1,i+2 \\
\end{cases}
\end{equation}
Therefore by induction on $n$, a cell is $T_i T_{i+1} T_i$-invertible if and only if it is $T_{i+1} T_i T_{i+1}$-invertible.
Let $A$ be such a cell. Let us show that $T_iT_{i+1}T_iA$ is the $T_{i+1}$-inverse of $T_iT_{i+1}A$. Indeed we have:

\begin{align*}
  \renewcommand{\arraystretch}{1.5}
  \begin{tabular}{ | c | c | }
  \hline			
    $\Gamma_{i+1}^+ T_i \partial_{i}^- A$ & $T_iT_{i+1}A$ \\
  \hline
  $T_iT_{i+1}T_iA$ & $\Gamma_{i+1}^- T_i \partial_{i+1}^+ A$ \\
  \hline  
\end{tabular}
\quad
\directions{i+1}{i+2}
& =
T_iT_{i+1}
  \renewcommand{\arraystretch}{1.5}
  \begin{tabular}{ | c | c | }
  \hline			
    $\Gamma_i^+ \partial_{i}^- A$ & $A$ \\
  \hline
  $T_iA$ & $\Gamma_i^- \partial_{i+1}^+ A$ \\
  \hline  
\end{tabular}
\quad
\directions{i}{i+1} \\
& =
T_iT_{i+1} (\Gamma_i^- \partial_{i+1}^+ A \star_i \Gamma_i^+ \partial_i^+ A) \\
& =
T_iT_{i+1} \Gamma_i^- \partial_{i+1}^+ A \star_{i+1} T_iT_{i+1} \Gamma_i^+ \partial_i^+ A \\
& =
\Gamma_{i+1}^- \partial_{i+2}^- T_i T_{i+1} A \star_{i+1}   \Gamma_{i+1}^+ \partial_{i+1}^+ T_i T_{i+1} A
\end{align*} 
The other axioms are verified in the same fashion.
\end{proof}

\begin{defn}\label{def:sym_cub_cat}
A symmetric cubical $\omega$-category $\CC$ is a cubical $\omega$-category $\CC$ equipped with (total) maps $T_i : \CC_n \to \CC_n$, for $1 \leq i \leq n-1$, satisfying the equalities \eqref{eq:Ti_and_face} to \eqref{eq:Ti_and_star} and \eqref{eq:Ti1} to \eqref{eq:Ti3}.
\end{defn}

\begin{remq}
Note that a symmetric cubical $\omega$-category is close but not the same as the notion of symmetric cubical category defined by Grandis in \cite{G07}. A symmetric cubical category in the sense of Grandis would be a symmetric cubical $\omega$-category (in the sense of \ref{def:sym_cub_cat}, but without connections) object in the category $\mathbf{Cat}$.
\end{remq}

\begin{prop}
Let $\CC$ be a cubical $\omega$-category. The maps $A \mapsto T_i A$ induce a structure of symmetric cubical category on $\CC$. 
\end{prop}
\begin{proof}
The fact that the maps $T_i$ are total is a consequence of Corollary \ref{cor:carac_omega_1}, and the equations they verify are a consequence of Proposition \ref{prop:Ti_inv_and_other} and Lemma \ref{lem:relations_respect_invert}.
\end{proof}

We now make explicit the (partial) action of the symmetric groups on the $n$-cells of a cubical category. To do so, we rely on Theorem \ref{thm:representability_of_permutations}, a classical result about the symmetric group.

\begin{defn}
For $u \in S_n$, we define the \emph{length} of $u$ as the integer $\l(u) = \min \{\l(v) | v \in \mathbb T_n \text{ and } \bar v = u \}$. A \emph{representative of minimal length} of $u$ in $\mathbb T_n$ is an element $v \in \mathbb T_n$ such that $\bar v = u$ and $\l(v) = \l(u)$.
\end{defn}

\begin{thm}\label{thm:representability_of_permutations}
Let $u,v \in \mathbb T_n$. If $u$ and $v$ are two representative of minimal length of a same permutation $\sigma$, then $u \equiv v$, where $\equiv$ is the congruence on $\mathbb T_n$ generated by \eqref{eq:braid_relation} and \eqref{eq:comm_relation}.
\end{thm}

\begin{defn}
Let $\CC$ be a cubical $\omega$-category. For every $A \in \CC_n$ and $\sigma \in S_n$, we say that $A$ is $\sigma$-invertible if there exists a representative of minimal length $u$ of $\sigma$ such that $A$ is $u$-invertible, and we define $\sigma \cdot A := u \cdot A$. By Lemma \ref{lem:relations_respect_invert} and Theorem \ref{thm:representability_of_permutations}, this is independent of the choice of a minimal representative of $\sigma$. 
\end{defn}

\begin{prop}\label{prop:sigma_inv}
The composites of the maps $\partial_i : \mathbb T_n \to \mathbb T_{n-1}$ with the projection $\mathbb T_{n-1} \twoheadrightarrow S_{n-1}$ are compatible with the relations \eqref{eq:transpo_nilpotente} to \eqref{eq:comm_relation}. Hence they induce maps $\partial_i : S_n \to S_{n-1}$, satisfying:

\[
\partial_i 1 = 1 \qquad
\partial_i \tau_j = 
\begin{cases}
1 & i = j,j+1 \\
\tau_{j_i} & i \neq j,j+1
\end{cases}
\qquad 
\partial_i (\sigma \cdot \tau) = \partial_{i} \sigma \cdot \partial_{i \cdot \sigma} \tau.
\]

Specifically, for $1 \leq i \leq n$ and $\sigma \in S_n$, $\partial_i \sigma$ is the (necessarily unique) permutation satisfying for all $1 \leq j \leq n-1$:
\begin{equation}\label{eq:image_by_partial}
j \cdot \partial_i \sigma = (j^i \cdot \sigma)_{i\cdot \sigma}
\end{equation}

Let $\CC$ be a cubical $n$-category, and $\sigma \in S_n$. A cell $A \in (\Box \CC)_{n+1}$ is $\sigma$-invertible if and only if for all $j \leq n$, $A_{j \cdot \sigma}^\alpha$ is $\partial_j \sigma$-invertible, and:
\begin{equation}\label{eq:sigma_and_faces}
\partial_j^\alpha(\sigma \cdot A) = \partial_j \sigma \cdot \partial_{j \cdot \sigma}^\alpha A
\end{equation}

Finally, let $\sigma \in S_n$. If $\sigma \neq 1$, then a cell $A \in \C_n$ is $\sigma$-invertible if and only if $A$ is invertible and $\bm \partial A$ is $\sigma$-invertible.
\end{prop}
\begin{proof}
For the first point we simply verify the equalities as needed (note in particular that the compatibility of $\partial_i$ with Equation \eqref{eq:braid_relation} is a consequence of Equation \eqref{eq:braid_relation_and_boundary}.

The rest of the results is a consequence of Proposition \ref{prop:u_inv_in_sells}, together with Lemma \ref{lem:image_by_partial} and \ref{lem:u_invertibility_of_shells}.
\end{proof}

\begin{remq}
The operations $\partial_i$ applied to a permutation $\sigma$ correspond to deleting the $i$-th string in the string diagram  representation of $\sigma$. For example, by definition we have:
\[
\partial_1 (\tau_1\tau_2) = (\partial_1 \tau_1) \cdot (\partial_2 \tau_2) = 1
\qquad
\partial_2 (\tau_1\tau_2)  = (\partial_2 \tau_1) \cdot (\partial_1 \tau_2) = \tau_1
\qquad
\partial_3 (\tau_1\tau_2) = (\partial_3 \tau_1) \cdot (\partial_3 \tau_2) = \tau_1
\]
Which can be diagrammatically represented as:
\[
\partial_1(\twocell{braid3}) = \twocell{2}
\qquad
\partial_2(\twocell{braid3}) = \twocell{braid2}
\qquad 
\partial_3(\twocell{braid3}) = \twocell{braid2}
\]
More generally, the relation $\partial_i (\sigma \cdot \tau) = \partial_i \sigma \cdot \partial_{i \cdot \sigma} \tau$ corresponds to the diagram:
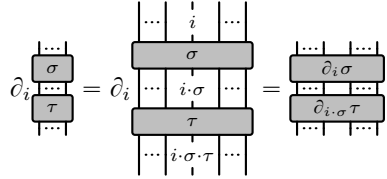
\[
\partial_i \twocell{dots *1 sigma2 *1 dots *1 tau2 *1 dots} = \partial_i \twocell{(dots *0 i0 *0 dots) *1 sigma5 *1 (dots *0 isigma *0 dots) *1 tau5 *1 (dots *0 ist *0 dots)}
= \twocell{(dots *0 dots) *1 sigma4 *1 (dots *0 dots) *1 tau4 *1 (dots *0 dots)}
\]

Finally, Equation \eqref{eq:sigma_and_faces} corresponds to the diagram:
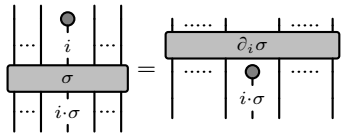
\[
\twocell{(2 *0 partial *0 2) *1 (dots *0 i0 *0 dots) *1 sigma5 *1 (dots *0 isigma *0 dots)}
 = 
\twocell{(dots *0 dots) *1 sigma4 *1 (dots *0 partial *0 dots) *1 (2 *0 isigma *0 2)}
\]
\end{remq}

\begin{lem}\label{lem:sigma_and_others}
Let $\CC$ be a cubical $\omega$-category, and $A \in \CC_n$. If $\epsilon_i A$  is $\sigma$-invertible, then $A$ is $\partial_{i \cdot \sigma^-} \sigma$-invertible and:
\[
\sigma \cdot \epsilon_i A = \epsilon_{i \cdot \sigma^-} (\partial_{i \cdot \sigma^-} \sigma \cdot A)
\]
If $\Gamma_i^\alpha$ is $\sigma$-invertible then $A$ is also $\partial_{i \cdot \sigma^-} \sigma$-invertible and if $(i+1) \cdot \sigma^-  = i \cdot \sigma^-  +1$ we have: 
\[
\sigma \cdot \Gamma_i^\alpha A = \Gamma_{i \cdot \sigma^-}^\alpha (\partial_{i \cdot \sigma^-} \sigma \cdot A)
\]
\end{lem}
\begin{proof}
If $\epsilon_i A$  is $\sigma$-invertible, then $A = \partial_i^- \epsilon_i A$ is $\partial_{i \cdot \sigma^-} \sigma$ by Proposition \ref{prop:sigma_inv}. 

To show the equality, we reason by induction on $n$. If $n  = 0$ then $\sigma = 1$ and the result is verified. Otherwise, suppose $n> 0$. By Lemma \ref{lem:carac_Ti_invertibility_Ri}, both sides of the equation are thin, and so they are equal if and only if their shells are equal. Note first that for $j = i \cdot \sigma^-$:
\[
\partial_j^\alpha (\sigma \cdot  \epsilon_i A) = 
\partial_j \sigma \cdot \partial_{i}^\alpha \epsilon_i A =
\partial_j \sigma \cdot A
=
\partial_j^\alpha \epsilon_j (\partial_j \sigma \cdot A)
\]
Now for $j \neq i \cdot \sigma^-$:
\begin{align*}
\partial_j^\alpha (\sigma \cdot \epsilon_i A) 
& = 
\partial_j \sigma \cdot \partial_{j \cdot \sigma}^\alpha \epsilon_i A
=
\partial_j \sigma \cdot \epsilon_{i_{j \cdot \sigma}} \partial_{(j \cdot \sigma)_i}^\alpha A
\end{align*}
Note that $\partial_j(\sigma \cdot \sigma^-) = \partial_j \sigma \cdot \partial_{j \cdot \sigma} \sigma^- = 1$, so $(\partial_j \sigma)^- = \partial_{j \cdot \sigma} \sigma^-$, and so by Proposition \ref{prop:sigma_inv}:
\[
i_{j \cdot \sigma} \cdot (\partial_j \sigma)^- = 
(i_{j \cdot \sigma}^j \cdot \sigma^-)_{j \cdot \sigma \cdot \sigma^-} = (i \cdot \sigma^-)_j
\]
So by induction hypothesis, we have $\partial_j^\alpha(\sigma \cdot \epsilon_i A) =
\epsilon_{(i \cdot \sigma^-)_j} (\partial_{(i \cdot \sigma^-)_j} \partial_j \sigma \cdot \partial_{(j \cdot \sigma)_i}^\alpha A)$.
On the other hand, note that $j_{i \cdot \sigma^-} \cdot \partial_{i \cdot \sigma^-} \sigma = (j_{i \cdot \sigma^-}^{i \cdot \sigma^-} \cdot \sigma)_{i \cdot \sigma^- \cdot \sigma} =
(j \cdot \sigma)_i$. Applying this we get:
\[
\partial_j^\alpha \epsilon_{i \cdot \sigma^-} (\partial_{i \cdot \sigma^-} \sigma \cdot A) =
\epsilon_{(i \cdot \sigma^-)_j} \partial_{j_{i \cdot \sigma^-}}^\alpha (\partial_{i \cdot \sigma^-} \sigma \cdot A)
=
\epsilon_{(i \cdot \sigma^-)_j} (\partial_{j_{i \cdot \sigma^-}} \partial_{i \cdot \sigma^-} \sigma \cdot \partial_{(j \cdot \sigma)_i}^\alpha A).
\]
Finally it remains to show that $\partial_{j_{i \cdot \sigma^-}} \partial_{i \cdot \sigma^-} \sigma = \partial_{(i \cdot \sigma^-)_j} \partial_j \sigma$. More generally, let us show that for any $i \neq j$, $\partial_{i_j} \partial_j \sigma = \partial_{j_i} \partial_i \sigma$. Indeed, for any $k$:
\begin{equation}\label{eq:commute_partial_sigma}
\partial_{i_j} \partial_j \sigma  \cdot k =
(((k^j)^{i_j} \cdot \sigma)_{i_j})_j = (k^{j,i} \cdot \sigma)_{i,j}
\end{equation}
And this formula is symmetric in $i$ and $j$ by Lemma \ref{lem:permut_indices}.

We now move on to the second equality. Once again if $\Gamma_i^\alpha A$ is $\sigma$-invertible, then $A = \partial_i^\alpha \Gamma_i^\alpha A$ is $\partial_{i \cdot \sigma^-} \sigma$-invertible by Proposition \ref{prop:sigma_inv}. We show the equality by induction on $n$. If $n = 1$, then the only permutation $\sigma$ satisfying $(i+1) \cdot \sigma^-  = i \cdot \sigma^-  +1$ is the identity, and the result is verified. Suppose now $n \geq 1$, and let $\sigma \in S_n$ such that  $(i+1) \cdot \sigma^-  = i \cdot \sigma^-  +1$. As previously, Lemma \ref{lem:carac_Ti_invertibility_Ri} show that both sides of the equation are thin, and so they are equal if and only if their shells are equal. Let us calculate their faces. Let $1 \leq j \leq n$ and $\beta = \pm$. We start by treating the case where $j  = i \cdot \sigma^- $. For $\beta = \alpha$ we have:
\begin{align*}
\partial_j^\alpha (\sigma \cdot \Gamma_i^\alpha A) & = 
\partial_j \sigma \cdot \partial_{j \cdot \sigma}^\alpha \Gamma_i^\alpha A  =
\partial_j \sigma \cdot \partial_i^\alpha \Gamma_i^\alpha A \\
& =
\partial_j \sigma \cdot A = \partial_j^\alpha \Gamma_j^\alpha (\partial_j \sigma \cdot A)
\end{align*}
Now for $\beta = - \alpha$. Note first that $j \cdot \partial_j \sigma = (j^j \cdot \sigma)_i = ((j+1) \cdot \sigma)_i = (i+1)_i = i$ (we here use the hypothesis on $\sigma$). As a consequence $i \cdot (\partial_j \sigma)^- = j$, and:
\begin{align*}
\partial_j^{-\alpha} (\sigma \cdot \Gamma_i^\alpha A) & =
\partial_j \sigma \cdot \partial_i^{-\alpha} \Gamma_i^\alpha A \\
& =
\partial_j \sigma \cdot \epsilon_i \partial_i^{-\alpha} A \\
& =
\epsilon_j (\partial_j \partial_j \sigma \cdot \partial_i^{-\alpha} A) \\
\partial_j^{-\alpha} \Gamma_j^\alpha (\partial_j \sigma \cdot A) & =
\epsilon_j \partial_j^{-\alpha} (\partial_j \sigma \cdot A) \\
& =
\epsilon_j (\partial_j \partial_j \sigma \cdot \partial_i^{-\alpha} A)
\end{align*}
The case where $j = i \cdot \sigma^- +1$ is similar. We now study the general case where $\beta = \pm$ and $j \neq i \cdot \sigma^-, i \cdot \sigma^- +1$:
\begin{align*}
\partial_j^\beta (\sigma \cdot \Gamma_i^\alpha A) & =
\partial_j \sigma \cdot \partial_{j \cdot \sigma}^\beta \Gamma_i^\alpha A \\
& =
\partial_j \sigma \cdot \Gamma_{i_{j \cdot \sigma}}^\alpha \partial_{(j\cdot \sigma)_i}^\beta A
\end{align*}
\begin{align*}
\partial_j^\beta \Gamma_{i \cdot \sigma^-}^\alpha (\partial_{i \cdot \sigma^-} \sigma \cdot A) & =
\Gamma_{(i \cdot \sigma^-)_j}^\alpha \partial_{j_{i \cdot \sigma^-}}^\beta (\partial_{i \cdot \sigma^-} \sigma \cdot A) \\
& =
\Gamma_{(i \cdot \sigma^-)_j}^\alpha (\partial_{j_{i \cdot \sigma^-}} \partial_{i \cdot \sigma^-} \sigma \cdot \partial_{j_{i \cdot \sigma^-} \cdot \partial_{i \cdot \sigma^-} \sigma}^\beta A)
\end{align*}
To conclude using the induction hypothesis, we need to show that $j_{i \cdot \sigma^-} \cdot \partial_{i \cdot \sigma^-} \sigma = (j \cdot \sigma)_i$, and that $i_{j \cdot \sigma} \cdot (\partial_j \sigma)^- = (i \cdot \sigma^-)_j$. These equations hold because we have:
\[
j_{i \cdot \sigma^-} \cdot \partial_{i \cdot \sigma^-} \sigma
=
(j_{i \cdot \sigma^-}^{i \cdot \sigma^-} \cdot \sigma)_{i \cdot \sigma^- \cdot \sigma} = (j \cdot \sigma)_i
\]
\[
(i \cdot \sigma^-)_j \cdot \partial_j \sigma =
((i \cdot \sigma^-)_j^j \cdot \sigma)_{j \cdot \sigma} = i_{j \cdot \sigma}
\]
\end{proof}

\begin{remq}
Diagrammatically, the equations from Lemma \ref{lem:sigma_and_others} correspond to the following diagrams:
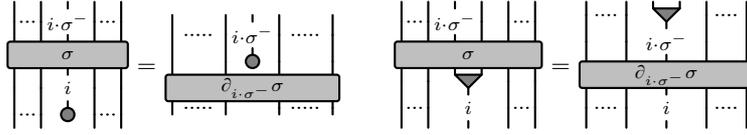
\[
\twocell{(dots *0 isigmaminus *0 dots) *1 sigma5 *1 (2 *0 i0 *0 2) *1 (dots *0 epsilon *0 dots)}
=
\twocell{(dots *0 isigmaminus *0 dots) *1 (2 *0 epsilon *0 2) *1 sigma4minus *1 (dots *0 dots)}
\qquad
\twocell{(dots *0 isigmaminus *0 1 *0 dots) *1 sigma6  *1 (2 *0 gamma *0 2) *1 (dots *0 i0 *0 dots)} 
=
\twocell{(dots *0 gamma *0 dots) *1 (2 *0 isigmaminus *0 2) *1 sigma5minus *1 (dots *0 i0 *0 dots)}
\]
\end{remq}

\begin{remq}
In this Section, we restricted ourselves to the $T_i$-inverses. However, all the previous results can be adapted to also consider the $R_i$-inverses. The action of the symmetric groups are then extended into an action of the Hyperoctahedral groups $BC_n$, which are the full groups of permutations of the hypercubes. A presentation of the group $BC_n$ is given by the generators $R_i$ (for $1 \leq i \leq n$ and $T_i$ (for $1 \leq i < n$), subject to the relations:

\begin{gather*}
T_i T_i = 1 
\qquad
T_i T_{i+1} T_i = T_{i+1} T_i T_{i+1}
\qquad
T_j  T_i = T_j T_i \quad |i-j| \geq 2  \\
R_iR_i = 1 
\qquad 
R_i R_j = R_j R_i \quad i \neq j \\
T_iR_i = R_{i+1}T_i 
\qquad 
T_iR_{i+1} = R_iT_i 
\qquad
T_iR_j = R_jT_i \quad j \neq i,i+1
\end{gather*}

In particular the groups $BC_n$ are Coxeter groups and they hence verify an analogue to Theorem \ref{thm:representability_of_permutations}, often called Matsumoto's Theorem \cite{M64}.
\end{remq}

\subsection{Transfors between cubical \texorpdfstring{$\omega$}{omega}-categories}
\label{subsec:transfors}

Let $\C$ and $\D$ be two categories, and $F,G : \C \to \D$ be functors. Recall that a natural transformation $\eta$ from $F$ to $G$ is given by a map $\eta : \C_0 \to \D_1$ such that, for all $x \in \C_0$, $\s(\eta_x) = F(x)$, $\t(\eta_x) = G(x)$, and for all $f:x \to y \in \C_1$ the following diagram commutes:  
\begin{equation}\label{eq:naturality_square}
\begin{tikzpicture}[baseline=(current  bounding  box.center)]
\matrix (m) [matrix of math nodes, 
			nodes in empty cells,
			column sep = 1cm, 
			row sep = 1cm] 
{
F(x) & F(y) \\
G(x) & G(y) \\
};
\draw[-to] (m-1-1) to node [above] {$F(f)$} (m-1-2.west|-m-1-1);
\draw[-to] (m-2-1) to node [below] {$G(f)$} (m-2-2.west|-m-2-1);
\draw[-to] (m-1-1) to node [left] {$\eta_x$} (m-2-1);
\draw[-to] (m-1-2) to node [right] {$\eta_y$} (m-2-2);
\end{tikzpicture}
\end{equation}
Natural transformations compose, and so for any categories $\C$ and $\D$ there is a category $\mathbf{Cat}(\C,\D)$. 

If $\C$ and $\D$ are two globular $2$-categories, and $F,G : \C \to \D$ are two functors, then there are multiple ways to extend the notion of natural transformation. A \emph{lax natural transformation} from $F$ to $G$ consists in maps $\eta : \C_0 \to \D_1$ and $\eta : \C_1 \to \D_2$, satisfying some compatibility conditions. In particular, for $f : x \to y \in \C_1$, the $2$-cell $\eta_f \in \D_2$ is required to have the following source and target:
\[
\begin{tikzpicture}[baseline=(current  bounding  box.center)]
\matrix (m) [matrix of math nodes, 
			nodes in empty cells,
			column sep = 1cm, 
			row sep = 1cm] 
{
F(x) & F(y) \\
G(x) & G(y) \\
};
\draw[-to] (m-1-1) to node [above] {$F(f)$} (m-1-2.west|-m-1-1);
\draw[-to] (m-2-1) to node [below] {$G(f)$} (m-2-2.west|-m-2-1);
\draw[-to] (m-1-1) to node [left] {$\eta_x$} (m-2-1);
\draw[-to] (m-1-2) to node [right] {$\eta_y$} (m-2-2);
\doublearrow{arrows = {->}, shorten <= .2cm, shorten >= .2cm}
{(m-1-2) -- node [fill = white] {$\eta_f$} (m-2-1)}
\end{tikzpicture}
\]
An \emph{oplax natural transformation} requires the $2$-cell $\eta_f$ to be in the opposite direction. This leads to two different notions of the $2$-category of functors between $\C$ and $\D$, where objects are functors from $\C$ to $\D$, $1$-cells are lax (resp. oplax) natural transformations, and $2$-cells are modifications. Modifications consist of a map $\C_0 \to \D_2$ satisfying some compatibility conditions. Notice that, if $\eta$ is a lax natural transformation and $\eta_f$ is invertible for all $f \in C_1$, then replacing $\eta_f$ by its inverse yields an oplax natural transformation (and reciprocally when reversing the role of lax and oplax natural transformation). Such natural transformations are called \emph{pseudo}.

More generally, if $\C$ and $\D$ are $\omega$-categories, there are notions of lax and oplax $k$-transfors between them (following terminology by Crans \cite{C03}), consisting of maps $\C_n \to \D_{n+k}$, for all $n \geq 0$. In particular, $0$-transfors correspond to functors, and lax (resp. oplax) $1$-transfors to lax (resp. oplax) natural transformations. 

Similar constructions can be made in cubical $\omega$-categories, and are recalled in Definition \ref{def:transfor}. This definition uses the notion of Crans-Grey tensor product between cubical $\omega$-categories. One benefit of working in cubical categories is that this tensor product has a very natural expression in this setting, and so we are able to make explicit the conditions that transfors between cubical $\omega$-categories have to satisfy. Next we define the notion of pseudo transfor, using the notion of $\sigma$-invertibility defined in Section \ref{subsec:sym_cubical_categories}. In Proposition \ref{prop:carac_pseudo} we give an alternative characterisation of pseudo transfors. Lastly we prove that the notions of pseudo lax and oplax transfors coincide in Proposition \ref{prop:transfors_isomorphism}.

\begin{defn}\label{def:transfor}
We exhibited in Section \ref{sec:recall_cub_cat} a structure of \emph{cubical $\omega$-category} object in $\Cat{\omega}^{\op}$ on the family $\cube{n}^\G$. Applying the functor $\lambda$ gives the family $\cube{n}^\CC := \lambda(\cube{n}^\G)$ the structure of a cubical $\omega$-category object in $\CCat{\omega}^{\op}$.

Consequently, if $\CC$ and $\DD$ are cubical $\omega$-categories, then both the families (of sets) $\Lax(\CC,\DD)_n = \CCat{\omega}(\cube{n}^C \otimes \CC, \DD)$ and $\oplax(\CC,\DD)_n = \CCat{\omega}(\CC \otimes \cube{n}^C,\DD)$ come equipped with cubical $\omega$-category structures (where we denote by $\otimes$ the monoidal product on $\CCat{\omega}$ as defined in \cite{ABS02}).

We call an element $F \in \Lax(\CC,\DD)_n$ (resp. $F \in \oplax(\CC,\DD)_n$) an \emph{lax $n$-transfor} (resp. an \emph{oplax $n$-transfor}) from $\CC$ to $\DD$. Unfolding the definition of the monoidal product on $\CCat{\omega}$ as defined in \cite{ABS02}, Section 10, a lax $p$-transfor (resp. oplax $p$-transfor) is a family of maps $F_n : \CC_n \to \DD_{n+p}$ satisfying the equations \eqref{eq:Lax_boundary} to \eqref{eq:Lax_composition} (resp. \eqref{eq:Oplax_boundary} to \eqref{eq:Oplax_composition}).
\begin{multicols}{2}
\begin{equation}\label{eq:Lax_boundary}
\partial_{p+i}^\alpha F_n (A) = F_{n-1} (\partial_{i}^\alpha A) 
\end{equation}
\begin{equation}\label{eq:Lax_epsilon}
F_n (\epsilon_i A) =  \epsilon_{p+i} F_{n-1} (A) 
\end{equation}
\begin{equation}\label{eq:Lax_gamma}
F_n (\Gamma_i^\alpha A) = \Gamma_{p+i}^\alpha F_{n-1} (A) 
\end{equation}
\begin{equation}\label{eq:Lax_composition}
F_n (A \star_i B) = F_n(A) \star_{p+i} F_n(B) 
\end{equation}

\begin{equation}\label{eq:Oplax_boundary}
\partial_{i}^\alpha F_n (A) = F_{n-1} (\partial_{i}^\alpha A) 
\end{equation}
\begin{equation}
F_n (\epsilon_i A) =  \epsilon_{i} F_{n-1} (A) 
\end{equation}
\begin{equation}
F_n (\Gamma_i^\alpha A) = \Gamma_{i}^\alpha F_{n-1} (A)
\end{equation}
\begin{equation}\label{eq:Oplax_composition}
F_n (A \star_i B) = F_n(A) \star_{i} F_n(B) 
\end{equation}
\end{multicols}

Moreover, the cubical $\omega$-category structure on $\Lax(\CC,\DD)$ (resp. on $\oplax(\CC,\DD)$) is given by the equations \eqref{eq:boundary_of_lax} to \eqref{eq:composition_of_lax} (resp. \eqref{eq:boundary_of_oplax} to \eqref{eq:composition_of_oplax}).
\begin{multicols}{2}
\begin{equation}\label{eq:boundary_of_lax}
(\partial_i^\alpha F)_n (A) = \partial_i^\alpha (F_n (A))
\end{equation}
\begin{equation}
(\epsilon_i F)_n (A) = \epsilon_i (F_n(A))
\end{equation}
\begin{equation}
(\Gamma_i^\alpha F)_n(A) = \Gamma_i^\alpha (F_n(A))
\end{equation}
\begin{equation}\label{eq:composition_of_lax}
(F \star_i G)_n (A) = F_n(A) \star_i G_n(A)
\end{equation}

\begin{equation}\label{eq:boundary_of_oplax}
(\partial_i^\alpha F)_n (A) = \partial_{n+i}^\alpha (F_n (A))
\end{equation}
\begin{equation}
(\epsilon_i F)_n (A) = \epsilon_{n+i} (F_n(A))
\end{equation}
\begin{equation}
(\Gamma_i^\alpha F)_n(A) = \Gamma_{n+i}^\alpha (F_n(A))
\end{equation}
\begin{equation}\label{eq:composition_of_oplax}
(F \star_i G)_n (A) = F_n(A) \star_{n+i} G_n(A)
\end{equation}
\end{multicols}
\end{defn}

The following Proposition is a consequence of \cite{ABS02}, Section 10.
\begin{prop}
Let $\CC$ be a cubical $\omega$-category. The functors $(\_ \otimes \CC)$ and $(\CC \otimes \_)$ are respectively left-adjoint to the functors $\Lax(\CC,\_)$ and $\oplax(\CC,\_)$. This implies that $\CCat{\omega}$ is a biclosed monoidal category.
\end{prop}

\begin{defn}
Let $n,m \geq 0$ be integers. We denote by $\rho_{n,m} \in S_{n+m}$ the following permutations:
\[
i \cdot \rho_{n,m} := \begin{cases}
i+n & i \leq n \\
i-n & i > n
\end{cases}
\]

Let $\CC$ and $\DD$ be cubical $\omega$-categories. We say that a lax $p$-transfor $F : \CC \to \DD$ is \emph{pseudo} if for all $A \in \CC_n$, $F(A)$ is $\rho_{n,p}$-invertible. We say that an oplax $p$-transfor $F : \CC \to \DD$ is \emph{pseudo} if for all $A \in \CC_n$, $F(A)$ is $\rho_{p,n}$-invertible.
\end{defn}

\begin{prop}\label{prop:carac_pseudo}
Let $\CC$ and $\DD$ be cubical $\omega$-categories, and $F : \CC \to \DD$ a lax $p$-transfor (resp. an oplax $p$-transfor). Then $F$ is pseudo if and only if:
\begin{itemize}
\item Either $p = 0$,
\item Or $p > 0$, for all $n > 0$ and all $A \in \CC_n$, $F(A)$ is invertible, and for all $1 \leq i \leq p$, $\partial_i^\alpha F$ is pseudo.
\end{itemize}

Moreover, if $F$ is pseudo, then so are $\Gamma_i^\alpha F$ ($1 \leq i \leq p$), $\epsilon_i F$ ($1 \leq i \leq p+1$) and, if $G$ is a pseudo lax $p$-transfor (resp. pseudo oplax $p$-transfor) then $F \star_i G$ (if defined) is also pseudo, for $1 \leq i \leq p$. 
\end{prop}
\begin{proof}
Let us prove the result for pseudo lax $p$-transfors, the case of pseudo oplax $p$-transfors being similar. If $p = 0$, then for all $n$, $\rho_{n,p} = 1$. Since any cell in $\DD$ is $1$-invertible, any lax $0$-transfor is pseudo.

Suppose now $p > 0$. Let $F \in \Lax(\CC,\DD)_p$, and suppose $F$ is pseudo.  Let $n > 0$ and $A \in \CC_n$. Then $\rho_{n,p} \neq 1$, and by Proposition \ref{prop:sigma_inv}, $F_n(A)$ is invertible. Moreover, for $1 \leq i \leq p$, $(\partial_i^\alpha F)_n(A) = \partial_{(p+i) \cdot \rho_{n,p}}^\alpha (F_n(A))$ is $\partial_{p+i} \rho_{n,p}$-invertible. Since $\partial_{p+i} \rho_{n,p} = \rho_{n,p-1}$, we just proved that for all $A \in \CC_n$, $(\partial^\alpha_i F)_n(A)$ is $\rho_{n,p-1}$-invertible, and so $\partial^\alpha_i F$ is pseudo.

Reciprocally, suppose that for all $n > 0$, $F_n(A)$ is invertible, and for all $1 \leq i \leq p$, $\partial_i^\alpha F$ is pseudo. We reason by induction on $n$ to show that for all $A  \in \CC_n$, $F_n(A)$ is $\rho_{n,p}$-invertible. If $n = 0$, $\rho_{n,p} = 1$ and $F_n(A)$ is $\rho_{n,p}$-invertible. If $n \geq 1$, then $F(A)$ is invertible and for all $1 \leq i \leq p$, $\partial_{(i+n) \cdot \rho_{n,p}}^\alpha (F_n(A)) = (\partial_i^\alpha F)(A) $ is $\rho_{n,p-1}$-invertible, while for all $1 \leq i \leq n$, $\partial_{i \cdot \rho_{n,p}}^\alpha (F_n(A)) = F_{n-1}(\partial_i^\alpha A)$ is $\rho_{n-1,p}$-invertible by induction. In conclusion, $F_n(A)$ is invertible, and for all $1 \leq i \leq p+n$, $\partial_i^\alpha (F_n(A))$ is $\partial_i \rho_{n,p}$-invertible. By Proposition \ref{prop:sigma_inv}, $F_n(A)$ is $\rho_{n,p}$-invertible.

We reason by induction on $p$ to show that, for any pseudo lax $p$-transfor.  $F$,  $\epsilon_i F$ and $\Gamma_i^\alpha F$ are pseudo. Let $A \in \CC_n$. By equations \eqref{eq:Lax_epsilon} and \eqref{eq:Lax_gamma}, $(\epsilon_i F) (A)$ and $(\Gamma_i^\alpha F)(A)$ are thin cells, and so in particular are invertible. Moreover the cubical $\omega$-category structure on $\Lax(\CC,\DD)$ show that for all $j$, we have:
\[
\partial_i^\alpha \epsilon_j F = 
\begin{cases}
\epsilon_{j_i} \partial_{i_j}^\alpha F & i \neq j \\
F & i = j
\end{cases}
\qquad 
\partial_i^\alpha \Gamma_j^\beta F = 
\begin{cases}
\Gamma_{j_i}^\beta \partial^\alpha_{i_j} F & i \neq j,j+1 \\
F & i = j,j+1 \text{ and } \alpha = \beta \\
\epsilon_j \partial_j^\alpha F & i = j,j+1 \text{ and } \alpha = - \beta
\end{cases}
\]
Using what we proved previously, $\partial_k^\alpha F$ is pseudo for all $k$, so by induction, $\partial_j^\alpha \epsilon_i F$ and $\partial_j^\beta \Gamma_i^\alpha F$ are always pseudo. Applying the criterion that we proved previously for a $p$-transfor to be pseudo, $\epsilon_i F$ and $\Gamma_i^\alpha F$ are pseudo.

Finally, we reason by induction on $p$ to show that for any two pseudo lax $p$-transfors $F$ and $G$, $F \star_i G$ is pseudo (if it is defined). Since any lax $0$-transfor is pseudo, it is true if $p=0$. Take now $p > 0$, and $A \in \CC_n$, for some $n > 0$. Then $F(A)$ and $G(A)$ are invertible, and so is $(F\star_i G)_n (A) = F_n(A) \star_i G_n(A)$ by Lemma \ref{lem:thin_and_composite_invertible}. Moreover, using the cubical $\omega$-category structure on $\Lax(\CC,\DD)$, we have:
\[
\partial_i^\alpha (F \star_j G) = 
\begin{cases}
\partial_i^\alpha F \star_{j_i} \partial_i^\alpha G & i \neq j \\
\partial_i^- F & i = j \text{ and } \alpha = - \\
\partial_i^+ G & i = j \text{ and } \alpha = +
\end{cases}
\]
So by the induction hypothesis, $\partial_j^\alpha (F \star_i G)$ is pseudo for all $j$, and therefore $F \star_i G$ is pseudo.
\end{proof}

\begin{defn}
Let $\CC$ and $\DD$ be cubical $\omega$-categories. We denote by $\plax(\CC,\DD)$ (resp. $\poplax(\CC,\DD)$) the pseudo lax transfors  (resp. the pseudo oplax transfors) from $\CC$ to $\DD$. By Proposition \ref{prop:carac_pseudo}, $\plax(\CC,\DD)$ and $\poplax(\CC,\DD)$ are cubical $\omega$-categories.
\end{defn}

\begin{prop}\label{prop:transfors_isomorphism}
For all cubical $\omega$-categories $\CC$ and $\DD$, the cubical $\omega$-categories $\plax(\CC,\DD)$ and $\poplax(\CC,\DD)$ are isomorphic.
\end{prop}
\begin{proof}
Let $F \in \plax(\CC,\DD)$, and define maps $G_n : \CC_n \to \DD_{n+p}$ as: $G_n(A) = \rho_{n,p} \cdot F_n(A)$. Let us show that $G$ is an oplax $p$-transfor (using formulas from Lemma \ref{lem:sigma_and_others}):
\begin{align*}
\partial_i^\alpha G_n(A) 
& = \partial_i^\alpha (\rho_{n,p} \cdot F_n(A)) 
= \partial_i \rho_{n,p} \cdot \partial_{i \cdot \rho_{n,p}}^\alpha F_n(A)
\\ & = \rho_{n-1,p} \cdot \partial_{i+p} F_n(A)
= \rho_{n-1,p} \cdot F_{n-1} (\partial_i^\alpha(A))
 = G_{n-1}(\partial_i^\alpha(A))
\end{align*}
\begin{align*}
G_{n}(\epsilon_i A) 
& = \rho_{n,p} \cdot F_{n}(\epsilon_i A) 
= \rho_{n,p} \cdot \epsilon_{p+i} F_{n-1}(A)
\\ & = \epsilon_{(p+i) \cdot \rho_{p,n}} (\partial_{(p+i) \cdot \rho_{p,n}} \rho_{n,p} \cdot F_{n-1}(A))
\\ & = \epsilon_i (\partial_{i} \rho_{n,p} \cdot F_{n-1}(A)) 
= \epsilon_i (\rho_{n-1,p} \cdot F_{n-1}(A)) 
 = \epsilon_i G_{n-1}(A)
\end{align*}
\begin{align*}
G_n(\Gamma_i^\alpha A)
& = \rho_{n,p} \cdot F_n (\Gamma_i^\alpha A) 
= \rho_{n,p} \cdot \Gamma_{p+i}^\alpha F_{n-1} (A) 
\\ & = \Gamma_{(p+i) \cdot \rho_{p,n}}^\alpha (\partial_{(p+i) \cdot \rho_{p,n}} \rho_{n,p} \cdot F_{n-1}(A))
\\ & = \Gamma_i^\alpha (\partial_i \rho_{n,p} \cdot F_{n-1} (A))
= \Gamma_i^\alpha (\rho_{n-1,p} \cdot F_{n-1} (A))
 = \Gamma_i^\alpha G_{n-1}(A)
\end{align*}
\begin{align*}
G_n(A \star_i B) 
& = \rho_{n,p} \cdot F_n(A \star_i B) 
 = \rho_{n,p} \cdot (F_n (A) \star_{p+i} F_n(B))
\\ & =  (\rho_{n,p} \cdot F_n(A)) \star_{(p+i) \cdot \rho_{p,n}} (\rho_{n,p} \cdot F_n(B))
= G_n(A) \star_i G_n(B) 
\end{align*}

We denote by $\mathbf P (F)$ this oplax $p$-transfor. For $A \in \CC_n$, $\rho \cdot F(A) = \rho_{n,p} \cdot F(A)$ is $\rho_{p,n}$-invertible (with $\rho_{p,n}$-inverse $A$), and so $\mathbf P(F)$ is actually pseudo. Let us show that $\mathbf P$ is functorial. Let $F \in \plax(\CC,\DD)_p$:
\begin{align*}
(\partial_i^\alpha (\mathbf{P}(F)))_n (A) 
& = \partial_{n+i}^\alpha ((\mathbf{P}(F))_n(A))
 = \partial_{n+i}^\alpha (\rho_{n,p} \cdot F(A))
 \\ & =
 \partial_{n+i} \rho_{n,p} \cdot \partial_{(n+i) \cdot \rho_{n,p}}^\alpha F(A) 
\\ & = \rho_{n,p-1} \cdot \partial_i^\alpha F(A)
 = \mathbf{P}(\partial_i^\alpha F) (A)
\end{align*}
\begin{align*}
(\mathbf{P}(\Gamma_i^\alpha F))_n(A) 
& = \rho_{n,p} \cdot ((\Gamma_i^\alpha F)_n (A))
= \rho_{n,p} \cdot \Gamma_i^\alpha (F_n(A))
\\ & = 
\Gamma_{i \cdot \rho_{p,n}}^\alpha (\partial_{i \cdot \rho_{p,n}} \rho_{n,p} \cdot F_n(A))
\\ & =
\Gamma_{n+i}^\alpha (\partial_{p+i} \rho_{n,p} \cdot F_n(A) )
= 
\Gamma_{n+i}^\alpha (\rho_{n,p-1} \cdot F_n(A) )
= 
(\Gamma_i^\alpha (\mathbf{P}(F)))_n(A)
\end{align*}
\begin{align*}
(\mathbf{P}(\epsilon_i F))_n(A) 
& = \rho_{n,p} \cdot ((\epsilon_i F)_n (A))
= \rho_{n,p} \cdot \epsilon_i (F_n(A))
\\ & = 
\epsilon_{i \cdot \rho_{p,n}} (\partial_{i \cdot \rho_{p,n}} \rho_{n,p} \cdot F_n(A))
\\ & =
\epsilon_{n+i} (\partial_{p+i} \rho_{n,p} \cdot F_n(A) )
= 
\epsilon_{n+i} (\rho_{n,p-1} \cdot F_n(A) )
= 
(\epsilon_i (\mathbf{P}(F)))_n(A)
\end{align*}
\begin{align*}
(\mathbf{P}(F \star_i G))_n(A) 
& =
\rho_{n,p} \cdot ((F \star_i G)_n(A)) 
= \rho_{n,p} \cdot (F_n(A) \star_i G_n(A))
\\ & = (\rho_{n,p} \cdot F_n(A)) \star_{i \cdot \rho_{p,n}} (\rho_{n,p} \cdot G_n(A))
\\ & 
= \mathbf{P}(F)_n(A) \star_i \mathbf{P}(G)_n(A)
= (\mathbf{P}(F) \star_{i} \mathbf{P}(G))_n(A)
\end{align*}

So $\mathbf{P}$ is a functor from $\plax(\CC,\DD)$ to $\poplax(\CC,\DD)$. Reciprocally, if $F$ is a pseudo oplax $p$-transfor, we define a family of maps $\mathbf R(F)_n : \CC_n \to \DD_{n+p}$ by setting $\mathbf R(F)_n(A) = \rho_{p,n} \cdot F_n(A)$. As we did for $\mathbf{P}$, we show that $\mathbf R$ induces a functor from $\poplax(\CC,\DD)$ to $\plax(\CC,\DD)$. Finally, since $\rho_{p,n} \cdot \rho_{n,p} = 1$, $\mathbf{P}$ and $\mathbf R$ are inverses of each other, and $\plax(\CC,\DD)$ is isomorphic to $\poplax(\CC,\DD)$.
\end{proof}

\bibliography{omega_cub_cat}
\bibliographystyle{plain}

\end{document}